\documentclass[11pt,oneside,reqno,english]{amsart}
\usepackage{amsmath,amsthm,amsfonts,amssymb,mathrsfs,bm,graphicx,stmaryrd}
\usepackage[T1]{fontenc}
\usepackage[latin9]{inputenc}
\usepackage[margin=1in]{geometry}
\usepackage{xcolor}
\usepackage{mathrsfs}
\usepackage{amstext}
\usepackage{amsthm}
\usepackage{amssymb}
\usepackage{stmaryrd}
\usepackage{setspace}
\usepackage[colorlinks=true,linkcolor=blue]{hyperref}
\usepackage{url}

\makeatletter
\numberwithin{equation}{section}
\numberwithin{figure}{section}
\theoremstyle{plain}
\newtheorem{thm}{\protect\theoremname}[section]
\theoremstyle{definition}
\newtheorem{defn}[thm]{\protect\definitionname}
\theoremstyle{plain}
\newtheorem{prop}[thm]{\protect\propositionname}
\theoremstyle{plain}
\newtheorem{cor}[thm]{\protect\corollaryname}
\theoremstyle{plain}
\newtheorem{lem}[thm]{\protect\lemmaname}
\theoremstyle{remark}
\newtheorem{notation}[thm]{\protect\notationname}
\theoremstyle{remark}
\newtheorem{rem}[thm]{\protect\remarkname}
\theoremstyle{remark}
\newtheorem*{claim*}{\protect\claimname}

\usepackage[colorlinks=true,linkcolor=blue]{hyperref}
\usepackage{url}
\newcommand{\mix}{{\rm mix}}

\definecolor{purple}{rgb}{0.9,0,0.8}

\definecolor{gray}{rgb}{0.7,0.7,0.7}

\DeclareFontFamily{OMX}{MnSymbolE}{}
\DeclareFontShape{OMX}{MnSymbolE}{m}{n}{
    <-6>  MnSymbolE5
   <6-7>  MnSymbolE6
   <7-8>  MnSymbolE7
   <8-9>  MnSymbolE8
   <9-10> MnSymbolE9
  <10-12> MnSymbolE10
  <12->   MnSymbolE12}{}
\DeclareSymbolFont{mnlargesymbols}{OMX}{MnSymbolE}{m}{n}
\SetSymbolFont{mnlargesymbols}{bold}{OMX}{MnSymbolE}{b}{n}
\DeclareMathDelimiter{\llangle}{\mathopen}{mnlargesymbols}{'164}{mnlargesymbols}{'164}
\DeclareMathDelimiter{\rrangle}{\mathclose}{mnlargesymbols}{'171}{mnlargesymbols}{'171}

\newcommand{\e}{\varepsilon}
\newcommand{\bP}{\mathbb{P}}
\newcommand{\Z}{\mathbb{Z}}
\newcommand{\upd}{\textup{Upd}}
\newcommand{\ful}{\mathscr{F}} 
\newcommand{\emp}{\mathscr{E}} 
\newcommand{\eful}{\overline{\mathscr{F}}} 
\newcommand{\eemp}{\overline{\mathscr{E}}} 
\newcommand{\per}{\textup{Perc}_n}
\newcommand{\perr}{\textup{Perc}_r}
\newcommand{\tini}{t_{\textrm{init}}} 
\newcommand{\pini}{p_{\textrm{init}}}
\newcommand{\mui}{\mu_{p,\,q}^n} 
\newcommand{\nup}{\mathscr{N}} 
\newcommand{\env}{\Xi} 
\newcommand{\conn}{\textup{Conn}} 
\newcommand{\smax}{t_{\textrm{max}}}
\newcommand{\spar}{\textrm{Spa}_n}

\newcommand{\sh}{\mathscr{H}}
\newcommand{\us}{\textrm{U}_s}
\newcommand{\pstar}{\overline{p}}
\newcommand{\hatl}{\hat{\lambda}}
\makeatother
\usepackage{babel}
\providecommand{\claimname}{Claim}
\providecommand{\corollaryname}{Corollary}
\providecommand{\definitionname}{Definition}
\providecommand{\lemmaname}{Lemma}
\providecommand{\notationname}{Notation}
\providecommand{\propositionname}{Proposition}
\providecommand{\remarkname}{Remark}
\providecommand{\theoremname}{Theorem}
\title[Cutoff for RCM]{Information Percolation and Cutoff for the Random-Cluster model}

\author[Ganguly]{Shirshendu Ganguly}
\address{S. Ganguly \hfill\break
	Department of Statistics\\ UC Berkeley \\
	Berkeley, California, CA 94720, USA.}
\email{sganguly@berkeley.edu}

\author[Seo]{Insuk Seo}
\address{I. Seo \hfill\break
Department of Mathematical Science and RIMS\\ Seoul National University\\
Seoul, South Korea. }
\email{insuk.seo@snu.ac.kr}
\begin{document}

\maketitle

\begin{abstract}
We consider the Random-Cluster model on $(\Z/n\Z)^d$  with parameters $p \in (0,1)$ and $q\ge 1$. This is a generalization of the standard bond percolation (with edges open independently with probability $p$) which is biased by a factor $q$ raised to the number of connected components.  We study the well known FK-dynamics on this model where the update at an edge depends on the global geometry of the system unlike the Glauber Heat-Bath dynamics for spin systems, and prove that for all small enough $p$ (depending on the dimension)
and any $q>1$, the
FK-dynamics  exhibits the cutoff phenomenon
at $\lambda_{\infty}^{-1}\log n$ with a window size  $O(\log\log n)$,
where $\lambda_{\infty}$ is the large $n$ limit of the spectral gap of the process. Our proof extends the \textit{Information Percolation} framework of Lubetzky and Sly \cite{LS4} to the Random-Cluster model and also relies on the arguments of Blanca and Sinclair \cite{BS} who proved a sharp $O(\log n)$ mixing time bound for the planar version. A key aspect of our proof is the analysis of the effect of a sequence of dependent (across time) Bernoulli percolations extracted from the graphfical construction of the dynamics, on how information propagates.
\end{abstract}

\maketitle
\tableofcontents
\section{Introduction and main result }
The random-cluster (Fortuin-Kasteleyn/FK) model is an extensively studied model in statistical physics,
generalizing electrical networks, percolation, and spin systems like the Ising and Potts models, under a single framework.
In this work, we study the so called heat-bath Glauber dynamics or FK-dynamics for the model on the $d$-dimensional torus. The main result of this paper establishes a sharp convergence to equilibrium for this Markov chain also known as the {\it{cutoff phenomenon}}.

\subsection{\label{sec11}Random-cluster model (RCM)}

For $d\ge2$, denote by $\Lambda_{n}=\mathbb{Z}_{n}^{d},$ the $d$-dimensional
discrete torus and by $E_{n}=E(\Lambda_{n}),$ the set of edges
in $\Lambda_{n}$.
{We will fix the dimension to be $d$ throughout the entire paper}. The random-cluster measure $\mui$
on the graph $(\Lambda_{n},\,E_{n})$ with parameters $p\in(0,\,1)$ and $q>0$
is a probability measure on the space of subsets of $E_{n}$ defined
by
\[
\mui(S)=\frac{1}{Z_{p,\,q}^{n}}p^{|S|}(1-p)^{|E_{n}\setminus S|}q^{c(S)}\;\;;\;S\subset E_{n}\;,
\]
where $Z_{p,\,q}^{n}$ is the partition function turning $\mui$ into
a probability measure, and $c(S)$
is the number of connected components of the graph $(\Lambda_{n},\,S)$.
Clearly the measure $\mui$ can be regarded as a probability measure on $\Omega_{n}=\{0,\,1\}^{E_{n}}$, i.e., we will identify $X=(X(e))_{e\in E_{N}}\in\Omega_{n}$
with a subset $A$ of $E_{n}$ where $e\in A$ if and only
if $X(e)=1$. Hence, by slight abuse of notation, we can always
regard $X\in\Omega_{n}$ as a subset of $E_{n}$.
The random-cluster model was introduced by Fortuin and Kasteleyn (see \cite{F1972,FK1972}) and unifies the study of various objects in statistical mechanics such as random graphs, spin systems and electrical networks (see \cite{gri}). When $q = 1,$ this model corresponds to the standard bond percolation but when $q>1$ (resp., $q < 1$), the probability measure biases
subgraphs with more (resp., fewer) connected components.
For the special case of integer $q\ge 2$ the random-cluster model is a dual to
the classical ferromagnetic $q$-state Potts model, via the so called Edward-Sokal coupling of the models (see, e.g., \cite{ES}).
 However, note that unlike spin systems, the probability that an edge $e$ belongs to $A$ does not depend only
on the dispositions of its neighboring edges but on the entire configuration $A$, since connectivity is
a global property (see Figure \ref{fig1} for an illustration).

\subsection{\label{sec12}FK-dynamics (Glauber/Heat-bath dynamics)}

The FK-dynamics is a reversible Markov process $X_t=\{X_t(e)\}_{e\in E_n}$ on $\Omega_{n},$ whose invariant
measure is given by  $\mui$.
Informally, at rate one, the state of every edge $X(e)$ is resampled conditionally on the state of the remaining edges i.e.,

\[
X(e)=\begin{cases}
1 & \text{w.p.}\,\, p \text{ if }e \text{ is not a cut-edge,} \\
1 & \text{w.p.}\,\, \frac{p}{p+(1-p)q}\,\, \text{if }e\,\,\text{is a cut-edge,} \\
0 & \text{otherwise},
\end{cases}
\]
where we use the standard terminology {\it{cut-edge}} to denote an edge whose removal increases the number of connected components by one. A more formal treatment appears in Definition \ref{def11}.

{Note that unlike Glauber dynamics on spin systems like Ising or Potts models, the FK-dynamics has long range dependencies (see Figure \ref{fig1}).}
The FK-dynamics has been an object of significant interest and has played a key role in several recent works. A by no means complete, but nonetheless representative list includes \cite{}.

The key statistic  we will consider is the time taken by the above dynamics to converge to equilibrium.

\begin{figure}[h]
\centering
\includegraphics[scale=.20]{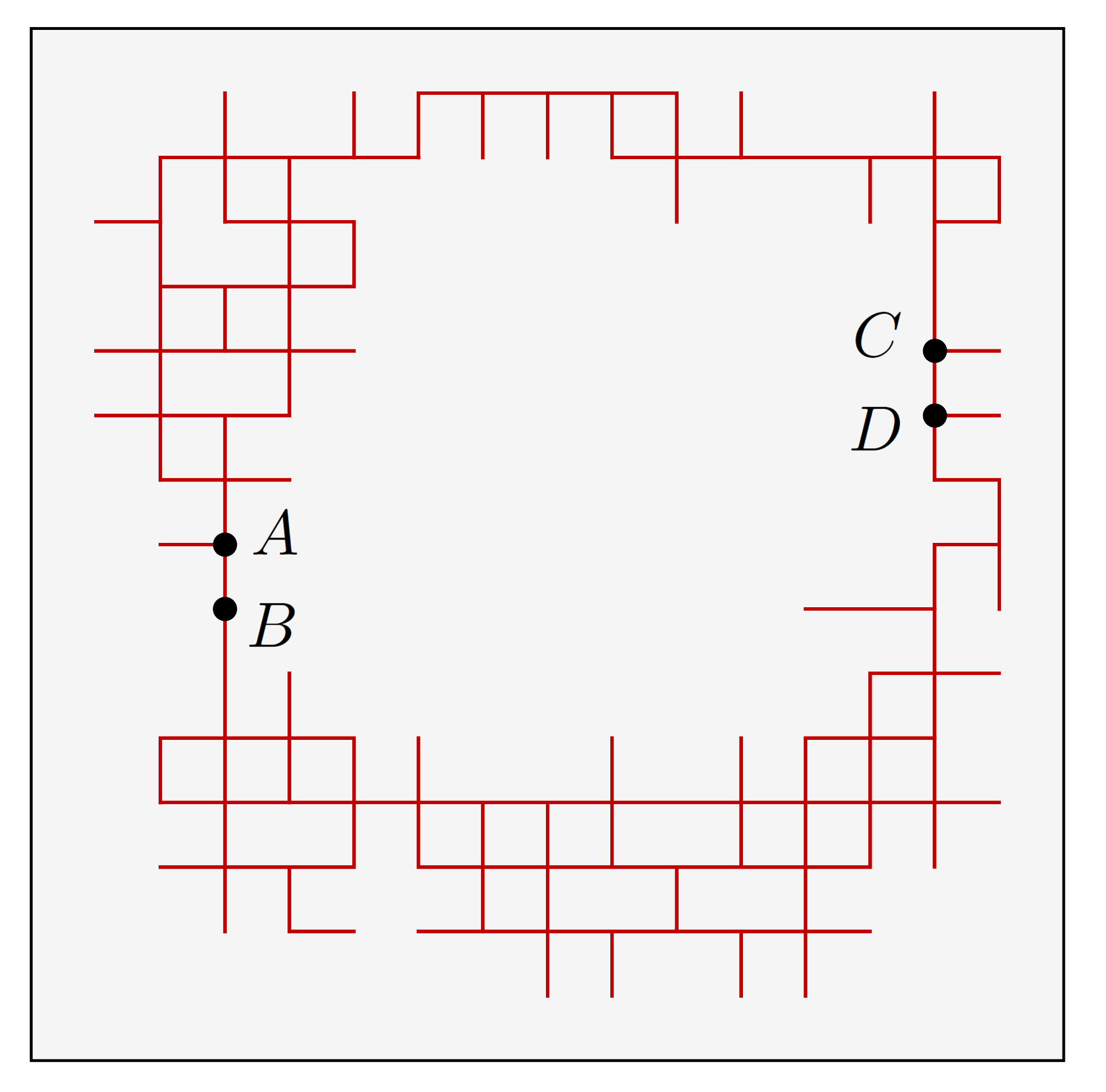}
\caption{Illustrating long range dependencies in FK-dynamics.
Consider a configuration where the red edges are open while everything else is closed.
The probability of the edge $AB$ to be open then depends on whether the edge $CD$ is open or not.}
\label{fig1}
\end{figure}
\subsection{\label{sec13}Mixing of Markov chains and cutoff phenomenon} We review in brief the  set up of interest for us from the theory of  reversible Markov chains with finite state spaces. For an extensive account of all the details and recent progress in various directions see \cite{LPW}.   For two probability
measures $\mu_1$ and $\mu_2$ on $S$ we will be interested in the $L^1$-distance or the so-called \textit{total variation distance} between them denoted by $\|\mu_1-\mu_2\|_{\textrm{TV}}$\footnote{$\|\mu_1-\mu_2\|_2$ will be used to denote the $L^2$-distance where the $1$-norm in \eqref{l1norm} is replaced by the $2$-norm, {{ i.e.,
$$
\Vert \mu_1 -\mu_2 \Vert_2 = \sum_{x\in S} \left| \frac{\mu_1(x)}{\mu_2(x)} - 1 \right|^2 \mu_2 (x)\;.
$$
}}
}:
\begin{equation}
\sup_{A \subset S} ( \mu_1(A)-\mu_2(A))=\frac{1}{2}\sum_{x\in S}|\mu_1(x)-\mu_2(x)|
=\frac{1}{2}\sum_{x\in S}\left|\frac{\mu_1(x)}{\mu_2(x)}-1\right|\mu_2(x)\;.
\label{l1norm}
\end{equation}
For concreteness consider a continuous time reversible Markov chain $Y_t$ with a finite state space $S$ and equilibrium measure $\pi$.

Denote by $\mathbb{P}_x$, $x\in S$, the law of Markov chain starting from $x$.

We will be primarily interested in the  total variation mixing time  defined by
$$t_{\mix}(\e)=\inf\big\{t: \sup_{y\in S}\|\bP_{y}[Y_t \in \cdot ]-\pi\|_\textrm{TV}\le \e \big\}\;\;;\;\e\in(0,\,1)\;.$$
For notational brevity, we will {denote by $d(t)$, the worst case total variation distance to stationarity for the FK-dynamics, i.e.,
\begin{equation}\label{dt}
d(t)=d_n (t):=\sup_{x\in \Omega_n}\|\bP_x[X_t \in \cdot\,]-\mui\|_\textrm{TV}
\end{equation}
from now on.
Many naturally occurring Markov chains are expected to exhibit a sharp transition in convergence, in the sense that the total variation distance to equilibrium drops from one to zero in a rather short time window.  This is formalized by the notion of {{cutoff}} formulated by Aldous and Diaconis \cite{AD86} (see also \cite{dia1}).
Formally a sequence of Markov chain $Y^{(1)}_t,Y^{(2)}_t, \ldots$ with mixing times given by $t^{(1)}_{\mix}(\e), \,t^{(2)}_{\mix}(\e),\ldots$ is said to exhibit the \textit{\textbf{Cutoff Phenomenon}} if for any $\e\le 1/2,$ $$\lim_{i\to \infty}\frac{t^{(i)}_{\mix}(\e)}{t^{(i)}_{\mix}(1-\e)}=1\;.$$
 Moreover cutoff is said to occur with window size $w_{i}$ if for any $\e\le 1/2$ one has  $${t^{(i)}_{\mix}(\e)}-{t^{(i)}_{\mix}(1-\e)}=O_{\e}(w_{i})\;,$$
 where $w_i= o(t^{(i)}_{\mix}(\frac{1}{4})).$
}

\subsection{\label{sec14}Main result}
Given the above definitions, our main result establishes cutoff for the FK-dynamics for a range of sub-critical values of the parameters $p, q.$

\begin{thm}
\label{tmain}For any $d\ge 2$, there exists $p_{0}={p_{0}(d)}>0$ such that, for all
$p\in(0,\,p_{0})$ and {$q>1$}, there exists a constant $\lambda_{\infty}=\lambda_{\infty}(p,q)$ such that the FK-dynamics on $\Omega_{n}$ exhibits
cutoff at $\frac{d}{2\lambda_{\infty}}\log n$ with
order $O(\log\log n)$ window size.
\end{thm}

Some remarks are in order. Note that the case $q=1$ is the well known example of random walk on a hypercube where cutoff occurs for all values of $p$ {{see \cite[Theorem 18.3]{LPW}}}. Similarly in the case $d=1,$ one notices that each edge is a cut edge unless the configuration is completely full.  Thus the process in this case can also be coupled with a random walk on a hypercube, implying cutoff for all values of $p$ and $q.$

The value of the threshold $p_{0}$ in the statement above, only depends on the dimension through the value of the critical bond percolation probability and does not depend on
$q$. We shall assume that $q> 1$ is fixed from now on. {Notice that by a duality argument as in  \cite[Section 7]{BS}, in the planar case (i.e., $d=2$)} it follows that Theorem \ref{tmain} holds also when  $p$ is close enough to $1$.
 We will also elaborate on a description of $\lambda_{\infty}$ in terms of the spectral gap of the Glauber dynamics for the infinite volume RCM in Section \ref{infinitevolume}.

\subsection{Background and related work}
There has been much activity over the past two decades in analyzing Glauber dynamics for spin systems in both statistical physics and
computer science leading to deep connections  between the
mixing time and the phase structure of the physical model. In contrast, the Glauber dynamics for
the RCM remains less understood. The main reason for this is that connectivity is a global property. Ullrich in a series of important  papers \cite{ullrich1,ullrich2,ullrich3} established comparison estimates between the FK-dynamics and the well known non-local Swendsen-Wang (SW) dynamics (\cite{SW1}) using functional analytic arguments. Although initially the arguments appeared only for integer
values of $q$ exploring
connections with the Ising/Potts models, the analysis extends to all $q>1$, which appeared in \cite{blanca1}.
Until recently, all existing bounds on the  FK-dynamics were via transferring results for the SW or related dynamics \cite{SW1} using comparison estimates as above.
 However these methods typically yield highly sub-optimal bounds  and does not provide any insight into the behavior of RCM. Recently the authors of  \cite{BS} established a fast mixing  time{ of order $O(n^2\log n)$}  bound for the {discrete time} FK-dynamics on RCM in a box of size $n$ in $\Z^2$ with a special class of boundary conditions.  The proof works for all $q\ge 1$ and $p\neq p_c(q).$ Furthermore, although not explicitly mentioned, the arguments extend to periodic boundary conditions as well.
The key ingredients used were {planar duality}, tools developed for mixing of spin systems in \cite{MS13} and most importantly the exponential decay of connectivity below $p_c(q)$ established in the breakthrough work \cite{BD}. More recently \cite{BGV} extends the results to a more general class of boundary conditions with weaker bounds. Among various things, the latter work in particular also shows that boundary conditions can have a drastic effect on the mixing time.

A general conjecture of Peres \cite{peres2004}  indicates that  one should expect cutoff to occur in the regime of fast mixing for  many natural chains as above.  In  the breakthrough papers, \cite{LS1,LS2}, Lubetzky and Sly verified the above conjecture for Glauber dynamics for Ising and Potts models, putting forward a host of new methods using ideas similar to  the  Propp-Wilson {\it{coupling from the past}} \cite{proppwilson} as well as relating $L^1$-mixing to $L^2$-mixing using powerful log-Sobolev inequalities \cite{diasaloff}.
Subsequently in \cite{LS4, LS3}, the results of the above papers were refined by inventing the general \textit{\textbf{Information percolation}} machinery.
Furthermore in very recent work, \cite{NS} extended the above framework to prove cutoff results for the non-local SW dynamics for Potts models on the torus in any dimension for suitably high temperatures.

{However as indicated above, the FK-dynamics has  significant differences with the above described spin models and whether cutoff occurs in the fast mixing regime in this case was left open. The main theorem of this paper answers this question in the affirmative as long as $p$ is small enough and $q>1$.  In the process, we extend the Information Percolation framework to the RCM setting as well. An elaborate description of the various geometric difficulties and how to encounter them is presented in the next section. We end this section by also mentioning the recent work of Lubetzky and Gheissari on proving quasi-polynomial bounds for the mixing time at criticality for FK-dynamics in two dimensions and related bounds for critical spin systems in \cite{gheissari2016,gheissari2018,gheissarimixing} based on recent breakthroughs in \cite{DC1,DC2} .}

\section{Idea of the proof and organization of the article}\label{iop}
We first develop a graphical construction (grand coupling of FK-dynamics) which will be quite useful in constructing coupling arguments.
We then discuss the key issues that one faces towards proving the main result and what new ideas one needs beyond the existing literature to address them.
 \subsection{Graphical construction/Monotone coupling}
 We will define the FK-dynamics formally through the following graphical construction by creating what is now popularly called in the literature as the \textit{\textbf{Update sequence}} (see \cite{LS1,NS}). For $e\in E_{n}$,
define the sequence of updates as
\begin{equation}
\upd(e)=\left\{ (t_{1},\,U_{1}),\,(t_{2},\,U_{2}),\,\cdots\right\} \;,\label{upseq}
\end{equation}
where $t_{1}<t_{2}<\cdots$ is a sequence of update times obtained
from an independent Poisson process with rate $1$ attached at $e$,
and for each $i$, $U_{i}$ is a uniform random variable in $[0,\,1]$
independent of all other randomness. The sequence $\upd(e)$ is
the update sequence corresponding to $e$. Then, we define the full
update sequence as
\begin{equation}
\upd=\bigcup_{e\in E_{n}}\upd(e)\;.\label{upseq2}
\end{equation}
Note that $t_{i}(e)\neq t_{j}(e')$ for all $i,\,j\in\mathbb{N}$
and $e,\,e'\in E_{n}$ almost surely. It would also be useful to define for $0<t_{1}<t_{2}$,
the update sequence of $e$ in the time interval $(t_{1},\,t_{2}]$
as
\begin{equation}
\upd[t_{1},\,t_{2}](e)=\{(s,\,U):(s,\,U)\in\upd(e),\,s\in(t_{1},\,t_{2}]\}\;,\label{updt}
\end{equation}
and
\[
\upd[t_{1},\,t_{2}]=\bigcup_{e\in E_{n}}\upd[t_{1},\,t_{2}](e)\;.
\]
For $X\in\Omega_{n}$, we say that $e\in E_{n}$ is a \textit{cut-edge
}if $c(X\setminus\{e\})\neq c(X\cup\{e\})$, (recall that $c(\cdot)$ denotes the number of connected components).
Furthermore, from now on, we shall assume $q>1$ and write
\[
p^{*}=\frac{p}{q(1-p)+p}<p
\]
for convenience.

We now introduce a construction of the FK-dynamics suitable for our purposes. This is the standard grand coupling for the FK-dynamics (see \cite{gri1})
\begin{defn}[FK-dynamics/Monotone Coupling]
\label{def11} For each $(t,\,U)\in\upd(e)$ for some $e\in E_{n}$,

\begin{enumerate}
\item \begin{enumerate}
\item If $U<1-p+p^{*}$, we let
\[
X_{t}(e)=\begin{cases}
0 & \text{if }U\in[0,\,1-p),\\
1 & \text{if }U\in[1-p,\,1-p+p^{*}).
\end{cases}
\]
\item If $U\ge1-p+p^{*}$, we let $X_{t}(e)=0$ if $e$ is a cut-edge in
$(\Lambda_{n},\,X_{t-})$, and $X_{t}(e)=1$ if $e$ is not a cut-edge
in $(\Lambda_{n},\,X_{t-})$.
\end{enumerate}
\item We set $X_{t}(e')=X_{t-}(e')$ for all $e'\neq e$.
\end{enumerate}
\end{defn}
 We will denote by $\mathbb{P}_{x_{0}}=\mathbb{P}_{x_{0}}^{p,\,q,\,n}$
the law of the FK-dynamics starting from $x_{0}\in\Omega_{n}$.
{{Similarly, for a probability measure $\nu$ on $\Omega_n$, denote by $\mathbb{P}_{\nu}$ the law of FK-dynamics starting from the initial distribution $\nu$.
}}
 Note that the FK-dynamics is reversible with respect to its invariant measure $\mui$. Naturally the update sequence allows a grand coupling of $(X_{t})$ started from all possible configurations $x_0.$ A well known fact is the monotonicity of FK-dynamics i.e., if $(X_{t})$ and $(Y_{t})$ are two copies of the Markov chain started
from $x_{0}$ and $y_{0}$ with $x_{0}\le y_0$ in the usual partial order on $\Omega_n$, then under the grand coupling for all later times $t$ one has $X_t\le Y_t.$ {Thus often this coupling is called the \textit{monotone
coupling }and the corresponding law is denoted by $\mathbb{P}_{x_{0},\,y_{0}}$.}
Note that another perhaps more canonical way to define the dynamics would be to first check if $e$ is a cut-edge (resp. not) and then accordingly set it to $0$ or $1$ depending on whether $U<1-p^*$ or not (resp. $U<1-p$ or not). However
the  above alternative formulation has the nice property that if $U<1-p+p^{*}$,
we do not need to check whether $e$ is a cut-edge or not, and the randomness
at $e$ only depends on $U$, not the entire configuration of $X_{t}$.
This will be used throughout the paper in various coupling arguments.

\subsection{The key ideas of the proof}
In the work of Lubetzky and Sly \cite{LS1} on the Ising model, the key idea was to break the dependencies in the Markov chain to reduce the analysis to the study of a product chain of Glauber dynamics on small boxes. The proof then relied on the relation between the $L^1$-mixing time of the product chain to $L^2$-mixing time of the individual coordinates and sharp estimates on the latter obtained via Log-Sobolev inequalities (LSI). Unfortunately such functional analytic tools are not available for the RCM. {Although it is perhaps natural to predict that such estimates hold at least in some part of the parameter space, it is important to point out that the standard arguments which work for nearest neighbor spin systems fail owing to long range effects. Whether the LSI indeed holds for the RCM thus remains an important open problem.}

Furthermore, to improve the size of the cutoff window to $O(1),$ in \cite{LS4,LS3}, the powerful machinery of information percolation was invented to bypass the use of log-Sobolev inequalities to estimate the $L^2$-mixing time. The proof however still relied heavily on the local nature of Glauber dynamics for spin systems. {On the other hand a non-local Markov chain admitting global changes is the well known Swendsen-Wang (SW) dynamics for Potts model.  In SW dynamics for the Potts model,  one proceeds by sampling an independent bond percolation on each of the mono-chromatic  components (connected component of vertices with the same spins) and then for each connected component of the percolation sampled, a uniformly random spin is assigned. This is done at every time step independently of the past and hence the interaction of the spin at every vertex at every time step in only limited to spins within its percolation cluster.}

 Very recently in \cite{NS} the strategy was extended to SW dynamics. The latter work is  based on the observation made above that while in Glauber dynamics, in one step the spin at a vertex can only depend on its immediate neighbors,  the state of a vertex  in SW by definition depends on all the vertices inside an independent percolation cluster sampled at each time step. Thus in the subcritical regime, since the cluster diameters have exponential tails, one can expect the same approach to go through and indeed this is what is made rigorous in \cite{NS}.
The arguments in this article draw inspiration mostly from this last article.

As indicated before, at a very high level, one of the main contributions of our approach is extending the  Information Percolation framework to the setting of FK-dynamics.
However in RCM, in one step the update of an edge can depend on the status of an arbitrarily far located edge (see Figure \ref{fig1}).
To bypass this, we first run the process for an $O(1)$ burn-in time which allows the process to be dominated by a subcritical Bernoulli percolation.

At this point we try to analyze the information percolation clusters. Very informally (see Section \ref{sec4} for precise definitions) this approach involves keeping track of the interactions between various edges as they are updated, {\it{backwards}} in time. For e.g.,: if an edge $e$ is updated using an element $(t,U)\in \upd(e)$ one of two things could happen (recall Definition \ref{def11}):
\begin{itemize}
\item $U<1-p+p^*$, in which case the updated value of the edge is a Bernoulli variable independent of the state of the system. In this case we call the edge to become \textit{\textbf{Oblivious}}.
\item However if $U>1-p+p^*$ one needs to check whether $e$ is a cut-edge or not and in the process interacts (shares information) with several edges.
\end{itemize}

Formally one considers a space-time slab (see Figure \ref{ip}) and evolves backward in time by branching out to all possible edges an update shares information with, or gets killed in case of an oblivious update.
The key usefulness of this approach as exploited in \cite{LS1,LS2,LS4,LS3, NS} is that if the backward branching process (called the \textit{\textbf{History diagram}}) is subcritical then, the process will be killed before reaching the initial configuration in this backward evolution causing the final configuration to be independent of the initial one implying coupling of all starting states under the grand coupling. However this is an overkill since for cutoff to occur one can tolerate some mild dependence on the initial condition as long as that is hidden inside the natural fluctuation of the system.

To bound the growth rate of the history diagram
we first discretize time with interval length $\Delta=\frac{1}{\sqrt{p}}$ (as the reader will notice, this choice of $\Delta$ is not  special and a host of other choices will work too) and consider the interval $[\tau_{i},\tau_{i+1}]$ where $\tau_i=i\Delta$ and define the history diagram only at times $\tau_i.$
We first extract several auxiliary percolation models based on the update sequence (see Table \ref{chart}), and one of which denoted by $\env_i(\cdot)$ captures the following: For every $i,$ $\env_i(e)$ is $1$ iff $e$ has not been updated in the interval $[\tau_{i},\tau_{i+1}]$ or $e$ is open at least once in  $[\tau_{i},\tau_{i+1}]$ for the Glauber dynamics for the standard Bernoulli percolation with parameter $p,$ (random walk on the hypercube) using the same update sequence and starting from the empty configuration. Now given the history diagram up to time $\tau_{i+1}$ for any edge we first check if it has been updated or not in an interval $[\tau_{i},\tau_{i+1}]$ (recall the history diagram flows backwards). If not, the edge continues to be a part of the history diagram, if it is updated using an oblivious update it gets killed, otherwise we bound the spreading of information by the connected component of $e$ in the percolation $\env_i \cup \env_{i-1}.$

{
Note that to ensure that the state of the edge $e$ throughout the interval $[\tau_{i},\tau_{i+1}]$ does not depend on any edge not included in the history diagram  we need the boundary of the latter to be closed throughout the entire interval i.e., we must consider its connected component `forward in time' which a priori depends on the entire time interval $[0,\tau_i].$  However this is the point at which we use the smallness of $p$ crucially, which creates an environment which is subcritical and hence the connected component can be bounded by the connected component of $\env_i \cup \env_{i-1},$ i.e., instead of the entire interval $[0,\tau_{i+1}]$ we can get by, just using the information on $[\tau_{i-1},\tau_{i+1}].$
}

Given the above, the situation is similar to the definition of the SW dynamics considered in \cite{NS}, except that the percolation sampled at every discrete time step is now $1$-dependent across time.
This creates the need for a refined and delicate analysis of the information percolation clusters to yield $L^2$-mixing bounds.
This is stated as Theorem \ref{t41} and Proposition \ref{p46}. The proof of the latter is the core of this work. The above approach adopted in the paper of extracting dependent percolation models that can be analyzed could be of independent interest and useful in other general  contexts in bounding how passage of information occurs in such dynamical settings.

Assuming these results, the arguments used to show cutoff are quite similar to the ones already appearing in \cite{NS} based on the methods in \cite{LS1}.
An additional ingredient needed to prove Theorem \ref{t41} from Proposition \ref{p46} is that the spectral gap of the FK-dynamics is positive uniformly in the system size.  In SW the lower bound on the spectral gap follows by path coupling by establishing a one step contraction which unfortunately is absent in our setting; instead we rely on the a priori mixing time bounds obtained in \cite{BS}.  c.

{ Finally, we mention that for the Ising model, \cite{LS4} exploited monotonicity of the system, to prove an $O(1)$ bound on the cutoff window without resorting to the methods of \cite{LS1}. Such sharp bounds are missing in \cite{NS} which deals with the general Potts model. However the RCM is monotone and whether this can be used to prove a similar improvement of Theorem \ref{tmain} is not pursued in this paper and is left for further research.  Furthermore, another possible direction to investigate is the effect of boundary conditions. While the current paper only deals with periodic boundary conditions, for local dynamics on Ising and Potts models \cite{LS2} proved sharp mixing time results for general boundary conditions.  Recall that typically in addressing such questions, there are two goals. One is to control the cutoff window size and the other is to pin down the location. Under certain special cases, in \cite{LS2}, the location of mixing was related to infinite volume objects.  Moreover, to bound the window size, \cite{LS2} relied on certain worst case Log-Sobolev constants. Since these are not available in our setting and  boundary conditions can lead to delicate global dependencies, the current arguments in the paper do not directly go through. Nonetheless, this is an important project to be taken up in the future.
}

\subsection{Organization of the article} We prove and collect results about a priori bounds on the mixing time and the spectral gap in Section \ref{specgap} to be used throughout the rest of the article.  As mentioned above, we need to define several auxiliary percolation models based on the update sequence. This is done in Section \ref{sec3}. \textbf{Section \ref{sec4} is the core of this work and the main contribution in this paper which bounds the $L^2$-mixing time by defining suitable information percolation clusters}.  This section is rather long and has several new constructions and delicate geometric arguments. However assuming the main result of this section, the proof of Theorem \ref{tmain} is quite similar to the arguments appearing in \cite{LS1, LS4, NS}.  The reader not familiar with the latter papers can choose to first assume the results of Section \ref{sec4} to see how they are used in the subsequent sections to then come back to the proofs of Section \ref{sec4}.

The proof of the main result Theorem \ref{tmain} spans  Section \ref{reduction1} where certain modifications of arguments of \cite{LS1} and  Section \ref{proofmain} where the final proof appears. The outstanding proofs of some of the stated claims are collected in the Appendix (Section \ref{appendix1}).

\subsection*{Acknowledgements}
The authors thank Antonio Blanca, Fabio Martinelli and Alistair Sinclair for several useful discussions.  They also thank the anonymous referees for the various useful comments and suggestions that helped improve the paper. IS was supported by the National Research Foundation of Korea (NRF) grant funded by the Korea government (MSIT) (No. 2018R1C1B6006896 and No. 2017R1A5A1015626) and Research Resettlement Fund for the new faculty of Seoul National University.

\section{A priori bounds on mixing time and spectral gap}\label{specgap}

We start by recalling the following standard result.
 \begin{prop}
\label{propgap}\cite[Theorems 12.3 and 12.4]{LPW} Let $(Z_{t})$ be a discrete time ergodic reversible
Markov chain on a finite state space $S$ with the equilibrium measure $\pi$, let $ \mathbb{Q}_z$ be the law of Markov chain $(Z_t)$ starting from $z\in S$,
and let $\gamma$ be the spectral gap of the Markov chain $(Z_t)$. Then,
\[
(1-\gamma)^{t}\,\le \,2\, \sup_{z\in S} \Vert \mathbb{Q}_z(Z_{t}\in \cdot) -\pi \Vert_{\textrm{TV}}
\,\le\,\frac{1}{\pi_{\min}}\,(1-\gamma)^{t}\;,
\]
where $\pi_{\textrm{min}}=\min_{x\in S}\pi(x)$.
\end{prop}

In \cite{BS}, a discrete version of FK-dynamics is considered where at every discrete time step, an uniformly chosen edge is updated. Denote
by $(\widehat{X}_{k})_{k\ge0}$ the discrete FK-dynamics in $\Omega_{n}$, and by $\mathbb{\widehat{\mathbb{P}}}_{x_{0},\,y_{0}},$
the law of the monotone coupling (Definition \ref{def11}) of two copies of discrete FK-dynamics
$\widehat{X}_{k}$ and $\widehat{Y}_{k}$ starting from two initial conditions $x_{0}, y_0 \in\Omega_{n}$ respectively.
Moreover, let $\hatl(n)=\hatl(n,\,p,\,q)$ denote the
spectral gap of the above process.
{Furthermore let $\hat{t}_{\rm mix}=\hat{t}_{\rm mix}(1/4)$ and $\hat{d}(t)$ be the mixing time   and  the worst-case distance to stationarity respectively in the sense of \eqref{dt} for the discrete time dynamics.}
Then, the following sharp mixing time results were either obtained or are consequences of the results
 in \cite{BS}. In the latter, only the two dimensional case was treated but one can easily verify that the
 arguments extend to general dimensions under exponential decay of connectivity.
We provide brief sketches of the proofs of these results with pinpoint references to the relevant literature for the remaining details.

\begin{thm}
\label{t21}For any dimension $d$, there exists $p_0=p_0(d)$ such that for all $q\ge1$ and $p<p_{0}$, there exists $C=C(p)>0$
and $\lambda=\lambda(p)>0$ such that:
\begin{enumerate}
\item For all $x_{0},\,y_{0}\in\Omega_{n}$, $k\le o(n^{1/(d+2)})$ and $e\in E_{n}$,
it holds that
\[
\mathbb{\widehat{\mathbb{P}}}_{x_{0},\,y_{0}}\left[\widehat{X}_{k {n^{d}}}(e)\neq\widehat{Y}_{k {n^{d}}}(e)\right]\le e^{-Ck}\;.
\]
\item The mixing time   $\hat{t}_{\rm mix}$ of discrete process $\widehat X_k$ is $\Theta({n^{d}}\log n)$.
\item For
all $n\in\mathbb{N}$, $\hatl(n)\ge\lambda  {n^{-d}}$.
\end{enumerate}
\end{thm}
{ \begin{rem}Indeed, one can take $p_0$  to be the critical Bernoulli bond percolation probability on $\mathbb Z^d$. For $d=2$, thanks to the complete knowledge about
exponential decay of connectivity up to the critical point established in \cite{BD}, the results of Theorem \ref{t21} were shown to hold for all
subcritical $p,$ for each $q\ge 1$ in \cite{BS}.
\end{rem} }
\begin{proof}
(1) and (2) appear as  \cite[Display (13)]{BS}, and \cite[Theorem 6.1]{BS} respectively. Note that  (1) proves the upper bound in (2)
 by taking $k=C\log n.$
The proof of the lower bound of mixing time appears in \cite[Theorem 6.1]{BS}. Although (3) does not quite appear in \cite{BS} it is a consequence of (1). To see this, we will use the well known lower bound of total
variation distance in terms of spectral gap recalled in Proposition
\ref{propgap}. Namely,
using the above and union bounding over all elements in $E_{n},$
we get that $\hat{d}(kn^{d})$, the worst-case total variation
distance at time $kn^{d}$ is $e^{-\Omega(k)+d\log n}$, and hence
\[
(1-\hatl)^{kn^{d}}\le e^{-\Omega(k)+d\log n}
\]
 for all $k\le o(n^{1/(d+2)})$. Now taking logs we get $-kn^{d}\hatl\le-\Omega(k)+d\log n$,
and therefore for some $C>0$,
\[
\frac{1}{n^{d}}\left(C-\frac{\log n}{k}\right)\le\hatl\;.
\]
Thus by choosing a large enough $k=o(n^{1/(d+2)})$ the result follows.
\end{proof}

However for our purposes, we will need a translation of the result for the continuous time setting.
Denote by $\lambda(n)=\lambda(n,\,p,\,q),$
the spectral gap of the continuous time FK-dynamics defined in Definition
\ref{def11}.
\begin{cor}
\label{cor23}For any dimension $d$, there exists $p_0=p_0(d)$ such that for all $q\ge1$ and $p<p_{0}$, there exists $C=C(p)>0$
and $\lambda=\lambda(p)>0$ such that:
\begin{enumerate}
\item For all $x_{0},\,y_{0}\in\Omega_{n}$ and  $k\le o(n^{1/(d+2)})$, it
holds that,
\[
\mathbb{{\mathbb{P}}}_{x_{0},\,y_{0}}\left[{X}_{t}(e)\neq {Y}_{t}(e)\right]\le e^{-Ct}\;.
\]
\item The FK-dynamics in $\Lambda$ has mixing time of order $\Omega(\log n)$.
\item For all $n\in\mathbb{N}$, it holds that $\lambda(n)\ge\lambda$.
\end{enumerate}
\end{cor}

\begin{proof}
All these results are immediate from Theorem \ref{t21} since the
continuous dynamics is $n^{d}$ times faster than the discrete counterpart.
In particular, to show part (3), see \cite[Lemmas 20.5 and 20.11]{LPW}.
\end{proof}

\section{\label{sec3}Auxiliary percolation models,  and disagreement propagation bounds} Given the randomness defined by the update sequence in \eqref{upseq}, we will need to define several auxiliary percolation models extracted from the graphical construction, which though simple will be useful in various comparison arguments appearing throughout the paper. We will also state useful bounds on the speed of propagation of disagreements. We start with the percolation models.
Before providing precise definitions, for the reader's benefit we give short descriptions off what each of these models capture. Furthermore, for ease of reference throughout the article, all the definitions are collected in Table \ref{chart} at the end of this section and the reader can choose to skip the precise definitions at first read referring to the table whenever needed.

\begin{enumerate}
\item Standard Percolation dynamics ($q=1$)/Random walk on the hypercube, i.e., edges are randomly refreshed at rate one with a Bernoulli($p$) variable  independently.  This will  dominate the FK-dynamics in the regime of our interest.
\item Enlarged percolation: An edge is said to be open if it was open at least once in the Standard Percolation dynamics in a given (to be specified) time interval.
\item Update/Non-update percolation: An edge is open if it has not been updated at least once in a given interval of time.
\end{enumerate}

\subsection{\label{sec31}Standard percolation dynamics (SPD) }
{It will be useful to discretize time as we will see in later applications. Throughout the article we will fix $\Delta:=\Delta(p)=p^{-1/2}$, to be the basic unit of discretization and let $\tau_{i}:=i\Delta$.} (The choice of $\Delta$ is not special as long as it satisfies the properties discussed in this section.) Also let $\mathbb{Z}_+$ be the set of non-negative integers.
\begin{defn}[SPD associated to the update sequence $\upd$]
\label{def31}For each $i\in\mathbb{Z}_{+}$, we construct a SPD $(\ful_{t}^{i})_{t\ge\tau_{i}}$ in $\Omega_{n}$ as follows:
\begin{enumerate}
\item $\ful_{\tau_{i}}^{i}=E_{n}$.
\item For each $t>\tau_i$ and $e\in E_{n}$,
\begin{enumerate}
\item If $\upd[\tau_{i},\,t](e)=\emptyset$, we let $\ful_{t}^{i}(e)=\ful_{\tau_{i}}^{i}(e)(=1)$.
\item Otherwise, let $(t^{*},\,U^{*})$ be the last update in $\upd[\tau_{i},\,t](e)$.
\begin{enumerate}
\item We let $\ful_{t}^{i}(e)=1$ if $U^{*}>1-p,$
\item else let $\ful_{t}^{i}(e)=0$ if $U^{*}\le1-p.$
\end{enumerate}
\end{enumerate}
\end{enumerate}
We define the dynamics $(\emp_{t}^{i})_{t\ge\tau_{i}}$ in an identical
manner by replacing step (1) with $\emp_{\tau_{i}}^{i}=\emptyset$.
In other words, $(\ful_{t}^{i})$ and $(\emp_{t}^{i})$ are the Glauber
dynamics of the percolation measure with open probability $p$  on
$\Omega_{n}$ starting at $t=\tau_{i}$ from the full and empty configurations,
respectively.
\end{defn}

Since $(\ful_{t}^{i})$ and $(\emp_{t}^{i})$, for $i\in\mathbb{Z_{+}},$ and
the FK-dynamics $(X_{t}),$ share the same update sequence, we can
couple all of them in the time window $[\tau_{i},\,\infty)$ in a
natural manner calling this as the \textit{canonical
coupling}.
We record some simple but useful lemmas below.
\begin{lem}
\label{lem031}Under the canonical coupling, for all $i\in\mathbb{Z}_{+}$, it holds that
\[
X_{t}\le\ful_{t}^{i}\text{ for all }t\ge\tau_{i}\;.
\]
\end{lem}

\begin{proof}
Denote by $X_{t}^{\textrm{full}}$ the FK-dynamics on $\Omega_{n}$ with $X_{0}=E_n,$
the full configuration. Via the monotone coupling, we have $X_{t}\le X_{t}^{\textrm{full}}$
for all $t\ge0$. Now the inclusion $X_{t}^{\textrm{full}}\le\ful_{t}^{0}$
for all $t\ge0$ comes directly from the definitions of FK-dynamics
and percolation dynamics.  Since we have $\ful_{t}^{0}\le\ful_{t}^{i}$
for all $t\ge\tau_{i}$ for all $i\in\mathbb{Z}_{+}$ under the canonical
coupling, we are done.
\end{proof}

For $s\in[0,\,1]$, denote by $\per(s)$
the standard bond percolation on $E_n$ where an edge $e$ is open with probability $s$. Denote by $\preceq$ the usual stochastic domination.

\begin{lem}
\label{lem032}For all $i\in\mathbb{Z}_{+}$ and $t\ge0$, the law
of $\ful_{t+\tau_{i}}^{i}$ is given by  $\per(e^{-t}+p[1-e^{-t}])$.
Therefore, for all $x_{0}\in\Omega_{n}$, it holds that
\[
\mathbb{P}_{x_{0}}\left[X_{t}\in\cdot\,\right]\preceq\per(e^{-t}+p[1-e^{-t}])\;.
\]
\end{lem}

\begin{proof}
By definition, $\ful_{t+\tau_{i}}^{i}(e)=1$ if $\upd[\tau_{i},\,\tau_{i}+t](e)=\emptyset$.
Otherwise, i.e., if $\upd[\tau_{i},\,\tau_{i}+t](e)\neq\emptyset$,
\[
\ful_{t+\tau_{i}}^{i}(e)=\begin{cases}
1 & \text{with probability }p,\\
0 & \text{with probability }1-p
\end{cases}
\]
since the status of $\ful_{t+\tau_{i}}^{i}(e)$ depends only on the
last update for this edge before $t+\tau_{i}$. Since
\[
\mathbb{P}[\upd[\tau_{i},\,\tau_{i}+t](e)=\emptyset]=e^{-t}\;,
\]
it follows that
\[
\mathbb{P}[\ful_{t+\tau_{i}}^{i}(e)=1]=e^{-t}+p[1-e^{-t}]\;.
\]
The proof of the first assertion is completed since the status of
edges are independent  under SPD. The second
assertion follows from Lemma \ref{lem031} and choosing  $i=0$.
\end{proof}

As indicated in {Section \ref{iop},} we will allow ourselves an $O(1)$ burn-in time which will be enough by the above domination results for the configuration to look like a sample of a subcritical percolation. This then creates a situation where no connected component is large and hence the interactions between various edges are still rather local.
To make this formal, denote by $p_{\textrm{perc}}(d)\in(0,\,1)$ the critical probability
of the edge percolation in $\mathbb{Z}^{d}$.
\textbf{From now on we will  assume that $p\in(0,\,p_{\textrm{perc}}(d))$ and
further arguments would put additional smallness conditions on $p$.}
Define
$$
\pini=\pini(p):=\frac{1}{2}(p+p_{\textrm{perc}}(d))\in(p,\,p_{\textrm{perc}}(d))\;,$$ and let  $\tini=\tini(p)$ be the solution of the following equation:
\begin{equation}
p(1-e^{-\tini})+e^{-\tini}=\pini\;.\label{pini}
\end{equation}
As the next lemma will show, we can restrict our initial conditions to the class
of measures $\nu$ satisfying $\nu\preceq\per(\pini)$. More precisely,
define
\[
\widehat{d}(t)=\sup_{\nu:\nu\preceq\per(\pini)}\left\Vert \mathbb{P}_{\nu}\left[X_{t}\in\cdot\,\right]-\mui\right\Vert _{\textrm{TV}}\;,
\]
and
\[
\widehat{t}_{\textrm{mix}}(\epsilon)=\inf\big\{ t:\widehat{d}(t)<\epsilon \big\}\;.
\]
{{Then, we obtain the following comparison result between $t_{\textrm{mix}}(\epsilon)$ and $\widehat{t}_{\textrm{mix}} (\epsilon)$.
\begin{lem}
\label{pro35}For all $p<p_{\textrm{perc}}(d)$ and $t>\tini$, we
have
\begin{equation}
\sup_{x_{0}\in\Omega_{n}}\left\Vert \mathbb{P}_{x_{0}}\left[X_{t}\in\cdot\,\right]-\mui\right\Vert _{\textrm{TV}}\le\sup_{\nu:\nu\preceq\per(\pini)}\left\Vert \mathbb{P}_{\nu}\left[X_{t-\tini}\in\cdot\,\right]-\mui\right\Vert _{\textrm{TV}}\;.\label{ep35}
\end{equation}
Therefore, we have
\begin{equation}
\widehat{t}_{\textrm{mix}} (\epsilon)\le t_{\textrm{mix}} (\epsilon) \le\widehat{t}_{\textrm{mix}(\epsilon)}+\tini
\;.\label{ep36}
\end{equation}
\end{lem}
}}
\begin{proof}
By Lemma \ref{lem032} and definition of $\tini$ and $\pini$, we
have that the distribution of $X_{\tini}$ given any initial configuration
is stochastically bounded by $\per(\pini)$. Hence, the first assertion
of proposition follows. The inequalities in \eqref{ep36} follow since $\widehat{d}(t)\le d(t)\le\widehat{d}(t-\tini)$
by \eqref{ep35}.
\end{proof}

\begin{center}\textit{\textbf{Thus we will take}} $\tini$ \textit{\textbf{to be our burn-in time.}}\end{center}

\subsection{\label{sec32}Enlarged and non-update percolations}
{{In this section we define the second and the third models indicated at the beginning of the section.}}

 \begin{defn}
\label{def37}We define two sequences of random configurations $(\eful_{i})_{i\in\mathbb{N}}$
and $(\eemp_{i})_{i\in\mathbb{N}}$ in $\Omega_{n}$ based on the definitions $(\ful^{i}_t)_{i\in\mathbb{N}}$
and $(\emp^{i}_t)_{i\in\mathbb{N}}$ as follows:
\begin{enumerate}
\item For $i\in\mathbb{N}$, define $\eful_{i}\in\Omega_{n}$ as
\[
\eful_{i}(e)=1\text{ iff }\ful_{t}^{i-1}(e)=1\text{ for some }t\in[\tau_{i},\,\tau_{i+1}]\;.
\]
Note that here we consider $\ful_{t}^{i-1}$ instead of $\ful_{t}^{i}$ since otherwise $\eful_{i}(e)$ would be deterministically $1.$
\item For $i\in\mathbb{Z}_{+}$, define $\eemp_{i}\in\Omega_{n}$ as
\[
\eemp_{i}(e)=1\text{ iff }\emp_{t}^{i}(e)=1\text{ for some }t\in[\tau_{i},\,\tau_{i+1}]\;.
\]
\end{enumerate}
\end{defn}
 The following result is a static version of Lemma \ref{lem031}.
\begin{lem}
\label{lem038}Under the canonical coupling, for all $i\in\mathbb{N}$,
we have
\[
X_{t}\le\eful_{i}\;\text{for all }t\in[\tau_{i},\,\tau_{i+1}]\;.
\]
\end{lem}

\begin{proof}
Since $X_{t}\le\ful_{t}^{i-1}$ for all $t\in[\tau_{i},\,\tau_{i+1}]$
by Lemma \ref{lem031}, the proof is immediate from the definition
of $\eful_{i}$.
\end{proof}
Now we investigate the distributions of $\eemp_{i}$ and $\eful_{i}$.
To this end we introduce the non-update percolation $\nup_{i}\in\Omega_{n}$,
for $i\in\mathbb{Z}_{+}$, as the following:
\begin{equation}\label{nonup21}
\nup_{i}(e)=\begin{cases}
1 & \text{if }\upd[\tau_{i},\,\tau_{i+1}](e)=\emptyset\;,\\
0 & \text{if }\upd[\tau_{i},\,\tau_{i+1}](e)\neq\emptyset\;.
\end{cases}
\end{equation}
In order words, $\nup_{i}(e)=0$ if and only if there is an update
$(t_{1},\,U_{1})\in\upd(e)$ such that $t_{1}\in(\tau_{i},\,\tau_{i+1}]$.
Given the above definitions, we have the following comparison results.
\begin{lem}
\label{pbd}The following holds:
\begin{enumerate}
\item For all $i\in\mathbb{Z}_{+}$, we have  $\eemp_{i}\preceq\per(p^{1/2})$.
\item For all $i\in\mathbb{Z}_{+}$, we have $\nup_{i}\preceq\per(p^{1/2})$.
\item For all $i\in\mathbb{N}$, we have $\eful_{i}\preceq\per(3p^{1/2})$.
\end{enumerate}
\end{lem}

\begin{proof}
We start by observing that $\eemp_{i}(e)=1$ if and only if
\begin{equation}
\text{\ensuremath{\{U>1-p} for some }(t,\,U)\in\upd[\tau_{i},\,\tau_{i+1}](e)\}\label{e391}\;.
\end{equation}
To compute the probability of the latter notice that given
the event $|\upd[\tau_{i},\,\tau_{i+1}](e)|=k$, the event \eqref{e391}
happens with probability $1-(1-p)^{k}$. Hence, the probability of
the event \eqref{e391} can be written as
\[
\sum_{k=0}^{\infty}e^{-\Delta}\frac{\Delta^{k}}{k!}(1-(1-p)^{k})=1-e^{-p\Delta}\le p\Delta=p^{1/2}\;.
\]
This finishes the proof of (1). Part (2) can be readily obtained from
the observation that
\[
\nup_{i}\sim\per(e^{-\Delta})\preceq\per(p^{1/2})\;.
\]
For part (3), we claim that
\begin{equation}\label{e391-1}
\eful_{i}\le\nup_{i-1}\cup\eemp_{i-1}\cup\eemp_{i}\;.
\end{equation}
This claim along with parts (1) and (2) will finish the proof. To  prove the claim, first suppose that $\eful_{i}(e)=1$ and $\nup_{i-1}(e)=0$.
Then, $\upd[\tau_{i-1},\,\tau_{i}](e)\neq\emptyset$ and hence we
can take the last update $(t_{1},\,U_{1})$ in $\upd[\tau_{i-1},\,\tau_{i}](e)$.
Since $\eful_{i}(e)=1$, at least one update $(t,\,U)$ in $\{(t_{1},\,U_{1})\}\cup\upd[\tau_{i},\,\tau_{i+1}](e)$
satisfies $U>1-p$. It implies either $\eemp_{i-1}(e)=1$ or $\eemp_{i}(e)=1$.
This finishes the proof.
\end{proof}
We end this section with a final definition.
For $i\in\mathbb{Z}_{+}$, let
\begin{equation}
\env_{i}:=\eemp_{i}\cup\nup_{i}\in\Omega_{n}\;.\label{env}
\end{equation}
We record a key fact in the next lemma. In short the lemma says that the FK-dynamics across time can be dominated by a sequence of Bernoulli percolations which are one dependent across time. This will be crucially used in the analysis of how information spreads in the FK-dynamics.
\begin{prop}
\label{penv}The following hold:
\begin{enumerate}
\item For all $i\in\mathbb{Z}_{+}$, the distribution of $\env_{i}$ is
stochastically dominated by $\per{(2p^{1/2})}.$
\item For all $i\in\mathbb{N}$, under the canonical coupling, we have that
\[
X_{t}\le\env_{i-1}\cup\env_{i}\;\text{for all }t\in[\tau_{i},\,\tau_{i+1}]\;.
\]
\end{enumerate}
\end{prop}

\begin{proof}
The proof of part (1) is immediate from (1) and (2) of Lemma \ref{pbd}, while the proof of part (2)
is an immediate consequence of Lemma \ref{lem038} and \eqref{e391-1}.
\end{proof}
For purpose of easy reference throughout the article we record all the percolation models defined so far in  Table \ref{chart}.
\begin{table}
\begin{tabular*}{.8615\textwidth}{ | c | c | c |}
  \hline
  {\rule{0pt}{2.4ex}}\textbf{Percolation}   & \textbf{Description} & \textbf{Defined in}{\rule[-1.2ex]{0pt}{0pt}}\\
  \hline
 {\rule{0pt}{2.4ex}}$\emp^{i}_t$ & Percolation on $[\tau_i,\infty)$,  starting at $\tau_i$ from empty& Def. \ref{def31}{\rule[-1.2ex]{0pt}{0pt}}\\
  \hline
 {\rule{0pt}{2.4ex}}$\ful^{i}_t$ & Percolation on $[\tau_i,\infty)$, starting at $\tau_{i}$ from full & Def. \ref{def31}{\rule[-1.2ex]{0pt}{0pt}}\\
  \hline
   {\rule{0pt}{2.4ex}}$\eemp_{i}$ & Open some time in $[\tau_i,\tau_{i+1}]$ starting with empty at {\rule{0pt}{2.4ex}}$\tau_{i}$ &  Def. \ref{def37}{\rule[-1.2ex]{0pt}{0pt}}\\
  \hline
   {\rule{0pt}{2.4ex}}$\eful_{i}$ & Open some time in $[\tau_i,\tau_{i+1}]$ starting with full at {\rule{0pt}{2.4ex}}$\tau_{i-1}$ &  Def. \ref{def37}{\rule[-1.2ex]{0pt}{0pt}}\\
  \hline
   {\rule{0pt}{2.4ex}}$\nup_{i}$ & Non-update in $[\tau_{i},\tau_{i+1}]$ implies open & \eqref{nonup21}{\rule[-1.2ex]{0pt}{0pt}}\\
  \hline
   {\rule{0pt}{2.4ex}}$\env_{i}$ & $\eemp_{i} \cup \nup_{i} $& \eqref{env}{\rule[-1.2ex]{0pt}{0pt}}\\
\hline
  \end{tabular*}\\
 \vspace{.2in}
\caption{Different kinds of percolation.}
\label{chart}
\end{table}

\subsection{\label{sec33}Decay of connectivity}

We now record some useful exponential decay of connectivity results for a non-equilibrium RCM. It is well-known that for a  sub-critical bond percolation or RCM, one observes an exponential decay of connectivity, i.e., {{the probability that two sites $u$ and $v$ belong to the same
cluster decays exponentially in the graph distance $d(u,\,v)$,
(cf.
\cite[Theorem 2]{BD}).}}
We  would need a dynamical version for our purposes and  start with some definitions.
Note that $\eful_{i}$ had so far been defined for $i\ge1$ only. We now
define $\eful_{0}$ as
\[
\eful_{0}=X_{0}\cup\eemp_{0}\;.
\]
Then, by definition
\begin{equation}
X_{t}\le\eful_{0}\text{ for all }t\in[0,\,\tau_{1}]\;.\label{ep312}
\end{equation}

\begin{prop}
\label{propdec}For all small enough $p$, there exists $\gamma=\gamma(p)>0$
such that,
\[
\sup_{\nu:\nu\preceq\per(\pini)}\mathbb{P}_{\nu}\bigg[\,u\stackrel{\eful_{i}}{\longleftrightarrow}v\,\bigg]\le e^{-\gamma d(u,\,v)}
\]
for all $i\in\mathbb{Z}_{+}$, $n\in\mathbb{N}$, and $u,\,v\in\Lambda_{n}$.
\end{prop}

\begin{proof}
By Lemma \ref{pbd}, the distribution of $\eful_{i}$ is dominated
by $\per(3p^{1/2})$ for $i\ge1$. For $i=0$, we notice from the
definition of $\eful_{0}$ that the distribution of the latter is
dominated by $\per(\pini+p^{1/2})$.

In conclusion, for all small enough $p$, the distribution
of $\eful_{i}$, $i\ge 0$, is dominated by $\per(s)$
for some $s<p_{\textrm{perc}}(d)$ and hence we are done by decay of connectivity for subcritical percolation  \cite[Theorem 3.7]{Hug18}.
\end{proof}

{{From now on, all the statements are asymptotic in $n$, so that they
 hold only when $n$ is large enough. In addition, we  write $C$ or $c$ for positive constants whose different occurrences might denote different values.  We shall not repeat stating these explicitly. 
}}

The next result follows from similar arguments as in the proof of the previous proposition. 
\begin{lem}\label{contour}Suppose that two disjoint subsets $A$ and $B$ of $E_n$ satisfy
  $d(A,\, B)\ge c\log^2 n$ for some $c>0$. Denote by $\mu_{B^c}^{+}$ the random-cluster measure on
  $B^c=E_n\setminus B$ under the full boundary condition on $B$. Denote by $X\in\{0,\,1\}^{B^c}$ a random-cluster configuration sampled
  according to  $\mu_{B^c}^{+}$, and denote by $\conn(B;X)$ the set of edges in $B^c$
   connected to an edge of $B$ via an open path in $X$. Then, for all small enough $p$,  we have
\begin{equation*}
\mu_{B^c}^{+} \left[ \conn(B;X) \cap A =\emptyset \right] \ge 1- \frac{1}{n^{2d}}\;.
\end{equation*}
\end{lem}
\begin{proof}
{{One can readily observe that the decay of connectivity established in the previous result holds for any connected domain with any  boundary condition. Hence, we get
 }}
\begin{eqnarray*}
\mu_{B^c}^{+} \left[\conn(B;X) \cap A  \neq \emptyset \right]
   &\le & \mu_{B^c}^{+} \left[u\stackrel{X}{\longleftrightarrow}v
   \mbox{ for some }u,\,v\in B^c \mbox{ such that }d(u,\,v)\ge c\log^2 n\right]\\
  &\le & e^{-c\log^2 n} {{n^d}\choose{2}} < \frac{1}{n^{2d}}\;,
\end{eqnarray*}
where the second inequality follows by the union bound.
\end{proof}

The final result of this section records a statement about how fast disagreement percolates in FK-dynamics.
\subsection{\label{sec5}Estimates on the propagation of disagreements }
We fix a subset $A\subset  E_{n}$ this section. Define an enlargement
$A^{+}$ of $A$ as
\begin{equation}
A^{+}=\{e\in E_n :d(e,\,A)\le\log^{4}n\}\;.\label{aplus}
\end{equation}
The main objective in this section is to show that,  under monotone coupling, FK-dynamics started from two configurations that agree on $A^{+}$ and are reasonably sparse,  continue to agree on $A$ for all $t\in[0,\,\smax]$ where
\begin{equation}
\smax=\log^{2}n\;.\label{esm}
\end{equation}
 Consider two censored dynamics $(Z_{t}^{+})$ (resp. $(Z_{t}^{-})$)
as FK-dynamics on $\{0,\,1\}^{A^{+}}$ conditioned on full (resp.
empty) configuration on $E_{n}\setminus A^{+}$. Let $\per^{A^{+}}(\pini)$ denote
the percolation measure on $A^{+}$ with open probability $\pini.$
\begin{lem}
\label{lem0311} Consider two copies of FK-dynamics $(Z_{t}^{+})$
and $(Z_{t}^{-})$ on $\{0,\,1\}^{A^{+}}$ coupled via the monotone
coupling. Suppose that the law of the initial condition $Z_{0}^{-}$
follows a law $\nu$ on $\{0,\,1\}^{A^{+}}$ satisfying $\nu\preceq\per^{A^{+}}(\pini)$,
and suppose further that $Z_{0}^{+}=Z_{0}^{-}$. Then, for all sufficiently
small $p$, we have that
\begin{equation}\label{esm1}
\mathbb{P}\,\big[\,Z_{t}^{+}(A)=Z_{t}^{-}(A)\text{ for all }t\in[0,\,\smax]\,\big]\ge1-\frac{1}{n^{3d}}\;.
\end{equation}
\end{lem}
\begin{figure}[h]
\centering
\includegraphics[scale=.24]{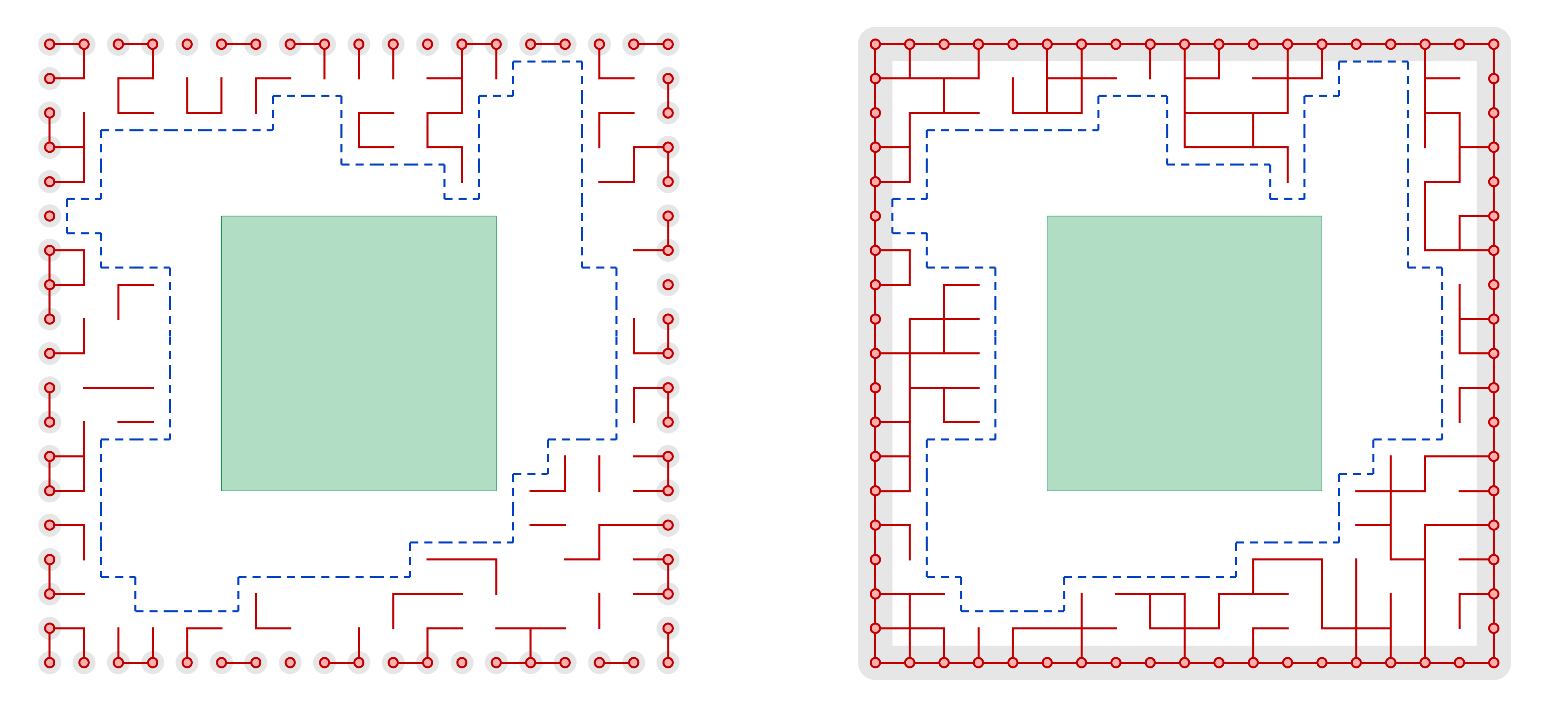}
\caption{Figure illustrating the weak spatial mixing property of the subcritical RCM.
Here we consider the equilibrium measures with free (LHS) and wired boundary conditions (RHS).
By monotonicity of the equilibrium measures with respect to their boundary conditions, there exists a
coupling such that the LHS is dominated by the RHS. However under this coupling by the exponential
 decay of connectivity the RHS (and hence the LHS) has a closed surface (contour in the planar case) within $O(\log n)$ distance
  from the boundary and they agree in the interior of the surface in particular on the green region. }
\label{fig2}
\end{figure}
\begin{rem}
{Note that the  probability in \eqref{esm1} is  with respect to both the FK-dynamics and also  the initial measure $\nu$. In other words, this is an annealed probability.}
\end{rem}

{
\begin{rem}
\label{rm52}Even though we considered the two worst boundary conditions, namely,
full and empty, a simple monotonicity consideration allows us to conclude that
$$
\mathbb{P}\,\big[\,Z_{t}^{+}(A)=Z_{t}^{-}(A)=Z_t(A)\text{ for all }t\in[0,\,\smax]\,\big]\ge1-\frac{1}{n^{3d}}\;,
$$
where $Z_t$ is one of the following processes on $\{0,\,1\}^{A^+}$:
\begin{itemize}
\item Censored FK-dynamics on $A^{+}$ conditioned on any configuration on
$E_{n}\setminus A^{+}$, i.e., one that only updates sites in $A^{+}.$
\item If $A$, and hence $A^{+}$, are square boxes, the FK-dynamics on $\{0,\,1\}^{A^{+}}$
with periodic boundary conditions.
\item The FK-dynamics on $E_n$ projected to $A^+\subset E_n$, i.e., $X_{t}(A^+).$
\end{itemize}
\end{rem}
}

\begin{rem}
\label{rm53}{In the above theorem, the size of the ambient space $\Lambda_n$ (which is $n$)  is not
important. }Taking the ambient space to be $\Lambda_m$ which contains $A^+$ suffices.  Moreover, we can replace $\log^{4}n$ in the
statement of lemma with $\log^{3+\delta}$ for any $\delta>0$ with
$\smax=\log^{1+\delta}n$.
\end{rem}
The proof follows the arguments in \cite{BS, MS13} and is postponed to the Appendix (Section \ref{appendix1}).

\section{\label{sec4}Information percolation clusters and time dependent Bernoulli percolations}
As emphasized before, this is the section which contains all the new ideas in the paper. The main result is the following bound on $L^{2}$-mixing. Recall the spectral gap $\lambda(n)$ from Corollary \ref{cor23}.
\begin{thm}
\label{t41}For all small enough $p>0$, there exists $C=C(p)>0$
such that the following $L^{2}$-bound holds for all large enough
$n$:
\[
\max_{x_{0}\in\Omega_{n}}\left\Vert \mathbb{P}_{x_{0}}\left[X_{t}\in\cdot\,\right]-\mui\right\Vert _{L^{2}(\mui)}\le2\exp\left\{ -\lambda(n)(t-C\log n)\right\}
\]
 for all $t\ge C\log n$.
\end{thm}

Recall that the spectral gap governs the rate of decay of $L^2$ norm. More precisely for any $s\le t$ and any starting state $x_0\in \Omega_n$ we have {{(see for example, \cite[Lemma 20.5]{LPW}),}}
\begin{equation}\label{l2contrac}
\left\Vert \mathbb{P}_{x_{0}}\left[X_{t}\in\cdot\,\right]-\mui\right\Vert _{L^{2}(\mui)}\le e^{-\lambda(n)(t-s)}\left\Vert \mathbb{P}_{x_{0}}\left[X_{s}\in\cdot\,\right]-\mui\right\Vert _{L^{2}(\mui)}\;.
\end{equation}

By Corollary \ref{cor23}, it suffices to prove the following
proposition.
\begin{prop}
\label{p46}For all small enough $p>0$, there exists $C=C(p)>0$
such that for $t_{\star}=C\log n$,
\begin{equation}\label{boundl2}
\max_{x_{0}\in\Omega_{n}}\left\Vert \mathbb{P}_{x_{0}}\left[X_{t_{\star}}\in\cdot\,\right]-\mui\right\Vert _{L^{2}(\mui)}\le2\;.
\end{equation}
\end{prop}

The proof of Proposition \ref{p46} is the heart of this work and is rather long, intricate and involves several percolation arguments based on the models introduced in Section \ref{sec3}. As mentioned earlier, using the results of this section as inputs, the arguments of the following sections are quite similar to the ones appearing in \cite{LS1, NS}.  Readers not familiar with these papers,  at first read, to get a sense of the overall flow of arguments, could choose to assume Theorem \ref{t41} and read the subsequent easier sections first, before coming back to this section.

We provide a roadmap for this section for the ease of reading.
\begin{itemize}
\item The construction of information percolation is done in Section \ref{sec41} relying on the definitions in Section \ref{sec3}, particularly the percolation models listed in Table \ref{chart}.  At a very high level it amounts to classifying vertices into green, red and blue where the state of the red vertices depend on the initial configuration, the blue vertices are independent Bernoulli variables independent of everything else, whereas the green vertices have a complicated dependency on each other but are still independent of the  initial configuration (Theorem \ref{pinfo}).
\item Using the above, the proof of Proposition \ref{p46} occupies  Sections \ref{sec42} and \ref{sec43}. The key steps are the following:
\begin{enumerate}
\item To bound the $L^2$-distance it suffices to condition on the green clusters. Then the strategy is to compute the $L^2$-distance of the conditional distribution to a product Bernoulli Measure instead of the equilibrium measure (Lemma \ref{lem48}). The Bernoulli measure is exactly the one which describes the law of the blue vertices. Thus this distance would be zero if there does not exist any red cluster.
\item We then establish the key estimate showing exponential unlikeliness of red vertices with time in Proposition \ref{p45} which makes the above step sufficient. The proof of this proposition uses a comparison with a subcritical branching process and is presented in Section \ref{sec43}. In particular, the proof involves delicate geometric arguments relying on several properties of the auxiliary percolation models defined in Table \ref{chart}.
\end{enumerate}

\end{itemize}

\subsection{\label{sec41}Information percolation (IP)}
As mentioned before (Section \ref{sec31}), we will discretize time using $\tau_i$ and will define IP on the space-time
slab $E_{n}\times[\tau_{1}, \,\tau_{m}]$ for some $m\in\mathbb{N}$.
We shall take $m=\Omega(\log n)$ later, but for the  moment we think of
$m$ as a fixed integer. We also recall the various percolations defined in Table \ref{chart}.

{{For $\Xi\in\Omega_{n},$ and $e=(u,v)\in E_{n}$  where $u,v \in \Lambda_n$,  if $\Xi(e)=1$, define  $\conn(e;\Xi)$ as
the connected component of $\Xi$ containing $(u,\,v)$. On the other hand, we define $\conn(e;\Xi)=\emptyset$ if $\Xi(e)=0$. }}

Furthermore define $\partial\conn(e;\Xi)$ as the edge boundary of $\conn(e;\Xi)$ i.e.,
as the set of edges in $E_{n}\setminus\conn(e;\Xi),$  which are adjacent
to an edge in $\conn(e;\Xi)$ and define
\begin{equation}
\overline{\conn}(e;\Xi)=\conn(e;\Xi)\cup\partial\conn(e;\Xi)\;.\label{barc}
\end{equation}
We set  $\overline{\conn}(e;\Xi)=\emptyset$ if $\Xi(e)=0$.
Given the above notations, we now define IP for the
FK-dynamics. It would be notationally convenient to define $\tau_{i+1/2}=(i+1/2)\Delta,$
for $i\in\mathbb{N}.$
Furthermore to distinguish between edges (elements of $E_n$) and connections across time,
we will call the former `space edges'  as just edges and the latter as `time edges'
 ({see} Figure \ref{ip} for an illustration).

\begin{defn}[Information percolation]
\label{def44}The information percolation cluster is defined on the
space-time slab $E_{n}\times[\tau_{1},\,\tau_{m}]$ for some fixed
$m\ge2$. For an edge $e\in E_{n}$, we define the history $\mathscr{H}_{e}=(\mathscr{H}_{e}(t))_{t\in[\tau_{1},\,\tau_{m}]}$
associated to the edge $e$ backward in time recursively as follows:
Start by setting $\mathscr{H}_{e}(\tau_{m})=\{e\}$.
\begin{enumerate}
\item For each $t=\tau_{i+1}$ with $i\in\llbracket1,\,m-1\rrbracket,$ suppose
that $\mathscr{H}_{e}(\tau_{i+1})$ is given by a subset of $E_{n}.$
Then we let $\mathscr{H}_{e}(\tau_{i+1/2})$ be the same as $\mathscr{H}_{e}(\tau_{i+1})$, as well  as for any  $w\in\mathscr{H}_{e}(\tau_{i+1})$ we connect  the two edges $(w,\,\tau_{i+1})$ and $(w,\,\tau_{i+1/2})$,
 by a `time edge' in the time direction ({see Figure \ref{ip}}.)
\item For each $w\in\mathscr{H}_{e}(\tau_{i+1/2})$, we check if it has been updated in the time interval $(\tau_i,\tau_{i+1})$ (recall the various notations from Table \ref{chart}).
\begin{enumerate}
\item If $\nup_{i}(e)=1$, then {introduce} the `space edge' $(w,\,\tau_{i})$ and connect
$(w,\,\tau_{i+1/2})$ and $(w,\,\tau_{i})$ by a time edge.
\item If $\nup_{i}(e)=0$, we take the last update $(t_0,\,U_{e})$ for
$e$ in $(\tau_{i},\,\tau_{i+1}]$.
\begin{enumerate}
\item If $U_{e}<1-p+p^{*}$, this update is called \textbf{\textit{oblivious }}and
we do not take any action on the edge $(w,\,\tau_{i+1/2})$.
\item If $U_{e}>1-p+p^{*}$, then $w$ is open in $\eemp_{i}$, and hence is  open in $\env_{i}$ as well (cf. \eqref{env}). In this case, we include all the edges in $\overline{\conn}(w;\env_{i-1}\cup\env_{i})$
in $\mathscr{H}_{e}(\tau_{i+1/2})$ and $\mathscr{H}_{e}(\tau_{i})$.
Finally we  connect the space edges $(w',\,\tau_{i+1/2})$ and $(w',\,\tau_{{i}})$ for
all the edges $w'$ in $\overline{\conn}(w;\env_{i-1}\cup\env_{i})$ using time edges.
\end{enumerate}
\end{enumerate}
\item Steps (1) and (2) above  define  $\mathscr{H}_{e}(\tau_{i})$ as a subset
of $E_{n}$. Now return to the first step if $i\ge2$ to use the above construction recursively.
\end{enumerate}
For $A\subset E_{n}$, define $\mathscr{H}_{A}=(\mathscr{H}_{A}(t))_{t\in[\tau_{1},\,\tau_{m}]}$
as $\mathscr{H}_{A}=\bigcup_{e\in A}\mathscr{H}_{e}$. Two histories
$\mathscr{H}_{e}$ and $\mathscr{H}_{e'}$ are connected if they share
an edge.
\end{defn}
Some remarks are in order. First, we emphasize  that
two histories $\mathscr{H}_{e}$ and $\mathscr{H}_{e'}$ are regarded as two disconnected pieces if
they share vertices only. Second, by the construction rule, one can observe that:
\begin{equation}\label{consistent}
\mathscr{H}_{e}(\tau_{i+1/2})=\mathscr{H}_{e}(\tau_{i+1}) \cup \mathscr{H}_{e}(\tau_{i})\;.
\end{equation}

{Using terminology from existing literature we will often refer to the collection $\sh:=\{\sh_e\}_{e\in E_n}$ as the history diagram.}
This induces a new graph structure on $E_n.$ i.e. $e$ and $e'$ are connected if $\mathscr{H}_{e}$ and $\mathscr{H}_{e'}$ are connected. {{Note that the vertex set for this graph is $E_n.$}}

With the above conventions, each connected component
of this new graph is called an \textit{\textbf{information percolation
cluster.}} We shall simply refer to them as  \textit{clusters}.
{Let them be indexed by the set $\mathcal{C}$.}

\begin{figure}[h]
\centering
\includegraphics[scale=.20]{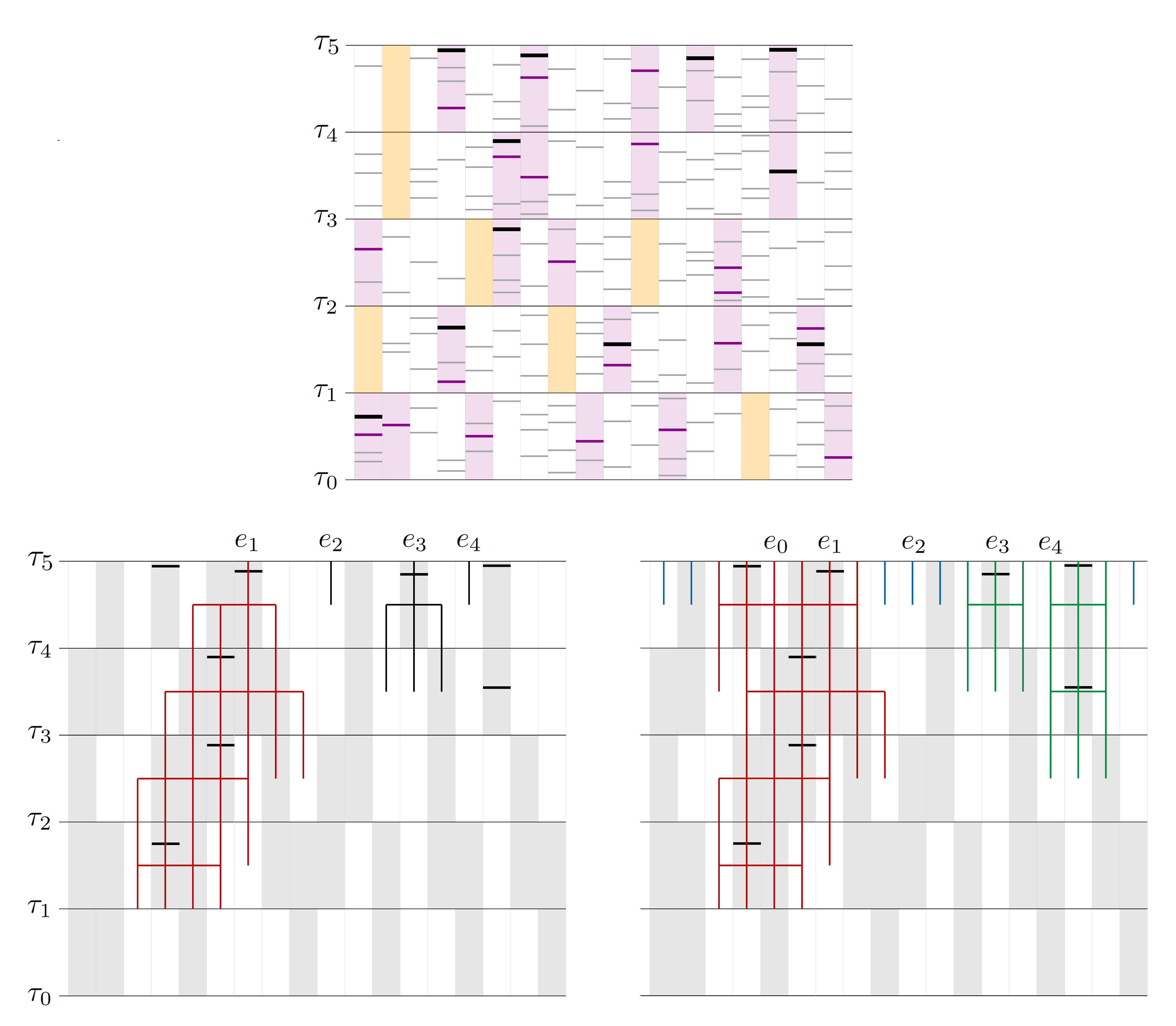}
\caption{{\textbf{(Up)} The various colors indicate the values of the uniform variables for each update: gray $\leftrightarrow \{U<1-p\}$, purple $\leftrightarrow \{1-p\le U \le 1-p+p^*\}$, black $\leftrightarrow \{U>1-p+p^*\}$. The purple region denotes $\eemp_i(e)=1$, while
the yellow region implies that $\nup_i(e)=1$.
\textbf{(Down)} In the two graphs, the gray region indicates whether $\Xi_i \cup \Xi_{i-1} (e)$ is $1$ or $0$ with gray indicating the former.
\textbf{(Down-left)} History diagrams for $e_1, \,e_2,\,e_3,\,e_4$. We can assert that $\mathscr H_{e_1}$ is red, but not able to say
anything about the remaining ones;
\textbf{(Down-right)} History diagram for $e_0$ is combined with that of $e_1$. $e_4$ belongs to green cluster although its last update is oblivious}. The vertical edges   acting as connections across time are referred to as `time edges' in the article. }
\label{ip}
\end{figure}

\begin{defn}[IP clusters and their colors]
\label{def45}Each cluster $C\in \mathcal{C}$ is colored  red, blue or green according to the following rule:
\begin{itemize}
\item Colored \textbf{red} if $\mathscr{H}_{C}(\tau_{1})\neq\emptyset$.
\item Colored \textbf{blue} if $\mathscr{H}_{C}(\tau_{1})=\emptyset$
and $|C|=1$.
\item Colored \textbf{green} if $\mathscr{H}_{C}(\tau_{1})=\emptyset$
and $|C|\ge2$.
\end{itemize}
Denote by $\mathcal{C}_{\mathcal{R}}$, $\mathcal{C}_{\mathcal{B}}$
and $\mathcal{C}_{\mathcal{G}}$ the collection of red, blue and green
clusters, respectively. Define
\begin{equation}\label{rc1}
E_{\mathcal{R}}=\left\{ e:e\in C \mbox{ for some } C\in \mathcal{C}_{\mathcal{R}}\right\}
\end{equation}
and define $E_{\mathcal{B}}$ and $E_{\mathcal{G}}$ similarly. We
use the following simplified notations to denote the history diagrams emanating from the various colored edges:
\[
\mathscr{H}_{\mathcal{R}}:=\mathscr{H}_{E_{\mathcal{R}}}\;,\;\; \mathscr{H}_{\mathcal{B}}:=\mathscr{H}_{E_{\mathcal{B}}}\;,\text{ and}\;\;\mathscr{H}_{\mathcal{G}}:=\mathscr{H}_{E_{\mathcal{G}}}\;.
\]

\end{defn}

The following theorem justifies the above definitions. In short, it says that to reconstruct the state of the edges in $\mathscr{H}_{A}(\tau_{i+1}),$ all one needs is the update sequence and the state of the edges $\mathscr{H}_{A}(\tau_{i})$ at time $\tau_i$ provided that $A$ is a cluster.

\begin{thm}
\label{pinfo}{Given a history diagram $\sh$, suppose that a set $A\subset E_{n}$ is a cluster.} Then, for each $i\in\llbracket1,\,m-1\rrbracket$,
the configuration $X_{\tau_{i+1}}(\mathscr{H}_{A}(\tau_{i+1}))$ is
a deterministic function of
\begin{equation}\label{tcw}
X_{\tau_{i}}(\mathscr{H}_{A}(\tau_{i}))\;\;\text{and\;\;}\bigcup_{e\in\mathscr{H}_{A}(\tau_{i+1/2})}\upd[\tau_{i},\,\tau_{i+1}](e)\;.
\end{equation}
In particular, if $\mathscr{H}_A (\tau_i)=\emptyset,$ for some $i\ge 1$, then $X_{\tau_{i+1}}$ is independent of $X_{\tau_{i}}$ and therefore of $X_{\tau_{1}}$.

\end{thm}

\begin{rem}
Note that not all update sequences are compatible with the diagram $\sh$. In particular, the inner boundary of Green cluster is always closed and hence any update sequence for which the diagram occurs with positive probability must respect such constraints.
\end{rem}

The proof of the above theorem is provided below after introducing some notations  and observing some consequences of the already stated definitions.
{{We momentarily}} fix $A\subset E_{n}$ and suppressing the dependence on $A,$ define
\begin{equation}
W_{j}=\mathscr{H}_{A}(\tau_{j})\;;\;j\in\llbracket1,\,m\rrbracket\;.\label{wj}
\end{equation}
{{
In the proof of the main result of this section (i.e., Theorem \ref{t41}), the key ingredient is the analysis of the evolution of $|W_j|$ backwards in time.  This is formulated in Proposition \ref{p413}. A crucial role is played by the following two decompositions of $W_j$. The first decomposition is according to the type of evolution that occurs in the time interval $[\tau_{j-1},\,\tau_j]$:}}
\begin{equation}
W_{j}=W_{j}^{\textrm{NU}}\cup W_{j}^{\textrm{Ob}}\cup W_{j}^{\textrm{NOb}}\;,\label{decw1}
\end{equation}
where,
\begin{align*}
W_{j}^{\textrm{NU}} & =\left\{ e\in W_{j}:\nup_{j-1}(e)=1\right\} \text{ i.e., the edges that have not been updated in } [\tau_{j-1},\tau_j] \;,\\
W_{j}^{\textrm{Ob}} & =\left\{ e\in W_{j}:\nup_{j-1}(e)=0\;\text{and }\text{the last update for }e\text{ in }[\tau_{{j-1}},\,\tau_{{j}}]\text{ is oblivious}\right\} \;,\\
W_{j}^{\textrm{NOb}} & =\left\{ e\in W_{j}:\nup_{j-1}(e)=0\;\text{and }\text{the last update for }e\text{ in }[\tau_{{j-1}},\,\tau_{{j}}]\text{ is non-oblivious}\right\} \;.
\end{align*}

{{
Now, we consider the second decomposition of $W_j$. For this, we classify each edge according to the origin of its evolution in $[\tau_j,\,\tau_{j+1}]$.
}}
For each $j\in\llbracket1,\,m-1\rrbracket$, we write
\begin{equation}
C_{j}  =\bigcup_{e\in W_{j+1}^{\textrm{NOb}}}\overline{\conn}(e;\env_{j-1}\cup\env_{j})\;.\label{epi1}
\end{equation}
{{
Hence, the set $C_j$ represents the collection of edges in $W_j$ that arise from non-oblivious expansions (i.e., step (2)-(b)-(ii) of Definition \ref{def44}). Since each edge in $W_j$ is either due to such an expansion or is inherited  from $W_{j+1}$ owing to no update at the corresponding edge in the time interval $[\tau_j,\,\tau_{j+1}]$, we obtain that
}}
\begin{equation}\label{ne01}
W_{j}=C_{j}\cup W_{j+1}^{\textrm{NU}}\;.
\end{equation}
Therefore, by writing
\begin{equation}
N_{j}  =W_{j}\setminus C_{j}\;,\label{epi2}
\end{equation}
we obtain another decomposition of $W_{j}$ given by
\begin{equation}
W_{j}=C_{j}\cup N_{j}\;.\label{decw2}
\end{equation}
We next record some basic properties of these decompositions.

\begin{lem}
\label{lem43}For all $j\in\llbracket1,\,m-1\rrbracket$, it holds
that
\[
W_{j+1}^{\textrm{NOb}}\subset C_{j}\;\;\text{and\;\;}N_{j}\subset W_{j+1}^{\textrm{NU}}\;.
\]
\end{lem}

\begin{proof}
For the first inclusion, we note that $e\in W_{j+1}^{\textrm{NOb}}$
implies that $\eemp_{j}(e)=1$ and thus $\env_{j}(e)=1$. Hence, the
definition \eqref{epi1} indicates that $e\in C_{j}$ as well and
thus the first inclusion trivially holds. For the latter one, it suffices
to recall \eqref{ne01} and the definition
\eqref{epi2} of $N_{j}$.
\end{proof}

For $S\subset E_{n}$, define $\partial^{-}S$ as the set of edges in
$S$ which are adjacent to at least one edge in $S^{c}$, i.e., $\partial^{-}S=\partial(E\setminus S)$.
We record the following simple fact.
\begin{lem}
\label{lem44}For all $j\in\llbracket1,\,m-1\rrbracket$, all the
edges in $\partial^{-}C_{j}$ are closed in $\env_{j-1}\cup\env_{j}$.
In particular, there is no open path in $\env_{j-1}\cup\env_{j}$
connecting an open edge in $C_{j}$ and an edge in $N_{j}$.
\end{lem}

\begin{proof}
The proof is direct from the definition of $C_{j}$ {where we included the closed (outer) boundary of ${\conn}(e;\env_{j-1}\cup\env_{j})$} .
\end{proof}
\begin{lem}
\label{lem45}For all $j\in\llbracket1,\,m-1\rrbracket$, for each
$e\in C_{j}$, and for all $t\in[\tau_{j},\,\tau_{j+1}]$, the process $X_{t}(e)$ is a deterministic function of
\[
X_{\tau_{j}}(\mathscr{H}_{C_{j}}(\tau_{j}))\;\;\text{and\;\;}\bigcup_{e'\in C_{j}}\upd[\tau_{j},\,t](e')\;.
\]
\end{lem}
\begin{proof}
Let
\[
\mathcal{U}_{t}=\bigcup_{e'\in C_{j}}\upd[\tau_{j},\,t](e')\;\;;\;t\in[\tau_{j},\,\tau_{j+1}]\;.
\]
We fix $e\in C_{j}$ and $t\in[\tau_{j},\,\tau_{j+1}]$ and denote
by $(t_{0},\,U_{0})$ the last update for $e$ in $[\tau_{j},\,t]$.
If $U_{0}<1-p+p^{*}$ then, in view of Definition \ref{def11}, the
configuration $X_{t}(e)$ is $1$ if $U_{0}<1-p$, and $0$ if $U_{0}\ge1-p$.
Thus, we can determine $X_{t}(e)$ solely in terms of $(t_{0},\,U_{0})\in\upd[\tau_{j},\,t](e)\subset\mathcal{U}_{t}$.
Now we consider the case $U_{0}>1-p+p^{*}$. In this case, the configuration
$X_{t}(e)=X_{t_0}(e)$ is determined by checking whether $e$ is
a cut-edge or not in the configuration $X_{t_{0}-}$. In order to
check this, one has to investigate $\conn(e;X_{t_0-}\cup\{e\})$
to determine whether removing $e$ disconnects some component of $X_{t_0-}\cup\{e\}$
or not. Note that $e$ is open in $\eemp_{i}$ (and hence in $\env_{i}$)
since $U_{0}>1-p+p^{*}>1-p$. Thus, by Proposition \ref{penv}, we
have
\begin{equation}
\conn(e;X_{t_{0}-}\cup\{e\})\subset\conn(e;\env_{j}\cup\env_{j+1})\subset C_{j} \setminus \partial^- C_j
\;.
\end{equation}
Therefore, we can determine $X_{t}(e)$
in terms of $X_{t_{0}-}(C_{j})$ and $(t_{0},\,U_{0})\in\mathcal{U}_{t}$.

If $\mathcal{U}_{t_{0}}=\{(t_{0},\,U_{0})\}$, we have $X_{t_{0}-}(C_{j})=X_{\tau_{j}}(C_{j}),$
so we can conclude the proof. Otherwise, we take the last update $(t_{1},\,U_{1})$
in $\mathcal{U}_{t_{0}}$ other than $(t_{0},\,U_{0})$. Then, we
have,
\[
X_{t_{0}-}(C_{j})=X_{t_{1}}(C_{j})\;.
\]
Since there are finitely many updates in $[\tau_{j},\,\tau_{j+1}]$
almost surely, we can repeat this procedure to finish the proof. An important fact implicitly used  above is that in repeating the argument all the edges $\tilde e$ that we encounter with an update time $\tilde t \in [\tau_j,t_0]$ has the property that the connected component of
$$
X_{\tilde t -}(\tilde e) \subset \overline{\conn}(e;\env_{j-1}\cup\env_{j}) \subset C_j\;.
$$
 since the edge boundary of $\conn(e;\env_{j-1}\cup\env_{j})$ remains closed throughout the interval $[\tau_j,\tau_{j+1}].$
\end{proof}
The proof of  Theorem \ref{pinfo} now follows.
\begin{proof}[Proof of Theorem \ref{pinfo}]
{In view of  \eqref{decw1} and the first inclusion of Lemma \ref{lem43},}  it suffices to consider the following
three cases
separately.
\begin{itemize}
\item Case 1: $e\in W_{i+1}^{\textrm{NU}}$. By (2)-(a) of Definition \ref{def44},
we have $X_{\tau_{i+1}}(e)=X_{\tau_{i}}(e)$ and thus configuration
of $X_{\tau_{i+1}}(e)$ is determined by $X_{\tau_{i}}(W_{i+1}^{\textrm{NU}})$.
Since $W_{i+1}^{\textrm{NU}}\subset W_{i}$ (cf. \eqref{ne01}), the proposition holds
for this case. {
\item Case 2: $e\in W_{i+1}^{\textrm{Ob}}\setminus C_i$. By (2)-(b)-(i) of Definition
\ref{def44}, the configuration $X_{\tau_{i+1}}(e)$ is solely determined
by the last update for $e$ in $(\tau_{i},\,\tau_{i+1}]$ and therefore
the proposition holds as well.
\item Case 3: $e\in W_{i+1}\cap C_i$. This case is immediate from
Lemma  \ref{lem45}.
}
\end{itemize}
\end{proof}
The following corollary is an immediate consequence of the previous
theorem.
\begin{cor}
\label{cor4s} Given a history diagram $\sh$, the following holds.
\begin{enumerate}
\item The configurations $X_{\tau_{m}}(E_{\mathcal{G}})$ and $X_{\tau_{m}}(E_{n}\setminus E_{\mathcal{G}})$
are independent.
\item The configuration $X_{\tau_{m}}(E_{\mathcal{G}})$ is independent
of  $X_{\tau_1}$.
\item For $e\in E_{\mathcal{B}}$, the distribution of $X_{\tau_{m}}(e)$
is a Bernoulli random variable with parameter $\frac{p^{*}}{1-p+p^{*}}$,
and is independent of all other randomness.
\end{enumerate}
\end{cor}

\begin{proof}
{Parts (1) and (2) are direct consequences of  Theorem \ref{pinfo} and the
definition of a green cluster.} We now consider part (3). For $e\in E_{\mathcal{B}}$,
the configuration $X_{\tau_{m}}(e)$ is determined by the last update
$(t,\,U)$ for $e$ in $[\tau_{1},\,\tau_{m}]$. Furthermore, since
$e\in E_{\mathcal{B}}$, this last update is oblivious and therefore
we know that $U<1-p+p^{*}$. Given this condition, we have $X_{\tau_{m}}(e)=1$
if $U<1-p$ and $X_{\tau_{m}}(e)=0$ if $U\in[1-p,\,1-p+p^{*}]$ otherwise.
This finishes the proof of part (3).
\end{proof}

For each $A\subset E_{n}$, define
\[
\mathscr{H}_{A}^{-}=\mathscr{H}_{E_{n}\setminus A}\;.
\]
As in \cite{NS,LS3}, it would be crucial to estimate the probability of $A$ being a red cluster or a collection of singleton blue clusters i.e.,
 \begin{equation}\label{coev}
 \{A\in\mathcal{C}_{\mathcal{R}}\}\cup\{A\subset E_{\mathcal{B}}\}\;.
 \end{equation}
 Furthermore, technical aspects make it important to estimate the above probabilities conditioned on the history diagram of the complement of $A.$ For this conditional probability to be non-zero a necessary condition is that,
 \begin{equation}\label{compatcond}
 \mathscr{H}_{A}^{-}\cap\left\{ A\times\{t=\tau_{m-1/2}\}\right\} =\emptyset\\;,
 \end{equation}
for the following reason. Suppose that $e\in A$ satisfies $(e,\,\tau_{m-1/2})\in\sh_{e'}$ for some $e'\in E_n \setminus A$. Then, by the definition of the information percolation cluster, the cluster containing $e$ must contain $e'$ as well.

Thus this is a  compatibility condition to guarantee that $\{A\in\mathcal{C}_{\mathcal{R}}\}\cup\{A\subset E_{\mathcal{B}}\}$
is a non-empty event which we denote by $\mathscr{H}_{A}^{-}\in\mathscr{H}_{\textup{com}}(A)$. Given this, we  define
\begin{equation}\label{probdef65}
\mathcal{P}_{A}=\sup_{\mathscr{H}_{A}^{-}\in\mathscr{H}_{\textup{com}}(A)}
\mathbb{P}\left[A\in\mathcal{C}_{\mathcal{R}}\,\vert\,\mathscr{H}_{A}^{-},\;
\{A\in\mathcal{C}_{\mathcal{R}}\}\cup\{A\subset E_{\mathcal{B}}\}\right],
\end{equation}
i.e.,  the maximum probability of $A$ being a red cluster conditioned on a compatible  $\mathscr{H}_{A}^{-}.$
Given the above preparation, the following proposition is the main estimate (similar to \cite[Lemma 4.8]{NS}) needed.  For $A\subset E_{n},$ we denote by $|\conn(A)|,$
the smallest  number of edges in any  connected subgraph of
$(\Lambda_{n},\,E_{n})$ containing $A$.
\begin{prop}
\label{p45}For any $\theta>0$, we can find
two constants $C=C(\theta)>0$ and $p_{0}=p_{0}(\theta)>0$ such that,
for any $p\in(0,\,p_{0})$, there exists a constant $\alpha=\alpha(p)>0$
satisfying
\[
\mathcal{P}_{A}\le Ce^{-(\theta|\conn(A)|+\alpha\tau_{m})}\;\;\text{for all  $A\subset E_{n}$}\;.
\]
\end{prop}

{A notable feature of this proposition is the fact that $\alpha$ is
independent of $\theta$. In the remaining part of the current section,
$\alpha$ always refers to the constant above.}
The proof of this proposition is postponed to Section \ref{sec43}.
A corollary of this proposition is the following lemma which lower bounds the probability that there are no red clusters.
\begin{lem}
\label{lem46}For all small enough $p$, there exists a constant $C=C(p)>0$
satisfying
\[
\sup_{\mathscr{H}_{\mathcal{G}}}\mathbb{P}\left[\mathscr{H}_{\mathcal{R}}
=\emptyset|\mathscr{H}_{\mathcal{G}}\right]\ge1-Cn^{2}e^{-\alpha\tau_{m}}\;.
\]
\end{lem}

\begin{proof}
By the union bound and the definition of $\mathcal{P}_{A}$,
\[
1-\mathbb{P}\left[\mathscr{H}_{\mathcal{R}}=\emptyset|\mathscr{H}_{\mathcal{G}}\right]\le\sum_{A\subset E_{n},\,A\neq\emptyset}\mathbb{P}\left[A\in\mathcal{C}_{\mathcal{R}}|\mathscr{H}_{\mathcal{G}}\right]
\le\sum_{A\subset E_{n},\,A\neq\emptyset}\mathcal{P}_{A}\;.
\]
Now, by Proposition \ref{p45} and the translation invariance of the
{periodic lattice},
\begin{equation}
\sum_{A\subset E_{n},\,A\neq\emptyset}\mathcal{P}_{A}\le\sum_{e\in E_{n}}\sum_{A:A\ni e}\mathcal{P}_{A}\le Cn^{2}e^{-\alpha\tau_{m}}\sum_{k=1}^{\infty} \;
\sum_{A:A\ni e ,\,|\conn(A)|=k }e^{-\theta k}\;.\label{er1}
\end{equation}
{{
For a  fixed $e\in E_{n}$, we have that
 $$
 |\{A\subset E_{n}:A\ni e,\,|\conn(A)|=k\}|\le (k+1)(8d^2)^k\;.
 $$
 The verification is elementary and we leave the proof to the reader. Finally, we can combine the last two displays to deduce
 \begin{equation}
\mathbb{P}\left[\mathscr{H}_{\mathcal{R}}=
\emptyset|\mathscr{H}_{\mathcal{G}}\right]
\ge1-Cn^{2}e^{-\alpha\tau_{m}}\sum_{k=1}^{\infty}(k+1)(8d^2e^{-\theta})^{k}\;.\label{er3}
\end{equation}
Now by taking $\theta$ large enough so that $8d^2 e^{-\theta}<1/2$, the proof of the lemma is complete.
}}
\end{proof}

The remainder of the section is now devoted to proving \eqref{boundl2}.
\subsection{\label{sec42}Proof of Proposition \ref{p46}}
For $A\subset E_{n}$, define
$\nu_{A}$ as a Bernoulli percolation measure on $A$ with open probability
$\pstar$ where
\begin{equation}
\pstar=\frac{p^{*}}{1-p+p^{*}}\;.\label{epstar}
\end{equation}
In the remaining part of the section, we will simply write $\mu:=\mui$,
$E:=E_{n}$ and denote by $\mu_{A}$, $A\subset E$, the projection
of $\mu$ on $A$. We first prove the following lemma which shows that the $L^2$-distance to $\mu$ can be controlled by the $L^2$-distance of the measure on the complement of the green clusters to the measure $\nu.$
\begin{lem}
\label{lem48}For all small enough $p$, we can find $C=C(p)>0$ such
that for $m\ge C\log n$ we have
\[
\left\Vert \mathbb{P}_{x_{0}}\left[X_{t}\in\cdot\,\right]-\mu\right\Vert _{L^{2}(\mu)}\le2\sup_{\mathscr{H}_{\mathcal{G}}}\left\Vert \mathbb{P}_{x_{0}}\left[X_{t}(E\setminus E_{\mathcal{G}})\in\cdot\,|\mathscr{H}_{\mathcal{G}}\right]-\nu_{E\setminus E_{\mathcal{G}}}\right\Vert _{L^{2}(\nu_{E\setminus E_{\mathcal{G}}})}+1
\]
for all $x_{0}\in\Omega_{n}$,
\end{lem}

\begin{proof}
Consider two copies of FK-dynamics $(X_{t})$ and $(Y_{t})$ where
$X_{0}=x_{0}$  and $Y_0$ is distributed according to $\mu$. We couple them via the monotone
coupling introduced in Definition \ref{def11}.   Now by   Jensen's inequality (for details see \cite[Lemma 4.13]{NS}) we obtain
{{
\begin{align}
  \left\Vert \mathbb{P}_{x_{0}}\left[X_{t}\in\cdot\,\right]-\mu\right\Vert^2 _{L^{2}(\mu)}& =\left\Vert \mathbb{P}_{x_{0}}\left[X_{t}\in\cdot\,\right]-\mathbb{P}_{\mu}\left[Y_{t}\in\cdot\,\right]\right\Vert^2 _{L^{2}(\mu)}\nonumber \\
 & \le\int\left\Vert \mathbb{P}_{x_{0}}\left[X_{t}\in\cdot\,|\mathscr{H}_{\mathcal{G}}\right]-\mathbb{P}_{\mu}\left[Y_{t}\in\cdot\,
 |\mathscr{H}_{\mathcal{G}}\right]\right\Vert^2 _{L^{2}(\mui(\cdot\,|\mathscr{H}_{\mathcal{G}}))}d\mathbb{P}(\mathscr{H}_{\mathcal{G}}) \nonumber \\
 & \le\sup_{\mathscr{H}_{\mathcal{G}}}\left\Vert \mathbb{P}_{x_{0}}\left[X_{t}\in\cdot\,|\mathscr{H}_{\mathcal{G}}\right]-\mathbb{P}_{\mu}\left[Y_{t}
 \in\cdot\,|\mathscr{H}_{\mathcal{G}}\right]\right\Vert^2_{L^{2}(\mui(\cdot\,|\mathscr{H}_{\mathcal{G}}))}\;.\label{e481}
\end{align}
}}

Given $\mathscr{H}_{\mathcal{G}}$, the diagram $\mathscr{H}_{E\setminus E_{\mathcal{G}}}$
is disjoint from $\mathscr{H}_{\mathcal{G}}=\mathscr{H}_{E_{\mathcal{G}}}$,
and as we noticed in Corollary \ref{cor4s} configurations $X_{t}(E_{\mathcal{G}})$
(resp. $Y_{t}(E_{\mathcal{G}})$) and $X_{t}(E\setminus E_{\mathcal{G}})$
(resp. $Y_{t}(E\setminus E_{\mathcal{G}})$) are independent. {Moreover, $Y_{t}(E_{\mathcal{G}})$ and $X_{t}(E_{\mathcal{G}})$ are identical} by Theorem \ref{pinfo}. Thus,
the projection onto $E\setminus E_{\mathcal{G}}$ does not change
the $L^{2}$-norm. Combining this observation with \eqref{e481},
we obtain
\begin{align}
 & \left\Vert \mathbb{P}_{x_{0}}\left[X_{\tau_{m}}\in\cdot\,\right]-\mu\right\Vert _{L^{2}(\mu)}\nonumber \\
 & \le\sup_{\mathscr{H}_{\mathcal{G}}}\left\Vert \mathbb{P}_{x_{0}}\left[X_{\tau_{m}}(E\setminus E_{\mathcal{G}})\in\cdot\,|\mathscr{H}_{\mathcal{G}}\right]-\mathbb{P}_{\mu}\left[Y_{\tau_{m}}(E\setminus E_{\mathcal{G}})\in\cdot\,|\mathscr{H}_{\mathcal{G}}\right]\right\Vert _{L^{2}(\mu_{E\setminus E_{\mathcal{G}}}(\cdot\,|\mathscr{H}_{\mathcal{G}}))}\;.\label{e482}
\end{align}
Now by Lemma \ref{lem46}, for $m\ge C\log n$ where $C=C(p)$ is large enough,
$$\mathbb{P}\left[\mathscr{H}_{\mathcal{R}}=\emptyset\,|\mathscr{H}_{\mathcal{G}}\right]\ge\frac{1}{2}\;.$$

Then, for all $Z\subset\{0,\,1\}^{E\setminus E_{\mathcal{G}}}$,
we can deduce that,
\begin{equation}
\mathbb{P}_{\mu}\left[Y_{\tau_{m}}(E\setminus E_{\mathcal{G}})=Z\,|\mathscr{H}_{\mathcal{G}}\right]\ge\mathbb{P}\left[\mathscr{H}_{\mathcal{R}}=\emptyset\,|\mathscr{H}_{\mathcal{G}}\right]\nu_{E\setminus E_{\mathcal{G}}}(Z)\ge\frac{1}{2}\nu_{E\setminus E_{\mathcal{G}}}(Z)\;.\label{e483}
\end{equation}
Note that the first inequality follows from the fact that the distribution
on $E_{\mathcal{B}}$ is $\nu_{E_{\mathcal{B}}}$, and that under
$\mathscr{H}_{\mathcal{R}}=\emptyset$, we have $E\setminus E_{\mathcal{G}}=E_{\mathcal{B}}$.
We are now able to complete the proof of the lemma by combining \eqref{e482},
\eqref{e483}, and the definition of $L^{2}$-norm.
\end{proof}

Thus the task has now been reduced to measuring the $L^2$-distance of certain measures to the product measure $\nu.$
The Miller-Peres inequality establishes a simple yet extremely useful bound for such cases. It
 first appeared in \cite{millerperes}  where the product measure was given by independent $Ber(1/2)$ variables. This was extended later in
 \cite[Lemma 4.3]{LS4} which is the version we will use.

\begin{lem}
\label{lem49}Let $\Omega=\{0,\,1\}^{S}$ for a finite set $S$, and
let $\eta$ be a probability measure on the space of subsets of $S$.
For each $R\subset S$, suppose that a probability measure $\varphi_{R}$
on $\{0,\,1\}^{R}$ is given. {For $p\in(0,\,1/2)$,} denote by $\nu_{p}$  the measure on $\{0,1\}^{S}$ given by the
product of independent $\text{Ber}(p)$ variables. Let $\mu_{p}$
be a measure on $\Omega$ obtained first by sampling a subset $R$
of $S$ according to $\eta$, and then sampling an element of $\{0,1\}^R$
according to $\varphi_{R}$, and sampling an element of $\{0,1\}^{S\setminus R}$ according to the restriction of $\nu_p$ on $\{0,\,1\}^{S\setminus R}$.
Then, we have
\[
\left\Vert \mu_{p}-\nu_{p}\right\Vert _{L^{2}(\nu_{p})}^{2}\le\mathbb{E}\left[p^{-|R\cap R'|}\right]-1\;,
\]
where $R$, $R'\subset S$ are two independent samples of  $\eta$.
\end{lem}

In view of Lemmas \ref{lem48} and \ref{lem49}, we obtain that
\begin{equation}
\left\Vert \mathbb{P}_{x_{0}}\left[X_{\tau_{m}}\in\cdot\,\right]-\mu\right\Vert _{L^{2}(\mu)}\le2\sup_{\mathscr{H}_{\mathcal{G}}}\mathbb{E}\left[\frac{1}{\pstar^{|E_{\mathcal{R}}\cap E_{\mathcal{R}'}|}}\,\big|\mathscr{H}_{\mathcal{G}}\right]+1\;,\label{aie}
\end{equation}
provided that $p$ is small enough so that $\pstar<1/2$,
for all $m>C_{1}\log n$ where $C_{1}$ is the constant in
Lemma \ref{lem48} and $E_{\mathcal{R}}$ and $E_{\mathcal{R}'}$,
are two independent samples of the set $E_{\mathcal{R}}$ of red clusters (see \eqref{rc1}) conditioned on $\mathscr{H}_{\mathcal{G}}.$
To analyze the right-hand side of \eqref{aie}, we recall the following
domination results from \cite{LS5,LS4}. Let $\{J_{A}:A\subset E\}$ be a family of independent indicators
such that $\mathbb{P}(J_{A}=1)=\mathcal{P}_{A}$ for all $A\subset E$ and similarly let $\{J_{A,\,A'}:A, A'\subset E\}$ be a family of independent indicators
such that $\mathbb{P}(J_{A,\,A'}=1)=\mathcal{P}_{A}\mathcal{P}_{A'}$ for all $A,\,A'\subset E$.
\begin{lem}[\cite{LS5}, Lemma 2.3, Corollary 2.4]
\label{lemls}Then following coupling results hold.
\begin{enumerate}
\item  The conditional distribution of red clusters given $\mathscr{H}_{\mathcal{G}}$
can be coupled to $J_{A}$ such that
\[
\{A:A\in\mathcal{C}_{\mathcal{R}}\}\subset\{A:J_{A}=1\}\;.
\]
\item
Similarly, the conditional distribution of $(E_{\mathcal{R}},\,E'_{\mathcal{R}})$
given $\mathscr{H}_{\mathcal{G}}$ can be coupled such that
\[
|E_{\mathcal{R}}\cap E_{\mathcal{R}'}|\le\sum_{A\cap A'\neq\emptyset}|A\cup A'|J_{A,\,A'}\;.
\]
\end{enumerate}
\end{lem}

We are now ready to prove Proposition \ref{p46}.
\begin{proof}[Proof of Proposition \ref{p46}]
It suffices to prove that the right-hand side of \eqref{aie} is
bounded by $2$ for $m=C\log n$ with large enough $C$. Write $\kappa:=\log (1/\pstar)>0$.
By part (2) of Lemma \ref{lemls}, we have
\begin{align*}
\sup_{\mathscr{H}_{\mathcal{G}}}\mathbb{E}\left[\,\pstar^{\,-|E_{\mathcal{R}}\cap E_{\mathcal{R}'}|}\,\big|\,\mathscr{H}_{\mathcal{G}}\,\right] &
\le\mathbb{E}\exp\Big\{ \,\kappa\sum_{A\cap A'\neq\emptyset}|A\cup A'|J_{A,\,A'}\,\Big\}
\\
 & =\prod_{A\cap A'\neq\emptyset}\mathbb{E}\exp\left\{ \kappa|A\cup A'|J_{A,\,A'}\right\}
 \\
 & \le\prod_{e\in E_{n}}\,\prod_{(A,\,A'):e\in A\;,e\in A'}\left[\,(e^{\kappa(|A|+|A'|)}-1)\mathcal{P}_{A}\mathcal{P}_{A'}+1\,\right]
 \\
 & \le\exp\bigg\{\, |E|\,\Big[\,\sum_{A:e\in A}e^{\kappa|A|}\mathcal{P}_{A}\,\Big]^{2}\,\bigg\} \;,
\end{align*}
where $e$ in the last line is an arbitrary edge in $E$. The last
inequality follows from $x+1\le e^{x}$ and the translation invariance
of the underlying graph. Hence, it suffice to show that
\[
\sum_{A:e\in A}e^{\kappa|A|}\mathcal{P}_{A}\le\frac{1}{n^{3}}
\]
for $m=C\log n$ with sufficiently large $C$. To this end, we recall
Proposition \ref{p45} so that
\[
\sum_{A:e\in A}e^{\kappa|A|}\mathcal{P}_{A}\le Ce^{-\alpha\tau_{m}}\sum_{A:e\in A}e^{\kappa|A|}{}^{-\theta|\conn(A)|}\le Ce^{-\alpha\tau_{m}}\sum_{A:e\in A}e^{(\kappa}{}^{-\theta)|\conn(A)|}\;.
\]
{Thus, we can proceed as in \eqref{er1}  and \eqref{er3}
to deduce that the last summation bounded by is bounded by $1$, provided that $\theta$ is large enough}. This
finishes the proof.
\end{proof}

\subsection{\label{sec43}Proof of Proposition \ref{p45}: domination by subcritical branching processes}

We now prove Proposition \ref{p45} to complete our discussion on
Theorem \ref{t41}. {For $S\subset E$, define $\mathcal{C}_{\mathcal{R}(S)}^{*}$
to be the collection of red clusters that arises when exposing the joint
histories of elements of $S$ i.e., $\mathscr{H}_S$ only.
 Similarly define $\mathcal{C}_{\mathcal{B}(S)}^{*}$
 for blue clusters.}
\begin{lem}\label{keybound}
There exists $c=c(p)>0$ such that, for all $A\subset E$ we have
\begin{equation}
\mathcal{P}_{A}\le e^{c|\conn(A)|}\,\mathbb{P}\left[A\in\mathcal{C}_{\mathcal{R}(A)}^{*}\right]\;.\label{e431}
\end{equation}
\end{lem}

To prove the above we will first attempt to understand the effect of conditioning on the event
$\mathscr{H}_{A}^{-}=\mathcal{X}\in\mathscr{H}_{\textrm{com}}(A)$.  We will determine a subset of $\upd[0,\,\tau_{m}]$ that is enough to determine
the event $\{\mathscr{H}_{A}^{-}=\mathcal{X}\}$. We write $\mathcal{X}_{i}=\mathcal{X}\cap\{t=\tau_{i}\}$
for $i\in\{n/2:n\in\mathbb{Z}\}$. Recall from \eqref{consistent}  that the event $\{\mathscr{H}_{A}^{-}=\mathcal{X}\}$
is non-empty only when $\mathcal{X}$ satisfies the consistency condition
\begin{equation}
\mathcal{X}_{i+1/2}=\mathcal{X}_{i}\cup\mathcal{X}_{i+1}\text{ for all }i\in\llbracket0,\,m-1\rrbracket\;.\label{xcon}
\end{equation}
For each $i\in\llbracket1,\,m-1\rrbracket$ and $e\in E_{n}$, we
define
\begin{equation}
\mathcal{U}_{i}(e)=\begin{cases}
\upd[\tau_{i-1},\,\tau_{i+1}](e) & \text{if }e\in\mathcal{X}_{i}\;,\\
\upd[\tau_{i},\,\tau_{i+1}](e) & \text{if }e\in\mathcal{X}_{i+1}\setminus\mathcal{X}_{i}\;,\\
\emptyset & \text{otherwise.}
\end{cases}\label{yi}
\end{equation}
Then, we define
\[
\mathcal{U}_{i}=\bigcup_{e\in E_{n}}\mathcal{U}_{i}(e)\;\;\;\text{and\;\;\;\ensuremath{\mathcal{U=}}}\bigcup_{i=1}^{m-1}\mathcal{U}_{i}\;.
\]
Note that $\mathcal{U}$ depends on $\mathcal{X}$.
\begin{lem}\label{lem518}
The event $\{\mathscr{H}_{A}^{-}=\mathcal{X}\}$ is independent of
the update variables not in $\mathcal{U}$.
\end{lem}
\begin{proof}
Write $\mathcal{Y}_{i}=\mathscr{H}_{A}^{-}\cap\{t=\tau_{i}\}$ and
define the event $\mathcal{E}_{i}$ by
\[
\mathcal{E}_{i}=\{\mathcal{Y}_{i}=\mathcal{X}_{i}\}\;.
\]
If $\mathcal{X}$ satisfies the condition \eqref{xcon}, we can write
\[
\{\mathscr{H}_{A}^{-}=\mathcal{X}\}=\bigcap_{i=1}^{m}\mathcal{E}_{i}\;.
\]
We claim that given $\mathcal{E}_{i+1}$, the event $\mathcal{E}_{i}$
depends only on the events in $\mathcal{U}_{i}$. Given $\mathcal{E}_{i+1}$,
we decompose $\mathcal{Y}_{i+1}=\mathcal{X}_{i+1}$ as following (similar
to those in Theorem \ref{pinfo}):
\begin{align*}
 & \mathcal{Y}_{i+1}^{\textrm{NU}}=\left\{ e\in\mathcal{Y}_{i+1}:\nup_{i}(e)=1\right\} \;,\\
 & \mathcal{Y}_{i+1}^{\textrm{Ob}}=\left\{ e\in\mathcal{Y}_{i+1}:\nup_{i}(e)=0\text{ and the last update for }e\text{ in }[\tau_{i},\,\tau_{i+1}]\text{ is oblivious}\right\} \;,\\
 & \mathcal{Y}_{i+1}^{\textrm{NOb}}=\left\{ e\in\mathcal{Y}_{i+1}:\nup_{i}(e)=1\text{ and the last update for }e\text{ in }[\tau_{i},\,\tau_{i+1}]\text{ is non-oblivious}\right\}\;.
\end{align*}
This classification can be carried out if we only know
\[
\bigcup_{e\in\mathcal{Y}_{i+1}}\upd[\tau_{i},\,\tau_{i+1}](e)\subset\mathcal{U}_{i}\;.
\]
Now we suppose that this classification is given. Then, we have
\[
\mathcal{Y}_{i}=\mathcal{Y}_{i+1}^{\textrm{NU}}\cup\bigcup_{e\in\mathcal{Y}_{i+1}^{\textrm{NOb}}}\overline{\conn}(e;\Xi_{i-1}\cup\Xi_{i})\;,
\]
and therefore $\mathcal{X}_{i}=\mathcal{Y}_{i}$ holds if
\[
\mathcal{X}_{i}\setminus\mathcal{Y}_{i+1}^{\textrm{NU}}\subset\bigcup_{e\in\mathcal{Y}_{i+1}^{\textrm{NOb}}}\overline{\conn}(e;\Xi_{i-1}\cup\Xi_{i})\subset\mathcal{X}_{i}\;.
\]
This event can be determined by knowing $\bigcup_{e\in\mathcal{X}_{i}}\mathcal{U}_{i}(e).$
This finishes the proof.
\end{proof}
\begin{lem}\label{lem519}For all $\mathcal{X}$ satisfying \eqref{xcon}, it holds that
$$\mathbb{P}\left[A\in\mathcal{C}_{\mathcal{R}(A)}^{*},\,\mathscr{H}_{A}\cap\mathcal{X}=\emptyset\,\big\vert\,\mathscr{H}_{A}^{-}
=\mathcal{X}\right]=\mathbb{P}\left[A\in\mathcal{C}_{\mathcal{R}(A)}^{*},\,\mathscr{H}_{A}\cap\mathcal{X}=\emptyset\right]\;.$$
\end{lem}

\begin{proof}
We keep the notation from the previous lemma. In view of the previous
lemma, it suffices to demonstrate that the event
\[
\mathcal{E}=\{A\in\mathcal{C}_{\mathcal{R}(A)}^{*}\}\cap\{\mathscr{H}_{A}\cap\mathcal{X}=\emptyset\}
\]
does not depend on the updates in $\mathcal{U}$. Note that this event is the same as saying $\mathscr{H}_{A}$ reaches $t=\tau_{1}$ without touching
 $\mathcal{X}$. We prove this by induction (see Figure \ref{fig52} for an illustration). Write $W_{i}(A)=\mathscr{H}_{A}\cap\{t=\tau_{i}\}$
and for each $i\in\llbracket2,\,m\rrbracket$, define the event $\mathcal{E}_{i}$
as
\[
\mathcal{E}_{i}=\{W_{i-1}(A)\neq\emptyset\;\;\text{and}\;\;W_{i-1}(A)\cap(\mathcal{X}_{i-3/2}\cup\mathcal{X}_{i-1/2})=\emptyset\}\;.
\]

\begin{figure}[h]
\centering
\includegraphics[scale=.20]{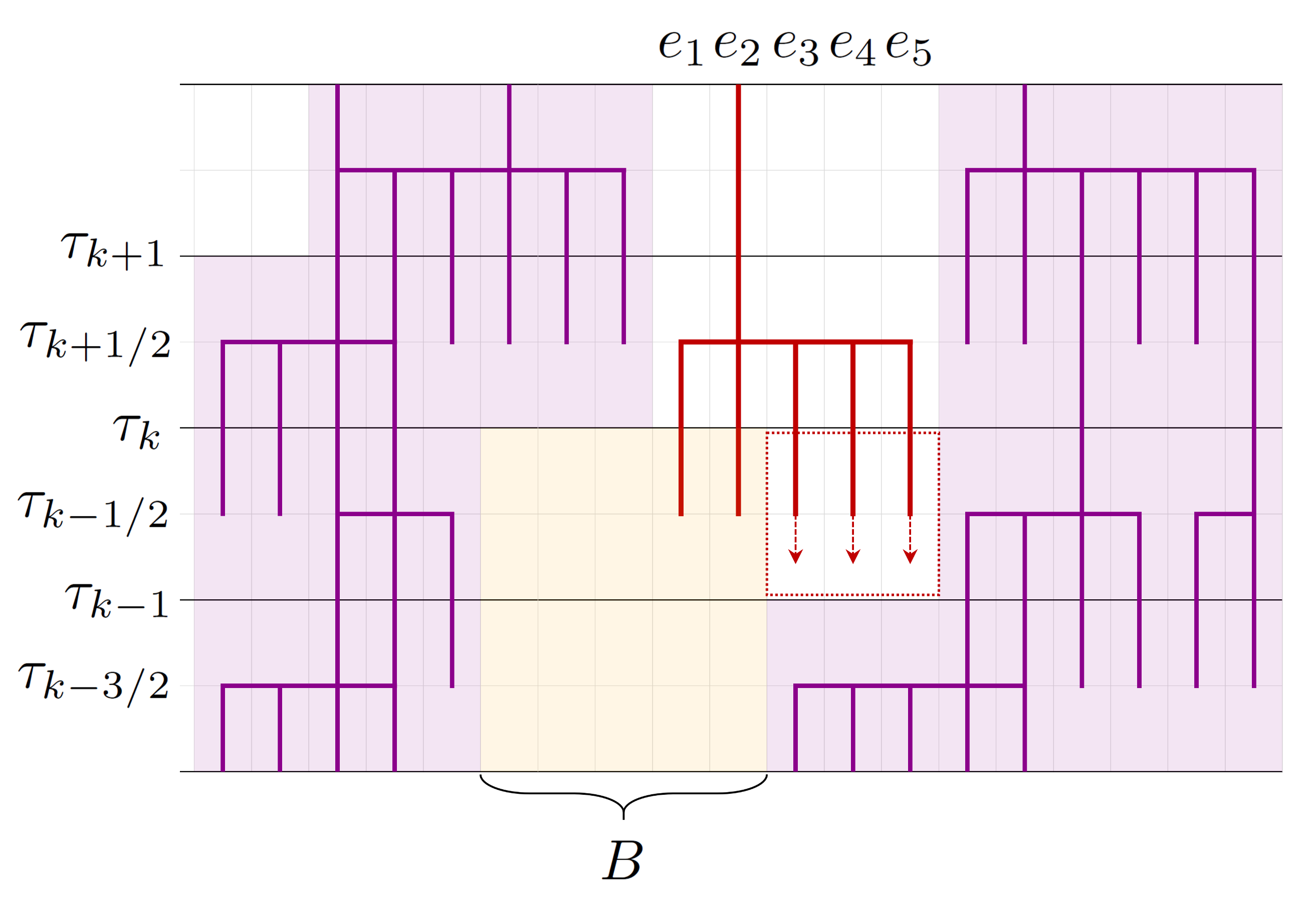}
\caption{Illustrating the proofs of Lemmas \ref{lem518} and \ref{lem519}. The purple graph is $\mathcal{X}$ and $\mathcal{U}$ is the set of updates in the purple region. The red graph is the history diagram $\mathscr{H}_A$.
At time $t=\tau_k$, the occurrence of the event $\mathcal{E}_k$ does not depend on the updates in the purple region. Note that for the latter event to occur  $e_3, \,e_4,\,e_5$ cannot hit the purple region and hence the last updates for each of them  in $(\tau_{k-1},\,\tau_k ]$ should be oblivious. This depends on the updates in the red box. For $e_1$ and $e_2$, they can be expanded and one of them must be to ensure that they all together form a red cluster. However this expansion should be confined to $B$. This can be determined by the updates in yellow region and therefore also independent of updates in the purple region.}
\label{fig52}
\end{figure}

We suppose that $\mathcal{X}$ satisfies $\mathcal{X}_{m}\cap A=\emptyset$
and $\mathcal{X}_{m-1/2}\cap A=\emptyset$ since otherwise the event
$\{\mathscr{H}_{A}\cap\mathcal{X}=\emptyset\}$ (and hence $\mathcal{E})$
cannot happen. Under this minimal consistency assumption, we can write
$\mathcal{E}=\bigcap_{i=2}^{m}\mathcal{E}_{i}$.

We now claim that, for each $k\in\llbracket2,\,m-1\rrbracket$, given
$\bigcap_{i=k+1}^{m}\mathcal{E}_{i}$, the event $\mathcal{\mathcal{E}}_{k}$
does not depend on updates in $\mathcal{U}$ . For each $e\in W_{k}(A)$,
we consider two cases:
\begin{enumerate}
\item $e\in\mathcal{X}_{k-3/2}$: The last update in $(\tau_{k-1},\,\tau_{k}]$
must be oblivious to have $W_{k-1}(A)$ disjoint from $\mathcal{X}_{k-3/2}$.
This update belongs to $\mathcal{U}$ only if $e\in\mathcal{X}_{k-1}$
or $e\in\mathcal{X}_{k}$, which cannot happen under $\mathcal{E}_{k+1}$,
since under $\mathcal{E}_{k+1}$ the set $W_{k}(A)$ is disjoint to
$\mathcal{X}_{k-1/2}=\mathcal{X}_{k-1}\cup\mathcal{X}_{k}$ (cf. \eqref{xcon}).
\item $e\notin\mathcal{X}_{k-3/2}$: We still have two cases: either there
is no update in $(\tau_{k-1},\,\tau_{k}]$ for the edge $e$, or the
last update in $(\tau_{k-1},\,\tau_{k}]$ for the edge $e$ is oblivious
and
\[
\conn(e;\env_{k-1}\cup\env_{k})\bigcap(\mathcal{X}_{k-3/2}\,\cup\,\mathcal{X}_{k-1/2})=\emptyset\;.
\]
Determining whether this holds or not can be performed by looking only
at
\[
\bigcup_{e'\notin\mathcal{X}_{k-3/2}\cup\mathcal{X}_{k-1/2}}\upd[\tau_{k-2},\,\tau_{k}](e')\;.
\]
These updates are disjoint to $\mathcal{U}$ since $e'\notin\mathcal{X}_{k-3/2}\cup\mathcal{X}_{k-1/2}$
implies $\upd[\tau_{k-2},\,\tau_{k}](e')\cap\mathcal{U}=\emptyset$
\end{enumerate}
Furthermore, the non-emptiness of $W_{k-1}(A)$ implies that at least
one of the last updates of $e\in W_{k}$ is non-oblivious or there
is an edge $e\in W_{k}$ such that there is no update in $(\tau_{k-1},\,\tau_{k}]$.
By the same reasoning as (1), this is independent of the updates
in $\mathcal{U}$.
Summing up, for the event $\mathcal{E}_{i}$ to occur, all the events
 described above must occur simultaneously and the probability of this is  independent of the conditioning on the randomness in $\mathcal{U}$.
\end{proof}

\begin{proof}[Proof of Lemma \ref{keybound}]
Given the above preparation, the remaining steps of  the proof already appears in \cite{LS2, NS}. Note first that, conditioned
on $\mathscr{H}_{A}^{-}=\mathcal{X}\in\mathscr{H}_{\textrm{com}}(A)$,
one has
$\{A\in\mathcal{C}_{\mathcal{R}}\}= \{ A\in\mathcal{C}_{\mathcal{R}(A)}^{*} \} \cap \{ \mathscr{H}_{A}\cap\mathcal{X}=\emptyset \} \;$
and similarly
$ \{A\subset\mathcal{C}_{\mathcal{B}}\}=\{ A\subset\mathcal{C}_{\mathcal{B}(A)}^{*} \} \cap \{ \mathscr{H}_{A}\cap\mathcal{X}=\emptyset \}.
$
Therefore, we can deduce
\begin{align}
 & \mathbb{P}\left[A\in\mathcal{C}_{\mathcal{R}}\,\big\vert\,\mathscr{H}_{A}^{-}=\mathcal{X},\;\{A\in\mathcal{C}_{\mathcal{R}}\}\cup\{A\subset E_{\mathcal{B}}\}\right]\nonumber \\
=\; & \mathbb{P}\left[A\in\mathcal{C}_{\mathcal{R}(A)}^{*},\,\mathscr{H}_{A}\cap\mathcal{X}=
\emptyset\,\big\vert\,\mathscr{H}_{A}^{-}=\mathcal{X},\;\{A\in\mathcal{C}_{\mathcal{R}}\}\cup\{A\subset E_{\mathcal{B}}\}\right]\nonumber \\
=\; & \frac{\mathbb{P}\left[A\in\mathcal{C}_{\mathcal{R}(A)}^{*},\,\mathscr{H}_{A}\cap\mathcal{X}=\emptyset
\,\big\vert\,\mathscr{H}_{A}^{-}=\mathcal{X}\right]}{\mathbb{P}\left[\{ A\in\mathcal{C}_{\mathcal{R}(A)}^{*} \} \cup \{ A\subset\mathcal{C}_{\mathcal{B}(A)}^{*}\} ,\,\mathscr{H}_{A}\cap\mathcal{X}=\emptyset\,\big\vert\,\mathscr{H}_{A}^{-}=\mathcal{X}\right]}\nonumber \\
\le \;& \frac{\mathbb{P}\left[A\in\mathcal{C}_{\mathcal{R}(A)}^{*},\,\mathscr{H}_{A}\cap\mathcal{X}=\emptyset\right]}{\mathbb{P}\left[A\subset\mathcal{C}_{\mathcal{B}(A)}^{*},
 \mathscr{H}_{A}\cap\mathcal{X}=\emptyset \,\big\vert\,\mathscr{H}_{A}^{-}=\mathcal{X}\right]}\;\;\;\nonumber (\text{by Lemma }\ref{lem519})\\
 \le\; &  \frac{\mathbb{P}\left[A\in\mathcal{C}_{\mathcal{R}(A)}^{*}\right]}{\mathbb{P}\left[A\subset\mathcal{C}_{\mathcal{B}(A)}^{*},\,
 \mathscr{H}_{A}\cap\mathcal{X}=\emptyset \,\big\vert\,\mathscr{H}_{A}^{-}=\mathcal{X}\right]}\;.
 \label{leis}
\end{align}

Now we bound the denominator of \eqref{leis} from below.  Since $\mathcal{X}$ satisfies the compatibility condition \eqref{compatcond} by hypothesis,
an event which implies the event in the denominator is the following: all the edges in $A$ are updated in the time interval $[\tau_{m-1/2},\tau_m]$ with oblivious updates.
Note that this implies that $\mathscr{H}_A$, the history diagram of $A,$
 will only intersect $E\times \{\tau_{m-1/2}, \tau_m\}$ and hence will not intersect $\mathcal{X}.$
Now the probability of an edge being updated in $[\tau_{m-1/2},\tau_m]$ is $1-e^{-\frac{\Delta}{2}}$ where $\Delta$ appeared in the definition of the $\tau_i$'s. Moreover the probability of an update being oblivious is $1-p+p^*.$
Putting the above together, we get that the denominator of \eqref{leis} is bounded below by $e^{-c(p)|A|}$ for some $c(p)>0$.
This completes the proof of \eqref{e431}.
\end{proof}
Thereby, it only remains to prove the following proposition.
\begin{prop}
\label{p413}For any $\theta>0$, we can find two constants $C=C(\theta)>0$
and $p_{0}=p_{0}(\theta)>0$ such that, for any $p\in(0,\,p_{0})$
there exists a constant $\alpha=\alpha(p)>0$ satisfying
\[
\mathbb{P}\left[A\in\mathcal{C}_{\mathcal{R}(A)}^{*}\right]\le Ce^{-(\theta|\conn(A)|+\alpha\tau_{m})}\;\;\text{for all }A\subset E.
\]
\end{prop}

\subsubsection{Domination by sub-critical branching process}\label{dscbp1}

To estimate the probability $\mathbb{P} [A\in\mathcal{C}_{\mathcal{R}(A)}^{*} ]$,
we fix $A$ and $m$ in the remaining part of the current section.
Recall the notation $W_{i}$ from \eqref{wj}. The main idea of the
proof is that for sufficiently small $p$, the sequence $W_{m},\,W_{m-1},\,\ldots,\,W_{1}$
is dominated by a subcritical branching process in a suitable sense
that will be explained below. Note that the event $\{A\in\mathcal{C}_{\mathcal{R}(A)}^{*}\}$
requires that
\begin{enumerate}
\item $\sh_e$ for some $e\in A$ starting  at time $t=\tau_{m}$ survives to time
{$t=\tau_1$}.
\item All the history diagrams $\mathscr{H}_{e}$, $e\in A$, are connected together
before arriving at {$t=\tau_1$}.
\end{enumerate}
Comparing them with sub-critical branching processes will allow us to bound the probabilities of the above events. As Lemma \ref{lem14} and the discussion following that will show, the analysis has to take into account that the $1$-dependence across time of the Bernoulli percolation clusters used to define the information percolation history diagrams prevents a contraction every time step. Nonetheless this is sufficient to yield subcritical behavior once every  two steps which is enough for our purposes.

We start with a general lemma. For $r\in(0,p_{\textrm{perc}}(d))$,
let $\omega_{r}$ be an i.i.d. standard bond percolation configuration on the lattice
$(\mathbb{Z}^{d},\,E(\mathbb{Z}^{d}))$ where each edge is open with probability $r$.
Denote by $\overline{\conn}(e;\omega_{r})$ the closure of the open cluster
containing an edge $e$  as in \eqref{barc}, and let $m_{r}$ be the distribution of $|\overline{\conn}(e;\omega_{r})|$, i.e.,
\begin{equation}\label{upsi}
m_r (k) = \mathbb{P}\left[|\overline{\conn}(e;\omega_{r})|=k\right]\;\;;\;k\in\mathbb{Z}_+\;.
\end{equation}
It is well-known (see \cite{Hug18, gri})
 that there exists a constant $\rho(r)>0$ such that, for all $e\in E(\mathbb{Z}^{d})$,
\begin{equation}
m_r (k)\le e^{-\rho(r)k}\;\;
\text{for all }k\ge 1\;.\label{eqp}
\end{equation}
\begin{lem}
\label{lem412}Fix a non-empty set $A\subset E$ and consider a random
configuration $X\in\Omega_{n}$ whose distribution is stochastically
dominated by $\per(r)$ for some $r\in(0,p_{\textrm{perc}}(d))$.

Given $X$, we define
\begin{equation*}
A(X)=\bigcup_{e\in A} \overline\conn(e;X)\;.
\end{equation*}
Let $(y_{i})_{i=1}^{\infty}$  be a sequence of i.i.d. random variables
in $\mathbb{Z}_{+}$ distributed according to $m_r$.
Then,  $|A(X)|$ is stochastically dominated
by   $y_{1}+y_{2}+\cdots+y_{|A|}$.

\end{lem}

\begin{proof}
Take an arbitrary enumeration $A=\{e_{1},\,e_{2},\,\ldots,\,e_{|A|}\}$ and define disjoint
sets $G_{1},\,G_{2},\,\ldots,\,G_{|A|}$ as $G_{1}= \overline\conn(e_1;X)$ and
\[
G_{k}= \overline\conn(e_k;X)\setminus\bigg[\,\bigcup_{i=1}^{k-1} \overline\conn(e_i;X)),\bigg]\;\;;\;k\in\llbracket2,\,|A|\rrbracket\;.
\]
Then, the set $A(X)$ can be represented as the disjoint union of
$G_{1},\,G_{2},\,\dots,\,G_{|A|}$ and thus
\[
|A(X)|=\sum_{i=1}^{|A|}|G_{i}|\;.
\]
We now claim that
\[
\sum_{i=1}^{k}|G_{i}|\preceq\sum_{i=1}^{k}y_{i}\;\;\text{for all }k\in\llbracket1,\,|A|\rrbracket\;.
\]
Clearly this is true for $k=1$. To finish the proof by the
induction, it suffices to prove that the distribution of $|G_{i+1}|$
given $G_{1},\,\dots,\,G_{i},$ is stochastically dominated by $y_{i+1}$.
This follows by the spatial independence of bond percolation.
More precisely, given $G_{1},\,\dots,\,G_{i}$, the edge configuration
on $(G_{1}\cup\cdots\cup G_{i})^{c}$ is a Bernoulli percolation with the same parameter, and thus the distribution
of $G_{i+1}$ is dominated by that of $ \overline\conn(e_{i+1};X)$. By \eqref{upsi} and the fact that $X$ is dominated by $\per(r)$,
the size of the latter is dominated by the distribution $m_{r}(\cdot)$ and
 the proof   is completed.
\end{proof}
Recalling \eqref{wj}, let
\begin{equation*}
a_{i}:=|W_{i}|\;;\;i\in\llbracket0,\,m-1\rrbracket\;,
\end{equation*}
and let
\begin{equation}
p_{1}:=2p^{1/2}\;.\label{p1}
\end{equation}
A direct application of Lemma \ref{lem412} is the following bound
on $a_{m-1}$ which along with the fact that  $\mathbb{P}(y_i = 0)= 1-o_p (1)$ (by  \eqref{eqp}), shows that the information percolation history diagram exhibits a contraction from $W_m = A$ to $W_{m-1}$ similar to a subcritical branching process.
\begin{lem}
\label{lem14}Suppose that $p_{1}<p_{\textrm{perc}}(d)$ and let $(y_{i})_{i=1}^{\infty}$
be a sequence of i.i.d. random variables with distribution $m_{2p_{1}}$.
Then, we have
\[
a_{m-1}\preceq y_{1}+\cdots+y_{|A|}\;.
\]
\end{lem}
\begin{proof}
We apply  Lemma \ref{lem412} with $X=\env_{m-1}\cup\env_{m-2}$.
By Proposition \ref{pbd} and union bound, the distribution of $X$
is stochastically dominated by $\per(2p_{1})$.
{{
We now claim that in the construction of the evolution of the history diagram of an edge $e$ over the time interval $[\tau_{m-1}, \tau_m]$,   $e$ gets expanded to a subset of $\overline\conn(e;X)$. To see this, we first note that there are only three possible cases: the oblivious update, the non-oblivious update, and the non-update and the corresponding expansions being the empty set, $\overline\conn(e;X)$, and $e$, respectively. Thus, to prove the claim, it suffices to just check the non-update case for which the expansion set is merely $e$. However in this case, the claim is verified by observing that $e$ is open in $\nup_{m-1}\preceq \Xi_{m-1}\preceq X$ where all the inclusions are by definition.
}}
Hence, we can conclude that $W_{m-1}\subset A(X)$. The assertion of the lemma is now immediate from Lemma \ref{lem412}.

\end{proof}

We now state the main result regarding the domination by branching process. For $i\in\llbracket1,\,m\rrbracket$, denote by $\mathcal{F}_{i}$
the $\sigma$-algebra on $\Omega_{n}$ generated by update sequence
$\upd[\tau_{i},\,\tau_{m}]$ (Hence $\mathcal{F}_{m}=\{\emptyset,\,\Omega_{n}\}$).
\begin{prop}
\label{lem413}Suppose that $p$ is small enough so that $3p_{1}<p_{\textrm{perc}}(d)$.
Let $(y_{i})_{i=1}^{\infty}$ be a sequence of i.i.d. random variables
with distribution $m_{3p_{1}}$ defined in \eqref{upsi}. For all $i\in\llbracket1,\,m-2\rrbracket$,
the distribution of $a_{i}$ given
$(\mathcal{F}_{i+2},\,W_{i+2})$
is stochastically dominated by
\[
y_{1}+y_{2}+\cdots+y_{a_{i+2}}\;.
\]
\end{prop}
One might expect that the proof of Proposition \ref{lem413} can be carried out similarly as that of Lemma \ref{lem14}. However this does not work, roughly because of the following:
Assume that we condition on $W_{i+1}$ and try to control $a_{i}=|W_{i}|$.
Then, $W_{i}$ is determined by $W_{i+1}$, the environment $\env_{i}\cup\env_{i-1}$,
and the update sequence in $[\tau_{i},\,\tau_{i+1}]$. However, by the same reasoning, $W_{i+1}$ is determined from $W_{i+2}$,
$\env_{i+1}\cup\env_{i}$ and the update sequence in $[\tau_{i+1},\,\tau_{i+2}]$,
and thus $W_{i+1}$ already contains some information on $\env_{i}$.
Therefore, the distribution of $\env_{i}$ given $W_{i}$ is hard to analyze. In particular, in the \textit{worst case}, if all the edges in $W_{i+1}$
belong to $\env_{i}$, one cannot expect a contraction estimate of $a_{i}$ in terms of $a_{i+1}$ described in the previous lemma.
However, at this point one notices that $\env_{i-1}$ and $\env_{i-2}$ are still independent of $W_{i+1}$ and hence one can possibly obtain a bound for $a_{i-1}$ instead. In other words, if we conditioned on $W_{i+2}$ and all the relevant
information prior to it, the distribution of $a_{i}$, instead of
$a_{i+1}$, can be dominated in an appropriate manner.

This is done through the next result whose proof crucially uses the definitions listed in Table \ref{chart}.
Recall the notations $C_{j}$ and $N_{j}$ from \eqref{epi1}.

\begin{prop}
\label{prop412}For $i\in\llbracket1,\,m-2\rrbracket$, define $\theta_{i}\in\Omega_{n}$
as following:
\[
\theta_{i}(e)=\begin{cases}
(\env_{i-1}\cup\env_{i})(e) & \text{if }e\in C_{i+2}\;,\\
(\env_{i-1}\cup\env_{i}\cup\env_{i+1})(e) & \text{if }e\in E\setminus C_{i+2}\;.
\end{cases}
\]
Define
\begin{equation}
Z_{i}=\bigcup_{e\in W_{i+2}}\overline{\conn}(e;\theta_{i})\;. \label{zia}
\end{equation}
Then, it holds that $W_{i}\subset Z_{i}$.
\end{prop}

The proof of this proposition is based on  two geometric lemmas (Lemmas \ref{lem421} and \ref{lem422}). {We refer to Figure \ref{fig51} for the
illustration of the proofs of these two lemmas and Proposition \ref{prop412}.} However before proving the latter we first finish the proof of Proposition \ref{lem413}.

\begin{figure}[h]
\centering
\includegraphics[scale=.17]{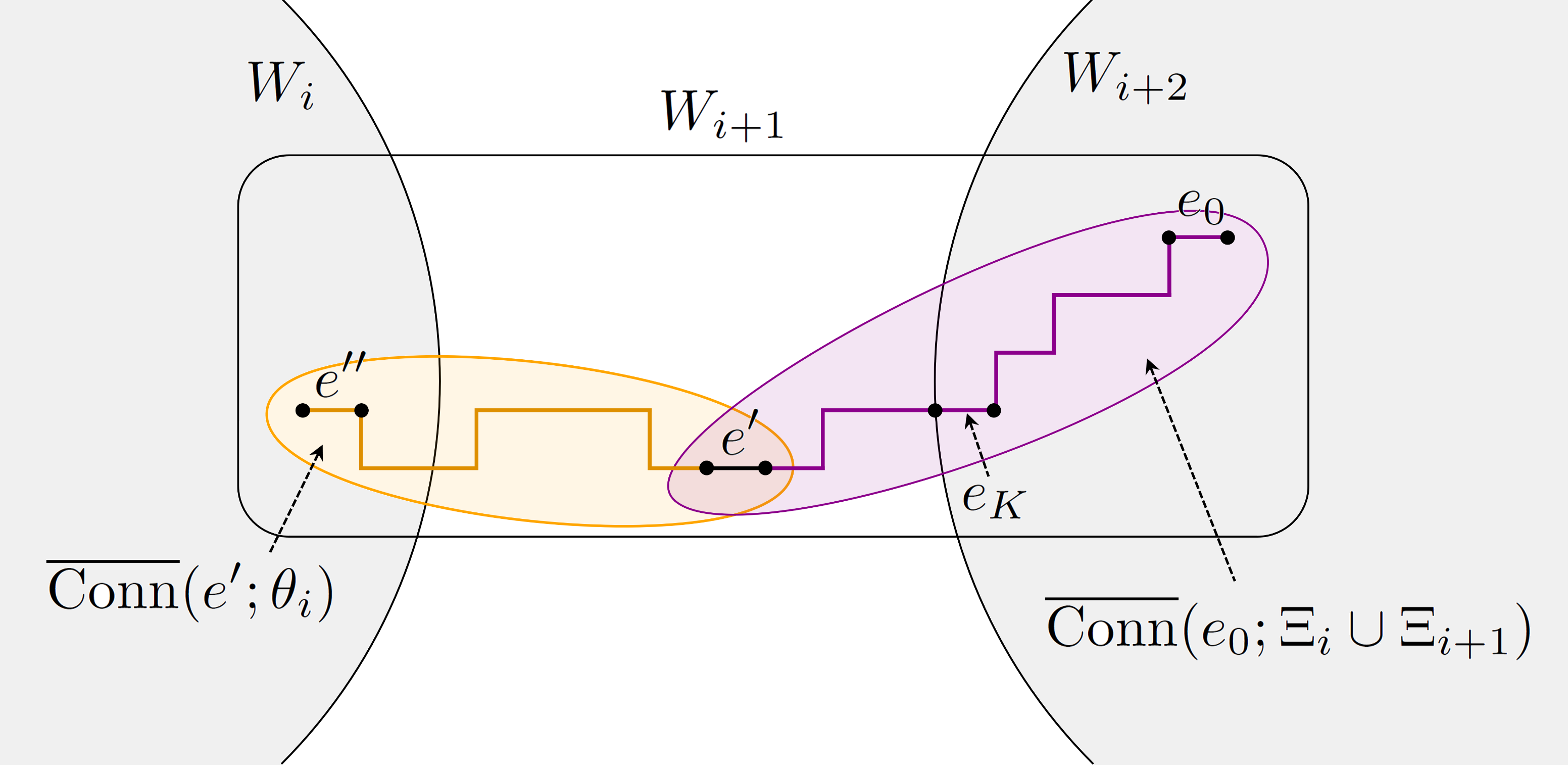}
\caption{{Figure illustrating the proof of Lemmas \ref{lem421}, \ref{lem422}, and Proposition \ref{prop412}. The crucial fact is that $e'\in \conn(e_0;\theta_i)$ (without bar). }}
\label{fig51}
\end{figure}

Note that $W_{i}$ is  determined by $\upd[\tau_{i-1},\,\tau_{m}]$,
and hence $a_{i}$ is a random variable measurable with respect to
$\mathcal{F}_{i-1}$.

\begin{proof}[Proof of Proposition \ref{lem413}]
We consider the following configuration
\[
\env_{i+1}^{o}(e)=\begin{cases}
0 & \text{if }e\in C_{i+2}\;,\\
\env_{i+1}(e) & \text{if }e\in E\setminus C_{i+2}\;.
\end{cases}
\]
We first make the following claim.

\smallskip
\noindent\textbf{Claim. }\textit{Given $(\mathcal{F}_{i+2},\,W_{i+2}),$ the
distribution of $\env_{i+1}^{o}$ is dominated by $\per(p_{1})$. }
\smallskip

Assuming this claim, since $\theta_{i}=\env_{i-1}\cup\env_{i}\cup\env_{i+1}^{o}$, by Proposition \ref{penv},
it follows that the distribution of $\theta_{i}$ given $(\mathcal{F}_{i+2},\,W_{i+2})$
is stochastically dominated by $\per(3p_{1})$. Hence, by  Lemma \ref{lem412} and the definition \eqref{zia} of $Z_i$, we can conclude that  $|Z_i|$ is stochastically bounded above by  $y_1 + \cdots + y_{a_{i+2}}$. Thus we are done by   Proposition \ref{prop412}.

{{
It remains to prove the claim. We start by noting that $C_{i+2}$ is not a deterministic function of $\mathcal{F}_{i+2}$ and $W_{i+2}.$  However, by \eqref{ne01}, $W_{i+3}^{\rm{NU}}$ and $W_{i+3}^{\textrm{NOb}}$  which are subsets of $W_{i+2}$ are indeed measurable with respect to $\mathcal{F}_{i+2}.$  Next recalling how $C_{i+2}$ is constructed from \eqref{epi1}, note that given $\mathcal{F}_{i+2}$ and $W_{i+2}$, by standard exploration of $\overline{\conn}(e;\env_{i+1}\cup\env_{i+2})$ for $e\in W_{i+3}^{\textrm{NOb}}$, further conditioning on
 $C_{i+2}$, does not affect the distribution of the updates  in
 \begin{equation}\label{noiseind}
 \bigcup_{e\in E\setminus C_{i+2}}\upd[\tau_{i+1},\,\tau_{i+2}](e),\end{equation}
(note that here we are crucially using the fact that $\overline{\conn}(e;\env_{j-1}\cup\env_{j})$ includes the closed boundary edges since otherwise conditioning on $C_{i+2}$ would yield information about its boundary edges which would then have been members of $E\setminus C_{i+2}$).
Thus, from now we assume that $C_{i+2}$ is given, and suppose that $e\notin C_{i+2}$.
Since the configuration  $\env_{i+1}(E\setminus C_{i+2})$ is determined by the updates in \eqref{noiseind}, we can conclude that  the distribution of
$\env_{i+1}(E\setminus C_{i+2})$ given \textit{$(\mathcal{F}_{i+2},\,W_{i+2}, C_{i+2})$}
is stochastically bounded by percolation on $E\setminus C_{i+2}$
with open probability $p_{1}$, by Proposition \ref{pbd} and the
definition of $p_{1}$ in \eqref{p1}. Since $\env_{i+1}^{o}(e)=0$
for $e\in C_{i+2}$, the claim holds conditionally on \textit{$(\mathcal{F}_{i+2},\,W_{i+2}, C_{i+2})$} and hence by averaging over $C_{i+2}$, conditionally on  \textit{$(\mathcal{F}_{i+2},\,W_{i+2})$}.}}
\end{proof}

\begin{lem}
\label{lem421}For $i\in\llbracket1,\,m-1\rrbracket$, we have that
\begin{equation}
W_{i}\subset\bigcup_{e\in W_{i+1}}\overline{\conn}(e;\env_{i-1}\cup\env_{i})\;.\label{eev}
\end{equation}
\end{lem}

\begin{proof}
{{
In view of the definition \eqref{epi1}, the decomposition \eqref{decw2}, and the fact that $W_{i+1}^{\textrm{NU}}\subset W_{i+1}$, it suffice to check that  }}
\begin{equation}
N_{i}\subset\bigcup_{e\in W_{i+1}^{\textrm{NU}}}\overline{\conn}(e;\env_{i-1}\cup\env_{i})\;.\label{eev2}
\end{equation}
If $e\in W_{i+1}^{\textrm{NU}}$, we have $\nup_{i}(e)=1$ by the
definition of $\nup_{i}$, and thus $(\env_{i}\cup\env_{i-1})(e)=1$
since $\nup_{i}\le\env_{i}$. Therefore, we have $e\in\overline{\conn}(e;\env_{i}\cup\env_{i-1})$.
Hence, the right-hand side of \eqref{eev2} contains $W_{i+1}^{\textrm{NU}}$,
and hence contains $N_{i}$ by the second inclusion of Lemma \ref{lem43}.
\end{proof}

\begin{lem}
\label{lem422}For $i\in\llbracket1,\,m-2 \rrbracket$, define $\xi_{i}\in\Omega_{n}$
as follows:
\[
\xi_{i}(e)=\begin{cases}
\env_{i}(e) & \text{if }e\in C_{i+2}\;,\\
(\env_{i}\cup\env_{i+1})(e) & \text{if }e\in E\setminus C_{i+2}\;.
\end{cases}
\]
Then, we have
\[
W_{i+1}\setminus W_{i+2} \subset \bigcup_{e\in W_{i+2}}\overline{\conn}(e;\xi_{i}) \;.
\]
\end{lem}

\begin{proof}
For $e'\in W_{i+1}\setminus W_{i+2}$, we know from Lemma \ref{lem421}
that there exists $e_0 \in W_{i+2}$ and a path
\[
e_{0},\,e_{1},\,\cdots,\,e_{k}(=e')
\]
in $E$ such that $(\env_{i}\cup\env_{i+1})(e_{l})=1$ for all $l\in\llbracket0,\,k-1\rrbracket$.
If none of $e_{0},\,e_{1},\,\cdots,\,e_{k}$ belongs to $C_{i+2}$
then the assertion of lemma is immediate since $\Xi_i \cup \Xi_{i+1} = \xi_i$ along this path. Otherwise, let
\[
K=\max\{h:e_{h}\in C_{i+2}\}\;.
\]
Since $e'\notin W_{i+2}$, we have $K<k$. Then, since $e_{K+1}\notin C_{i+2}$,
we have $e_{K}\in\partial^{-}C_{i+2}$ and thus $(\env_{i+1}\cup\env_{i+2})(e_{K})=0$
by Lemma \ref{lem44}. Since $(\env_{i}\cup\env_{i+1})(e_{K})=1$,
we can conclude that $\env_{i}(e_{K})=1$. This implies that $e'=e_{k}\in\overline{\conn}(e_{K};\xi_{i})$,
where $e_{K}\in C_{i+2}\subset W_{i+2}$. This completes the proof.
\end{proof}
Now we are ready to prove Proposition \ref{prop412}.
\begin{proof}[Proof of Proposition \ref{prop412}]
Fix arbitrary $e''\in W_{i}$. It suffices to verify that $e''\in Z_{i}$.
Since $\env_{i-1}\cup \env_{i}\le\theta_{i}$, by Lemma \ref{lem421}, there exists $e'\in W_{i+1}$
such that
\begin{equation}
e''\in\overline{\conn}(e';\env_{i-1}\cup \env_{i})\subset \overline{\conn}(e';\theta_i)\;.\label{ii1}
\end{equation}
If $e'\in W_{i+2}$, we can immediately assert that $e'\in Z_{i}$ by the definition of $Z_i$.

On the other hand, if $e'\in W_{i+1}\setminus W_{i+2}$, then by
Lemma \ref{lem422} and by the fact that $\xi_{i}\le\theta_{i}$, there exists $e_0 \in W_{i+2}$ such that
\begin{equation}
e'\in\overline{\conn}(e_{0};\xi_{i}) \subset \overline{\conn}(e_{0};\theta_{i})\;.\label{ii2}
\end{equation}
We remark that \eqref{ii1} implies that $\theta_{i}(e')=1$ since otherwise $\overline{\conn}(e';\theta_{i})=\emptyset$.
Therefore we can replace $e'\in\overline{\conn}(e_{0};\theta_{i})$
in \eqref{ii2} with $e'\in\conn(e_{0};\theta_{i})$. Combining
this with \eqref{ii1} ensures that $e''\in\overline{\conn}(e_{0},\theta_{i})\subset Z_{i}$.
This completes the proof.
\end{proof}

\subsubsection{Bounds on $a_i$ based on domination by branching processes}
Now we present two consequences of the previous branching process
type estimate. These will play a fundamental role in the proof of Proposition \ref{p413}.

 For a random variable $y$ in $\mathbb{Z}_{+}$
following the law $m_{3p_{1}}$ defined in \eqref{upsi}, we define $M=M(p)$ as the solution
of
\[
e^{-2M}=\mathbb{E}(y)\;.
\]
It readily follows that
\begin{equation}\label{mm}
\lim_{p\rightarrow{0}}M=\infty\;.
\end{equation}

\begin{lem}
\label{lem414}For $k\in\llbracket1,\,m\rrbracket$, select ${\mathfrak{r}}\in\{1,\,2\}$
so that $(k-\mathfrak{r})\,{\rm{ mod }}\,\,2=0$. Then, for some constant
$C=C(p)>0$, it holds that
\[
\mathbb{E}\left[a_{\mathfrak{r}}|a_{k}\right]\le Ce^{-Mk}a_{k}\;.
\]
\end{lem}
\begin{proof}
It follows from Proposition \ref{lem413} that, for all $k\in\llbracket3,\,m\rrbracket$,
\begin{equation}
\mathbb{E}\left[a_{k-2}|a_{k},\,\mathcal{F}_{k}\right]\le e^{-2M}a_{k}\;.\label{e414}
\end{equation}
Then, the proof of lemma is completed by the induction.
\end{proof}
{
\begin{lem}
\label{lem415}For all sufficiently small $p$,
there exists $c_{0}=c_{0}(p)>0$ such that,
\[
\mathbb{E}\exp\bigg\{ c_{0}\sum_{i=1}^{m-1}a_{i}\bigg\} \le e^{|A|}\;.
\]
Furthermore, $\lim_{p\rightarrow 0}c_{0}(p)=+\infty$.
\end{lem}
}
\begin{proof}
By the Cauchy-Schwarz inequality,
\begin{equation}
\mathbb{E}\bigg[\exp\bigg\{ c\sum_{i=1}^{m-1}a_{i}\bigg\} \bigg]^{2}\le\mathbb{E}\exp\bigg\{ 2c\sum_{i:2i\in\llbracket 1,\,m-1\rrbracket}a_{2i}\bigg\} \cdot \mathbb{E}\exp\bigg\{ 2c\sum_{i:2i+1\in\llbracket 1,\,m-1\rrbracket}a_{2i+1}\bigg\} \;.\label{e415}
\end{equation}

Denote by $y$   the random variable with distribution $m_{3p_1}$ defined in \eqref{upsi}. Note
that the following equation  on $c$
\begin{equation}
\mathbb{E}e^{(2c+1)y}=e \label{eqp2}
\end{equation}
has a positive solution $c_0 = c_{0}(p)$ and we can readily check that $\lim_{p\rightarrow0}c_{0}=+\infty$.
Now it suffices to prove that, for
all $\ell$,
\begin{equation}
\mathbb{E}\exp\bigg\{ 2c_0\sum_{i=1}^{\ell}a_{m-2i}\bigg\} \le e^{|A|}\;\;\;\text{and\;\;\;}\mathbb{E}\exp\bigg\{ 2c_0\sum_{i=0}^{\ell}a_{m-2i-1}\bigg\} \le e^{|A|}\;.\label{e440}
\end{equation}
By Proposition \ref{lem413} and \eqref{eqp2}, for all $i\in\llbracket3,\,m\rrbracket$,
we have
\[
\mathbb{E}\left[\left.e^{(2c_0+1)a_{i-2}}\right|a_{i},\,\mathcal{F}_{i}\right]
\le\mathbb{E}\left[e^{(2c_0+1)y}\right]^{a_{i}}\le e^{a_{i}}\;.
\]
Consequently, for all $\ell\ge1$,
\[
\mathbb{E}\bigg[e^{a_{m-2\ell}}\cdot\exp\bigg\{ 2c_0\sum_{i=1}^{\ell}a_{m-2i}\bigg\} \bigg|a_{m-2\ell+2},\,\mathcal{F}_{m-2\ell+2}\bigg]\le e^{a_{m-2\ell+2}}\exp\bigg\{ 2c_0\sum_{i=1}^{\ell-1}a_{m-2i}\bigg\}\;.
\]
Repeating this procedure, we obtain
\begin{align}
\mathbb{E}\exp\bigg\{ 2c_0\sum_{i=1}^{\ell}a_{m-2i}\bigg\}  & \le\mathbb{E}\bigg[e^{a_{2\ell}}\exp\bigg\{ 2c_0\sum_{i=1}^{\ell}a_{m-2i}\bigg\} \bigg]\nonumber \\
 & \le\mathbb{E}\bigg[e^{a_{2\ell+2}}\exp\bigg\{ 2c_0\sum_{i=1}^{\ell-1}a_{m-2i}\bigg\} \bigg]\label{e441}\\
 & \le\cdots\le\mathbb{E}\bigg[e^{(2c_0+1)a_{m-2}}\bigg]\le e^{a_{m}}=e^{|A|}\;.\nonumber
\end{align}
This proves the first inequality in \eqref{e440}.
By a similar argument as above, one can show that
\begin{equation}
\mathbb{E} \exp\bigg\{ 2c_0\sum_{i=0}^{\ell}a_{m-2i-1}\bigg\} \le\mathbb{E} e^{(2c_0+1)a_{m-1}} \;.\label{e442}
\end{equation}
By Lemma \ref{lem14} and the fact that $m_{2p_1}$ is dominated by $m_{3p_1}$, we have
\begin{equation}
\mathbb{E}e^{(2c_0+1)a_{m-1}}
\le\mathbb{E} \big[\,e^{(2c_0+1)y} \,\big]^{|A|}
\le e^{|A|}\;,\label{e443}
\end{equation}
where the last inequality follows from \eqref{eqp2}. Now, \eqref{e440} is proven by combining \eqref{e441}, \eqref{e442}, and \eqref{e443}.
\end{proof}

We now proceed to proving Proposition \ref{p413}.
\subsubsection{Proof of Proposition \ref{p413}.}
Recall that we want to bound $\mathbb{P}\left[A\in\mathcal{C}_{\mathcal{R}(A)}^{*}\right].$

We start by defining
\[
\sigma:=\begin{cases}
\max\left\{ i:a_{i}=1\right\}  & \text{if }a_{i}=1\:\text{for some }i\in\llbracket1,\,m\rrbracket,\\
0 & \text{otherwise.}
\end{cases}
\]
{We further define events $\mathcal{A}$ and $\mathcal{B}$ as
\begin{align*}
\mathcal{A}= & \Big(\{\sigma>0\}\cap \{\mathscr{H}_{A}\text{ merges to one point in }[\tau_{\sigma},\,\tau_{m}]\}\Big)\bigcup\Big(\{\sigma=0\}\cap\{A\in\mathcal{C}_{\mathcal{R}(A)}^{*}\Big),\\
\mathcal{B}= & \Big(\{\sigma>0\}\cap\{\text{The history diagram starting from }W_{\sigma}\text{ at }t=\tau_{\sigma}\text{ survives until }t=\tau_{1}\}\Big)\\
 & \bigcup\{\sigma=0\}.
\end{align*}
Note that in the definition of $\mathcal{A}$, on the event $\{\sigma>0\},$ we put the additional constraint that all the history diagrams in $\mathcal{H}_{A}$  merge to a point in $[\tau_{\sigma}, \tau_m].$ Note that this is not guaranteed just by assuming $\sigma > 0,$ since  it may happen that all the history diagrams have been killed except for one edge which survives on its own up to $\tau_{\sigma}.$  
}
Clearly, the event $\{A\in\mathcal{C}_{\mathcal{R}(A)}^{*}\}$ is a subset
of $\mathcal{A}\cap\mathcal{B}$ since if $\sigma>0,$ the only way $\{A\in\mathcal{C}_{\mathcal{R}(A)}^{*}\}$ can occur is if $\mathcal{A}\cap \mathcal{B}$ occurs since otherwise $\mathcal{C}_{\mathcal{R}(A)}$ has multiple connected components.
Thus we have
\begin{equation}
\mathbb{P}\left[A\in\mathcal{C}_{\mathcal{R}(A)}^{*}\right]
\le\mathbb{P}\left[\mathcal{A}\cap\mathcal{B}\right]\;.\label{e450}
\end{equation}
We next make and prove the following claim.

\smallskip
\smallskip
\smallskip
\noindent\textbf{Claim.}\textit{ Conditioned on the event $\{\sigma=k\}$ with $k \ge0$, two events $\mathcal{A}$ and $\mathcal{B}$
are independent.}
\smallskip

\begin{proof}
For $k=0$, the claim is immediate from the definitions of $\mathcal{A}$ and $\mathcal{B}$. For $k\ge 1$, let us write $W_{k}=\{e\}$. Then, conditioned on the event $\{\sigma=k\}$,  it suffices prove that the event $\mathcal{A}$ is independent
of $\upd[0,\,\tau_{k}]$ since $\mathcal{B}$ depends only on
$\upd[0,\,\tau_{k}]$. Clearly the behavior of {the history diagram starting from }$e$ in $(\tau_{k+1},\,\tau_{m}]$ is independent
of $\upd[0,\,\tau_{k}]$. Hence, it only suffices to check the
interval $(\tau_{k},\,\tau_{k+1}]$. The event $\mathcal{A}$
imposes that all the edges in $W_{k+1}\setminus\{e\}$ exhibit
the oblivious update in $(\tau_{k},\,\tau_{k+1}]$, while
$e\in W_{k+1}$  survives to $\tau_{k}$ without expanding
to $\overline{\conn}(e;\env_{k-1}\cup\env_{k}$) which by definition includes $e$ as well as its adjacent edges. The first
event is determined by $\upd[\tau_{k},\,\tau_{k+1}]$ and
hence is independent of $\upd[0,\,\tau_{k}]$. The second event
occurs only when there is no update at $e$ in $\upd[\tau_{k},\,\tau_{k+1}]$,
and hence this event is also independent of $\upd[0,\,\tau_{k}]$
as well. This completes the proof.
\end{proof}
By this claim and \eqref{e450}, we deduce that
\begin{equation}
\mathbb{P}\left[A\in\mathcal{C}_{\mathcal{R}(A)}^{*}\right]
\le
\mathbb{E}\,\Big[\,
\mathbb{P}\left[ \mathcal{A}\,|\,\sigma \right] \, \mathbb{P}\left[ \mathcal{B}\,|\,\sigma \right]
\,\Big]
\;.\label{epq1}
\end{equation}
{{
We now  claim that there exists $C,\,M>0$ such that for all $k\ge 0,$
\begin{equation}
\mathbb{P}\left[ \mathcal{B}\,|\,\sigma =k  \right]\le  Ce^{-Mk}\;.\label{epq2}
\end{equation}
Since this bound trivially holds for $k=0$ if we take $C>1$, it suffice to consider the case $k\ge 1$. Select ${\mathfrak{r}}\in\{1,\,2\}$
so that $(\sigma-\mathfrak{r})\,{\rm{ mod }}\,\,2=0$.
 Then, by Lemma \ref{lem414}, we get
 \begin{equation*}
\mathbb{P}\left[ \mathcal{B}\,|\,\sigma=k \right]\le\mathbb{P}\left[\,a_{\mathfrak{r}}>0\,|\,a_{k}=1\,\right]
\le\mathbb{E}\left[\,a_{\mathfrak{r}}\,|\,a_{k}=1\,\right]\le Ce^{-Mk}\;.
\end{equation*}
}}This proves the bound \eqref{epq2}. Now by \eqref{epq1} and \eqref{epq2}, we have
\begin{equation}
\mathbb{P}\left[A\in\mathcal{C}_{\mathcal{R}(A)}^{*}\right]
\le
\mathbb{E}\Big[Ce^{-M\sigma}\mathbf{1}_{\mathcal{A}}\Big]\;.\label{epp1}
\end{equation}
If $|A|=1$ so that $\sigma=m$, this inequality proves the assertion
of the proposition. Now we assume that $|A|\ge2,$ so that $\sigma\le m-2$.
Note that $\sigma$ cannot be $m-1$ since $W_{m-1}\supset A$ under
$\mathcal{A}$ {since otherwise for $e\in A\setminus W_{m-1}$, the set $\mathscr{H}_e$ is a singleton and hence $\mathscr{H}_e \cap \mathscr{H}_{A\setminus\{e\}} =\emptyset$.}
 Thus conditioned on the event $\{\sigma=k\}$ {{with $k\ge 1$}}, the event
$\mathcal{A}$ implies that
\[
\left\{ a_{k+1}+\cdots+a_{m-1}\ge|\conn(A)|-1\right\}\;,
\]
{where, recall that $|\conn(A)|$ is
the smallest possible number of edges of the connected subgraph of
$(\Lambda_{n},\,E_{n})$ containing $A$.} We can neglect $a_{m}$ since $W_{m}=A\subset W_{m-1}$ under $\mathcal{A}$,
and we used the fact that $a_{k}=1$.
{{
On the other hand, for $k=0$, the event $\mathcal{A}$ implies that
\[
\left\{ a_{1}+\cdots+a_{m-1}\ge|\conn(A)|\right\}\;.
\]
Recall from Definition $\ref{def44},$ that our definition of the information percolation clusters only extended up to $\tau_1$ and not $\tau_0$. This is reflected in the fact that the above sum starts from $a_1$ instead of $a_0$.
}} Therefore, we can bound the
right-hand side of \eqref{epp1} from above by
\begin{align*}
&  \sum_{k=1}^{m-2}Ce^{-Mk}\,\mathbb{E}\Big[\mathbf{1}\left\{ a_{k+1}+\cdots+a_{m-1}
   \ge |\conn(A)|-1\right\} \mathbf{1}\{\sigma=k\}\Big]\\
& \qquad    +C\, \mathbb{E}\Big[\mathbf{1}\left\{ a_{1}+\cdots+a_{m-1}\ge|\conn(A)|\right\} \,\mathbf{1}\{\sigma=0\}\Big]\;.
\end{align*}
By applying $\mathbf{1}\{\sigma=k\}\le\mathbf{1}\left\{ a_{k+1},\,\cdots,a_{m-1}\ge2\right\} $
here, we obtain
\begin{equation}
\mathbb{E}\Big[Ce^{-M\sigma}\mathbf{1}_{\mathcal{A}}\Big]\le C\sum_{k=0}^{m+1}H_{k}\;,\label{epp0}
\end{equation}
where
\[
H_{k}=e^{-Mk}\,\mathbb{E}\Big[\mathbf{1}\left\{ a_{k+1}+\cdots+a_{m-1}\ge|\conn(A)|-1\right\} \, \mathbf{1}\left\{ a_{k+1},\,\cdots,a_{m-1}\ge2\right\} \Big]\;.
\]
For any $C_{1},\,C_{2}>0$, by the Chebyshev inequality,
\[
H_{k}\le e^{-Mk}e^{-C_{1}(|\conn(A)|-1)}e^{-2C_{2}(m-k)}
\,\mathbb{E} e^{(C_{1}+C_{2})(a_{k+1}+\cdots+a_{m-1})} \;.
\]
Now we take $C_{1}$ and $C_{2}$ such that $C_{1}+C_{2}<c_{0}$ where
the constant $c_{0}$ is the one appeared in Lemma \ref{lem415}.
Then, by Lemma \ref{lem415} and the fact that $|A|\le|\conn(A)|$,
we can further obtain
\begin{align}
 & H_{k}\le e^{-Mk}e^{-C_{1}(|\conn(A)|-1)}e^{-2C_{2}(m-k)}e^{|\conn(A)|}\;.\label{epp4}
\end{align}
For given $\theta>0$, we first take $p$ small enough so that $c_{0}>\theta+2$
and $M>1$. This is possible since
\[
\lim_{p\rightarrow0}c_{0}=\lim_{p\rightarrow0}M=+\infty\;
\]
by \eqref{mm} and Lemma \ref{lem415}.
Take $C_{1}=\theta+1$ and $C_{2}=1/2$. With this selection, the
bound \eqref{epp4} becomes
\[
H_{k}\le e^{\theta+1}e^{-(m+\theta|\conn(A)|)}e^{-(M-1)k}\;.
\]
Combining this with \eqref{epp0} yields
\[
\mathbb{E}\left[Ce^{-M\sigma}\mathbf{1}_{\mathcal{A}}\right]\le C(\theta)e^{-(m+\theta|\conn(A)|)}\;.
\]
Thus the statement of the proposition follows by recalling that  $\tau_{m}=m\Delta$. \qed

\section{Reduction to a product chain}\label{reduction1}
From now on, we define
\[
r=r(n)=3\log^{5}n\;.
\] Moreover recall $\smax$ from \eqref{esm}.
Denote by $(X_{t}^{\dagger})_{t\ge0}$ the FK-dynamics defined on
the periodic lattice $\mathbb{Z}_{r}^{d}.$ Let $\Omega_{r}=\{0,\,1\}^{E_{r}}$ where $E_{r}=E(\mathbb{Z}_{r}^{d})$,
and denote by $\pi^{\dagger}:=\mu_{p,q}^{r}$ the random-cluster measure
on $\Omega_{r}=\{0,\,1\}^{E_{r}}$. Let $\Lambda\subset E_{r}$ be
a {box} of size $2\log^{5}n$. Then, define
\begin{equation}
\mathbf{d}_{t}=\mathbf{d}_{t,n}=\max_{x_{0}^{\dagger}\in\Omega_{r}}\left\Vert \mathbb{P}_{x_{0}}\left[X_{t}^{\dagger}(\Lambda)\in\cdot\,\right]-\pi_{\Lambda}^{\dagger}\right\Vert _{L^{2}(\pi_{\Lambda}^{\dagger})}\;,\label{emt}
\end{equation}
where $X_{t}^{\dagger}(\Lambda)$ represents the configuration of
$X_{t}^{\dagger}$ on $\Lambda$, and $\pi_{\Lambda}^{\dagger}$ stands
for the projection of $\pi^{\dagger}$ onto the set $\Lambda$. The
main result of this section is the following theorem.
\begin{thm}
\label{t61}For all sufficiently small $p$, there exists a constant
$C_{1}=C_{1}(p)$ such that the following hold.
\begin{enumerate}
\item For $s\in[C_{1}\log\log n,\,\smax]$ and $t\in[0,\,\smax]$, it holds
that
\[
\max_{\nu:\nu\preceq\per(\tini)}\left\Vert \mathbb{P}_{\nu}\left[X_{t+s}\in\cdot\,\right]-\mui\right\Vert _{\textrm{TV}}\le\frac{1}{2}\left[\exp\left\{ \frac{n^{d}}{\log^{12d}n}\mathbf{d}_{t}^{2}\right\} -1\right]^{1/2}+\frac{4}{n^{2d}}\;.
\]
\item If $t\ge C_{1}\log\log n$ and
\[
\lim_{n\rightarrow\infty}\bigg(\frac{n}{\log^{10}n}\bigg)^{d}\mathbf{d}_{t}^{2}=+\infty\;,
\]
then we have
\[
\liminf_{n\rightarrow\infty}\max_{\nu:\nu\preceq\per(\tini)}\left\Vert \mathbb{P}_{\nu}\left[X_{t}\in\cdot\,\right]-\mui\right\Vert _{\textrm{TV}}=1\;.
\]
\end{enumerate}
\end{thm}
{{
\begin{rem}
\label{rm62} As the proof will reveal, part (1) of the theorem holds even when $X_t$ is the FK-dynamics on $\mathbb{Z}_{m}^{d}$, instead of $\mathbb{Z}_n^d$, where $m\in \llbracket \log^5 n,\,n\rrbracket$. The inequality in this case is 
\[
\max_{\nu:\nu\preceq\textrm{Perc}_m (\tini)}\left\Vert \mathbb{P}_{\nu}\left[X_{t+s}\in\cdot\,\right]-\mu_{p,\,q}^m \right\Vert _{\textrm{TV}}\le\frac{1}{2}\left[\exp\left\{ m^d \mathbf{d}_{t}^{2}\right\} -1\right]^{1/2}+\frac{4}{n^{2d}}\;, 
\]
where $\mathbf{d}_t$ is as in  \eqref{emt}. 
\end{rem}
}}
Henceforth, the constant $C_{1}>0$ will always refer to the constant
appeared in this theorem. The proof of this theorem will be presented
in the remaining part of the current section. We shall assume that
$p$ is small enough so that all the results established in Sections
\ref{sec3} (including \ref{sec5}) and \ref{sec4} are valid.
As indicated in Section \ref{iop}, following the strategy in \cite{LS1} where a similar statement as Theorem \ref{t61} appears,  we will reduce the chain to an approximate product chain.  The only major difference in the statement of Theorem \ref{t61}, as compared to statements appearing in previous articles is that owing to the non-locality of the dynamics, we initialize from a sparse initial condition dominated by a sub-critical percolation which can be obtained by evolving the initial configuration for a burning time (see Lemma \ref{pro35}).

We start by giving a short roadmap of what the various subsections achieve.
\begin{itemize}
\item The first part (Section \ref{sec61}) constructs the so called Barrier dynamics where the FK-dynamics on $(\Z/n\Z)^d$ gets compared to FK-dynamics on a disjoint collection of $(\Z/r\Z)^d$ where $r=\log^5n.$
\item We define the notion of Update support in Section \ref{sec62}.
{
To get an upper bound on the mixing time, we bound the total variation distance at $d(t+s)$ where $t=t_{\mix}=O(\log n)$ and $s=\log\log(n).$ The natural strategy is to couple the configurations at time $t+s$ starting from any two arbitrary initial configurations.
}
At this point the key observation is that irrespective of the configuration at time $t$, all but a sparse set of small boxes couple at time $t+s.$ The remainder is called the `Update support' for reasons which will be clear later and hence the remaining task is to ensure that the time interval $[0,t]$ is sufficient  for the FK-dynamics starting from two arbitrary configurations to couple on the `Update support'.
\item We prove Theorem \ref{t61} in Section \ref{sec63}.
\end{itemize}

\subsection{\label{sec61}Coupling with barrier-dynamics}

Divide $E_{n}$ into disjoint squares of size $\log^{5}n$ as follows.
Let us write $K=n/\log^{5}n$ and assume that $K$ and $\log^{5}n$
are integers for the simplification of notation. Define
\[
V_{n}=\{0,\,\log^{5}n,\,2\log^{5}n,\,\dots,\,(K-1)\log^{5}n\}^{d}\subset\Lambda_{n}\;.
\]
For each $v\in V_{n}$, we define an \textit{\textbf{edge box}} $B_{v}$ by
\begin{equation}\label{edgebox1}
B_{v}=\left\{ (u,\,u+e_{j}):u\in v+\llbracket0,\,\log^{5}n-1\rrbracket^{d}\text{ and }j\in\llbracket1,\,d\rrbracket\right\}\;,
\end{equation}
where $e_{j}$ represents the $j$th standard normal vector in $\mathbb{R}^{d}$.
One can think of $B_{v}$ as a box of size $\log^{5}n$ with some
boundary edges are removed. Note that $(B_{v})_{v\in V_{n}}$ is a
{decomposition} of $E_{n}$. {Furthermore, we mention that all the boxes below of various sizes, are edge boxes and hence for brevity we will refer to them as boxes.}

\begin{figure}[h]
\centering
\includegraphics[scale=.20]{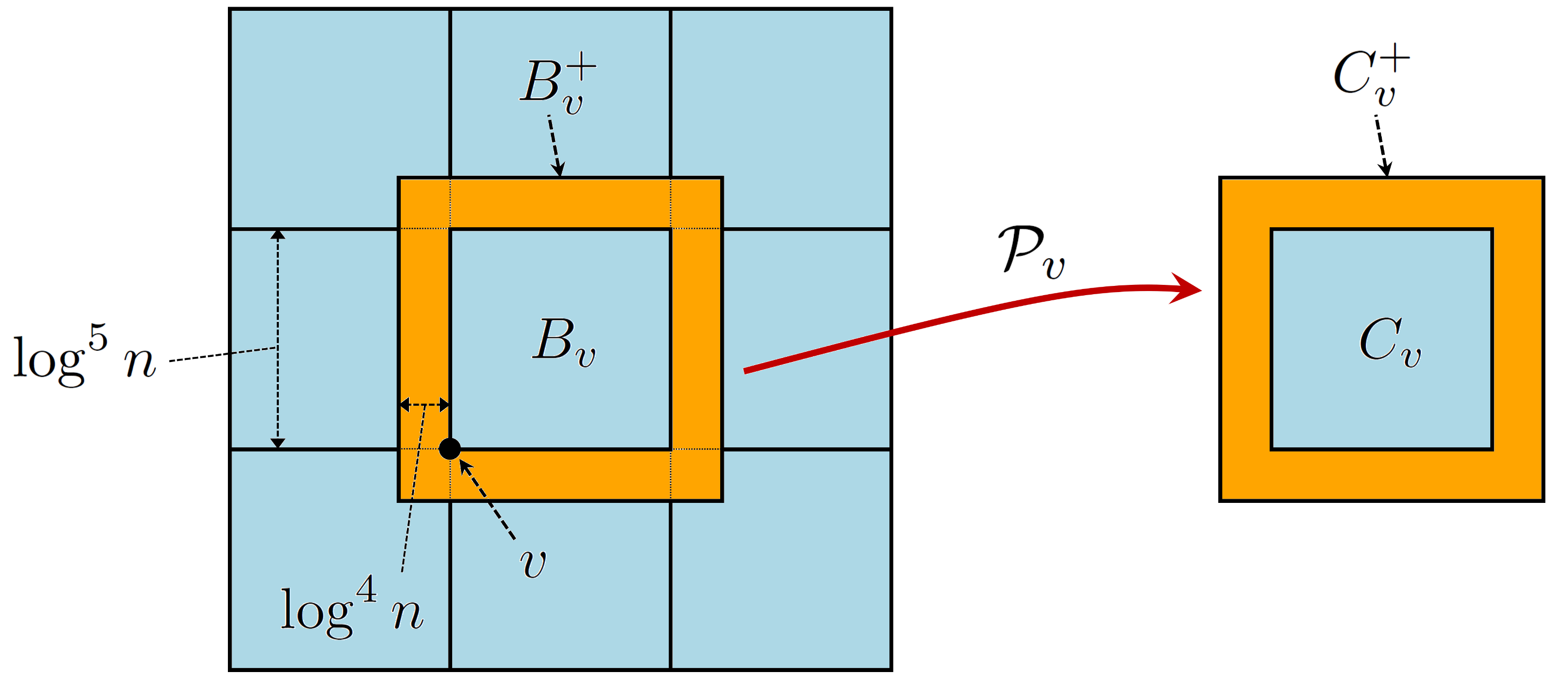}
\caption{Figure illustrating the maps $\mathcal{P}_v$. }
\label{fig3}
\end{figure}

Then, for each $v\in V_{n}$, consider the expanded box $B_{v}^{+}\subset\Lambda_{n}$
of $B_{v}$ in the sense of \eqref{aplus}. Then, $B_{v}^{+}$ is
a box of size $\log^{5}n+2\log^{4}n$ which is concentric
with $B_{v}$. Let $C_{v}^{+}$ be another square lattice of size
$\log^{5}n+2\log^{4}n$ and define a natural identification map $\mathcal{P}_{v}:B_{v}^{+}\rightarrow C_{v}^{+}$.
Define $C_{v}=\mathcal{P}_{v}(B_{v})$ so that $(C_{v},\,C_{v}^{+})$ is
a copy of $(B_{v},\,B_{v}^{+})$ (see Figure \ref{fig3}). We define
\[
\widehat{E}_{n}=\bigsqcup_{v\in V_{n}}C_{v}^{+}\;\;\text{and\;\;}\widehat{\Omega}_{n}=\{0,\,1\}^{\widehat{E}_{n}}\;.
\]
Note that the last union is a disjoint union.
\begin{defn}[Barrier-dynamics]
For each $v\in V_{n}$, the barrier-dynamics is a FK-dynamics $X_{t}^{v}$
on $C_{v}^{+}$ coupled with $X_{t}$ by sharing the same update sequence
via the following rules:
\begin{enumerate}
\item (Initial condition) The initial edge configuration on $C_{v}^{+}$
is identical to that of $B_{v}^{+}$ through $\mathcal{P}_{v}$. In other
words, $X_{0}^{v}(\mathcal{P}_{v}(e))=X_{0}(e)$ for all $e\in B_{v}$.
\item (Dynamics) We define the FK-dynamics $(X_{t}^{v})_{t\ge0}$ on $C_{v}^{+}$
with periodic boundary condition by using the update sequence of
$B_{v}^{+}$. Formally stating, we perform updates for each $e\in C_{v}^{+}$
by using the update sequence $\upd(\mathcal{P}_{v}^{-1}(e))$ of the edge
$\mathcal{P}_{v}^{-1}(e)\in B_{v}^{+}$.
\end{enumerate}
\end{defn}

For $t\ge0$, we define a random map $\mathcal{G}_{t}:\Omega_{n}\rightarrow\Omega_{n}$
such that, for all $X_{0}\in\Omega_{n}$,
\[
[\mathcal{G}_{t}(X_{0})](e)=X_{t}^{v}(\mathcal{P}_{v}(e))\;,
\]
where $v\in V_{n}$ is the unique index such that $e\in B_{v}$. The next lemma now says that the actual dynamics and the barrier dynamics stay coupled for a significant amount of time provided the initial condition is sparse enough (note that for a spin system the latter condition is not needed since each update only depends on its immediate neighbors).
\begin{lem}
\label{lem32}Suppose that $p$ is small enough and the law of the
initial condition $X_{0}$ follows the law $\nu$ such that $\nu\preceq\per(\pini)$.
Then, we have
\[
\mathbb{P}\big[X_{t}=\mathcal{G}_{t}(X_{0})\text{ for all}\;t\in[0,\,\smax]\,\big]\ge1-n^{-2d}\;.
\]

\end{lem}

\begin{proof}
By Lemma \ref{lem0311} (cf. Remark \ref{rm52}), it holds that
\[
\mathbb{P}\big[X_{t}(B_{i})=X_{t}^{v}(C_{i})\text{ for all}\;t\in[0,\,\smax]\,\big]\ge1-n^{-3d}\;.
\]
Thus, the conclusion of the lemma follows from the union bound since
$|V_{n}|<n^{d}.$
\end{proof}

\subsection{\label{sec62}Sparsity of update support}
\begin{defn}[Update support]
For each $s>0$, denote by $\us=\upd[0,\,s]$   the update sequence
between time $[0,\,s]$. Then, the random map $\mathcal{G}_{s}$ is
completely determined by $\us$ and hence we can write $\mathcal{G}_{s}=g_{\us}$
for some function $g_{\us}:\Omega_{n}\rightarrow\Omega_{n}$. The
update support of $\us$ is the minimum subset $\Gamma_{\us}\subset E_{n}$
such that $\mathcal{G}_{s}$ is a function of $X(\Gamma_{\us})$
for all $X\in\Omega_{n}$, i.e.,
\[
g_{\us}(X)=f_{\us}(X(\Gamma_{\us}))
\]
for some $f_{\us}:\{0,\,1\}^{\Gamma_{\us}}\rightarrow\Omega_{n}.$
\end{defn}

\begin{lem}\cite[Lemma 3.8]{LS1}
\label{lem64}Fix $t\ge0$ and let $\us$ represent the update sequence
for the time interval $[t,\,t+s]$ for $s\le\smax$ where $\smax$ was defined in \eqref{esm}. Suppose that
$p$ is small enough and a probability measure $\nu$ in $\Omega_{n}$
satisfies $\nu\preceq\per(\tini)$. Then, we have
\begin{equation*}
\left\Vert \mathbb{P}_{\nu}[X_{t+s}\in\cdot\,]-\mui\right\Vert _{\textrm{TV}}\le\int \left\Vert \mathbb{P}_{\nu}[X_{t}(\Gamma_{\us})\in\cdot\,]-\mu_{\Gamma_{\us}}\right\Vert _{\textrm{TV}}d\mathbb{P}(\us)+2n^{-3d}\;,
\end{equation*}
where $\mu_{\Gamma_{\us}}$ represents the projection of $\mui$
on $\Gamma_{\us}$.
\end{lem}

\begin{proof}The proof in the above reference  relies only on the coupling of $X_t$ and $\mathcal{G}_t (X_0)$ for $t\in [0,\,\smax]$. For our model this has been established in Lemma \ref{lem32} based on the bound on disagreement percolation using the sparse initial conditions.
\end{proof}
Now we establish the sparsity of the update support $\Gamma_{\us}$.
\begin{defn}[Sparse set]
\label{def310}A  $\mathcal{S}\in\Omega_{n}$ is called
\textit{sparse} if for some $K\le n^{d}(\log n)^{-12d}$, the graph
induced by $\mathcal{S}$ can be decomposed into disjoint components
$A_{1},\,A_{2},\,\cdots,\,A_{K}$ such that
\begin{enumerate}
\item For all distinct $i,\,j\in\llbracket1,\,K\rrbracket$, there is no
open path in $S$ connecting $A_{i}$ and $A_{j}$.
\item Every $A_{i}$, $i\in\llbracket1,\,K\rrbracket$, has diameter at
most $\log^{5}n$. In particular, there is a box of size $2\log^{5}n$
containing $A_{i}$.
\item The distance between any distinct $A_{i}$ and $A_{j}$ is at least
$4\log^{4}n$.
\end{enumerate}
We write $\spar$ to denote the set of sparse configurations in $\Omega_{n}$.
\end{defn}

\begin{figure}[h]
\centering
\includegraphics[scale=0.17]{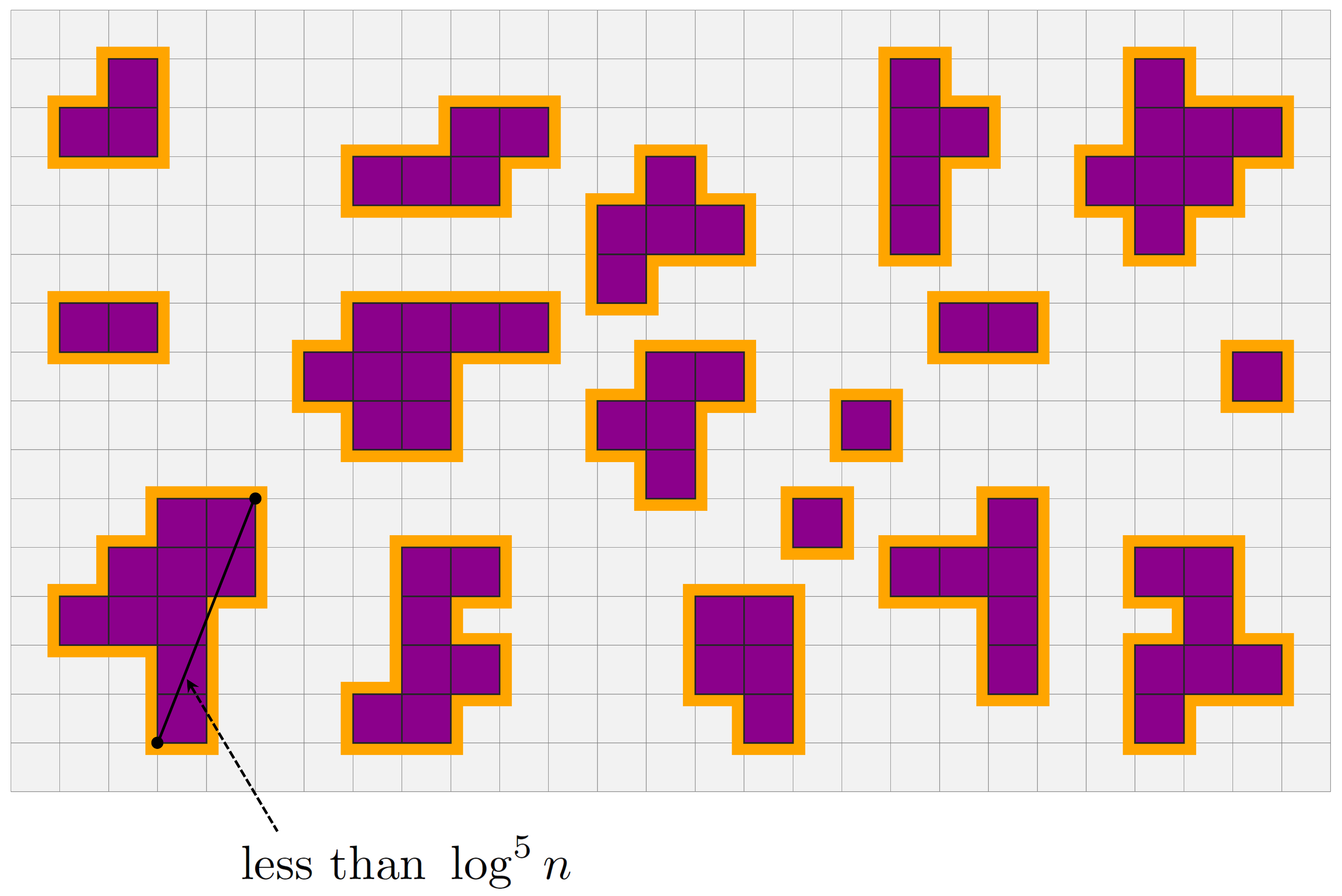}
\caption{{The figure illustrates the sparse update support. The purple regions (including the orange buffer around them) denote the sets $A_i$ (and $A_i^+$) as in Definitions \ref{def310} and \ref{mgood}}.}
\label{fig4}
\end{figure}

\begin{lem}\cite[Lemma 3.9]{LS1}
\label{lem66}There exists $C_{2}=C_{2}(p)>0$ such that, for all
$s\ge C_{2}\log\log n$,
\begin{equation}
\mathbb{P}[\Gamma_{\us}\in{\mathrm\spar}]\ge1-n^{-3d}\;.\label{e311}
\end{equation}
\end{lem}

\begin{proof}
 The only model-dependent part is the proof
of the following fact: For $t\ge C_{2}\log\log n$ with a large enough
$C_{2}$,
\begin{equation}
\sum_{e\in C_{v}^{+}}\mathbb{P}\left[X_{t}^{v,\textrm{full}}(e)\neq X_{t}^{v,\textrm{empty}}(e)\right]\le\log^{-10d}n\;,\label{e312}
\end{equation}
where $(X_{t}^{v,\textrm{full}})_{t\ge0}$ (resp. $(X_{t}^{v,\textrm{empty}})_{t\ge0})$
is the FK-dynamics on periodic lattice $C_{v}^{+}$ with full (resp.
empty) initial condition.
The proof of this fact in our setting follows from Corollary \ref{cor23}
which indicates that, for some     $C>0$,
\[
\sum_{e\in C_{v}^{+}}\mathbb{P}\left[X_{t}^{v,\textrm{full}}(e)
\neq X_{t}^{v,\textrm{empty}}(e)\right]\le|C_{v}^{+}|e^{-C\log\log n}\;.
\]
Hence, the bound \eqref{e312} follows if we take $C$ large enough. The
remaining part is identical to cited proofs and will not be repeated
here.
\end{proof}
By Lemmas \ref{lem64} and \ref{lem66}, we obtain the following result.
\begin{prop}
\label{p313}Suppose that $p$ is small enough and $\nu$ is a probability
distribution on $\Omega_{n}$ satisfying $\nu\preceq\per(\pini)$.
Then, for all $s\in[C_{0}\log\log n,\,\smax]$ where $C_{0}$ is the
constant appearing in Lemma \ref{lem66}, there exists a measure $\mathbb{Q}$
on $\spar$ such that,
\[
\left\Vert \mathbb{P}_{\nu}(X_{t+s}\in\cdot)-\mui\right\Vert _{\textrm{TV}}\le\int_{\spar}\left\Vert \mathbb{P}_{\nu}(X_{t}(\Gamma\in\cdot)-\mu_{\Gamma}\right\Vert _{\textrm{TV}}d\mathbb{Q}(\Gamma)+3n^{-3d}\;.
\]

\end{prop}

\subsection{\label{sec63}Proof of Theorem \ref{t61}} Before jumping into the proof we will need some technical preparation. The next few results use coupling arguments to compare the actual chain to a product chain.
We start by defining a notion of good sets, and then introduce a generalized
version of barrier dynamics.
\begin{defn}
\label{mgood}A collection of disjoint subsets $A_{1},\,A_{2},\,\cdots,\,A_{K}$
of $\Omega_{n}$ are $m$-good for some $m\in\left[\log^{4}n,\,(1/2)\log^{5}n\right]$
if, each $A_{i}$ is contained in a box of size $2\log^{5}n$, and
the expanded sets $A_{i}^{+}$, $i\in\llbracket1,\,K\rrbracket$,
are disjoint where
\[
A_{i}^{+}=\{e\in E_{n}:d(e,\,A_{i})\le m\}\;.
\]
As a consequence of Lemma \ref{lem66}, the sets $A_i$ in the update support are $\log^4 n-$good (see Figure \ref{fig4}).

Let us take a box of size $r=3\log^{5}n$ containing $A_{i}^{+}$
and denote this box by $\widehat{A}_{i}^{+}$ (this is possible since
$m\le(1/2)\log^{5}n$). Take $K$ copies $\mathcal{L}_{1},\,\mathcal{L}_{2},\,\cdots,\mathcal{L}_{K}$
of the periodic lattice $\mathbb{Z}_{r}^{d}$, and embed each $\widehat{A}_{i}^{+}$
to $\mathcal{L}_{i}$ by a identification map $\mathcal{P}_{i}:\widehat{A}_{i}^{+}\rightarrow\mathcal{L}_{i}$.

For $i\in\llbracket1,\,K\rrbracket$, denote by $(Y_{t}^{(i)})_{t\ge0},$
the FK-dynamics on $\mathcal{L}_{i},$ whose update sequence and initial
condition are inherited from that of $\widehat{A}_{i}^{+}=\mathcal{P}_{i}^{-1}(\mathcal{L}_{i})$
of $(X_{t}(\widehat{A}_{i}^{+}))_{t\ge0}$. Let $\pi^{(i)}$ be the
random-cluster measure on $\mathcal{L}_{i}$ so that $\pi^{(i)}$
is the invariant measure of $Y_{t}^{(i)}$. Define the product spaces:
\[
\mathcal{L}=\prod_{i=1}^{K}\mathcal{L}_{i}\;,\;\pi=\prod_{i=1}^{K}\pi^{(i)},\;\text{and\;\;}Y_{t}^{*}=\prod_{i=1}^{K}Y_{t}^{(i)}\;,
\]
and let
\[
\Gamma=\bigcup_{i=1}^{K}A_{i}\subset\Omega_{n}\;\;\text{and\;\;}\Gamma^{*}=\bigcup_{i=1}^{K}\mathcal{P}(A_{i})\subset\mathcal{L}\;.
\]
By slight abuse of notations, we identify $A_i$ and $\mathcal{P}(A_i)$, for
$i\in \llbracket 1,\,K \rrbracket$ or $\Gamma$ and $\Gamma^*$ and simply write $\mathcal{P}(A_i)=A_i$ and
$\Gamma^* = \Gamma$.  With this identification, we can regard  $X_t(A_i)$ and $Y_t^*(A_i)$
 or $X_t(\Gamma)$ and $Y_t^*(\Gamma)$ as processes defined on the same space.
\end{defn}

{We first recall from Remark \ref{rm52}, that  we can couple $(X_{t})$ and
$(Y_{t}^{*})$.
 In the lemmas below where we record various coupling statements, we assume that the
collection $A_{1},\,A_{2},\,\cdots,\,A_{K}$ of subsets of $\Omega_{n}$
is $m$-good for some $m\in\left[\log^{4}n,\,(1/2)\log^{5}n\right]$.
\begin{lem}
\label{lemcou}Suppose that $p$ is small enough and the law of the
initial condition $X_{0}$ follows the law $\nu$ such that $\nu\preceq\per(\pini)$.
Then, we have
\[
\mathbb{P}\left[ X_{t}(\Gamma)=Y_{t}^{*}(\Gamma)\text{ for all }t\in[0,\,\smax]\,\right]\ge1-\frac{1}{n^{2d}}\;.
\]
\end{lem}
\begin{proof}
Since $K\le n^{d}$, it suffices to show that, for all $i\in\llbracket1,\,K\rrbracket$,
\[
\mathbb{P}\left[X_{t}(A_{i})=Y_{t}^{*}(A_{i})\text{ for all }t\in[0,\,\smax]\,\right]\ge1-\frac{1}{n^{3d}}\;.
\]
This follows directly from Lemma \ref{lem0311}.
\end{proof}
Now we obtain  upper and lower bounds for the total-variation distance
of $(Y_{t}^{*})$ in the two lemmas below. Combined with the previous
coupling result, they yield bounds on the total-variation distance for
 $(X_{t})$.
\begin{lem}
\label{lempro}For all sufficiently small $p$, we have that
\[
\sup_{x_{0}\in\Omega_{n}}\left\Vert \mathbb{P}_{x_{0}}[Y_{t}^{*}(\Gamma)\in\cdot\,]-\pi_{\Gamma}\right\Vert _{\textrm{TV}}\le\frac{1}{2}\left[e^{K\mathbf{d}_{t}^{2}}-1\right]^{\frac{1}{2}},
\]
where $\pi_{\Gamma}$ represents the projection of $\pi$ onto
$\Gamma $.
\end{lem}

\begin{notation}
\label{nots}In the statement of lemma, $\mathbb{P}_{x_{0}}[Y_{t}^{*}(\Gamma)\in\cdot\,]$
means that the starting configuration of $Y_{t}^{*}$ is inherited
from $x_{0}\in\Omega_{n}$ by the collection map $\prod_{i=1}^{K}\mathcal{P}^{(i)}:\prod_{i=1}^{K}\widehat{A}_{i}^{+}\rightarrow\prod_{i=1}^{K}\mathcal{L}_{i}$.
We define $\mathbb{P}_{\nu}[Y_{t}^{*}(\Gamma)\in\cdot\,]$ for a
probability distribution $\nu$  on $\Omega_n$ in the same manner.
\end{notation}

\begin{proof}
By the $L^{1}$-$L^{2}$ inequality we have
\begin{equation}
\left\Vert \mathbb{P}_{x_{0}}[Y_{t}^{*}(\Gamma )\in\cdot\,]-\pi_{\Gamma }\right\Vert _{\textrm{TV}}\le\frac{1}{2}\left\Vert \mathbb{P}_{x_{0}}[Y_{t}^{*}(\Gamma )\in\cdot\,]-\pi_{\Gamma }\right\Vert _{L^{2}(\pi_{\Gamma} )}\;.\label{e611}
\end{equation}
Denote by $\pi_{A_{i}}^{(i)}$ the  projection of $\pi^{(i)}$
onto $A_{i}$. Then, since $\pi_{\Gamma }=\prod_{i=1}^{K}\pi_{A_{i}}^{(i)}$,
by the bound of $L^{2}$-norm for product space (cf. \cite[Section 3.2]{LS1}),
we obtain that
\begin{equation}
\left\Vert \mathbb{P}_{x_{0}}[Y_{t}^{*}(\Gamma )\in\cdot\,]-\pi_{\Gamma }\right\Vert _{L^{2}(\pi_{\Gamma} )}\le\left[\exp  \left\{ \sum_{i=1}^{K}\left\Vert \mathbb{P}_{x_{0}}[Y_{t}^{(i)}(A_{i})\in\cdot\,]-\pi_{A_{i}}^{(i)}\right\Vert _{L^{2}(\pi_{i}^{*})}^{2}  \right\}-1\right] ^{1/2} \label{e612}.
\end{equation}
By the definition of $\mathbf{d}_{t}$ (see \eqref{emt}) and by the fact that $A_{i}$
is a subset of box of size $2\log^{5}n$, we can deduce from the definition
of $\mathbf{d}_{t}$ that
\begin{equation}
\left\Vert \mathbb{P}_{x_{0}}[Y_{t}^{(i)}(A_{i})\in\cdot\,]-\pi_{A_{i}}^{(i)}\right\Vert _{L^{2}(\pi_{i}^{*})}\le\mathbf{d}_{t}\;.\label{e613}
\end{equation}
We now conclude using \eqref{e611}, \eqref{e612} and \eqref{e613}.
\end{proof}

Recall that $\mu_{\Gamma}$ represents the projection of $\mui$ to $\Gamma$. Using spatial mixing properties, we conclude now that $\mu_\Gamma$ is close to $\pi_\Gamma$.
This follows from Lemma \ref{contour} which implies that the effect of the boundary condition does not reach beyond the buffer region $A_i^+ \setminus A_i $ for each $i$ (see Figure \ref{fig4}). Using this we prove that the total-variation distance between $\mu_\Gamma$ and $\pi_\Gamma$ is small.
\begin{lem}\label{lemmu}
It holds that
\begin{equation*}
 \left\Vert \mu_{\Gamma}-\pi_{\Gamma }\right\Vert _{\textrm{TV}}\le\frac{1}{n^{2d}}\;.
\end{equation*}
\end{lem}
\begin{proof}
We apply Lemma \ref{contour} with $A=\Gamma$ and $B=E_n\setminus \bigcup_{i=1}^K A_i^+ $. Recall the
measure $\mu_{B^c}^{+}$  and the configuration $X$ from Lemma \ref{contour}. The latter implies that,
with probability more than $1-n^{-2d}$, there exists a closed surface in $X(A_i^+ \setminus A_i)$ enclosing $A_i$
for all $i\in \llbracket 1,\,K\rrbracket$. This implies the Lemma by the domain Markov property of random cluster measure.
For details about this argument, see \cite[Proof of Claim 4.2]{BS}.
\end{proof}

\begin{lem}
\label{lempro2}Suppose that $p$ is small enough, $t\in[0,\,\smax]$
and the collection $A_{1},\,A_{2},\,\cdots,\,A_{K}$ of subsets of
$\Omega_{n}$ is $m$-good for some $m\in\left[\log^{4}n,\,(1/2)\log^{5}n\right]$.
Then under the notations of Definition \ref{mgood}, we have
\[
d(t)+\frac{2}{n^{2d}}\ge\sup_{x_{0}\in\Omega_{n}}\left\Vert \mathbb{P}_{x_{0}}[Y_{t}^{*}(\Gamma )\in\cdot\,]-\pi_{\Gamma }\right\Vert _{\textrm{TV}}\;,
\]
where $d(t)$ the total-variation distance at time $t$ was defined in Section \ref{sec13}.
\end{lem}

\begin{proof}
Since projection does not
increase total-variation norm, we have
\begin{equation}\label{cw09}
  d(t)\ge\sup_{x_{0}\in\Omega_{n}}\left\Vert \mathbb{P}_{x_{0}}[ X_{t}(\Gamma)
\in\cdot\,]-\mu_{\Gamma}\right\Vert _{\textrm{TV}}\;.
\end{equation}
By Lemma \ref{lemcou}, we have
\begin{equation}\label{cw10}
   \sup_{x_{0}\in\Omega_{n}}\left\Vert \mathbb{P}_{x_{0}}[ X_{t}(\Gamma) \in\cdot\,]
   -\mathbb{P}_{x_{0}}[Y_{t}^{*}(\Gamma )\in\cdot\,]\,\right\Vert _{\textrm{TV}}\le\frac{1}{n^{2d}}\;.
\end{equation}
By combining \eqref{cw09}, \eqref{cw10}, and  Lemma \ref{lemmu}, the proof is completed.
\end{proof}

We are finally ready to finish the proof of Theorem \ref{t61}

\subsubsection{Proof of part (1): upper bound}
In view of Proposition \ref{p313}, it suffices to prove the following
proposition.
\begin{prop}
Suppose that $p$ is sufficiently small, $\Gamma\in\spar$, and $t\in[0,\,\smax]$.
Then, we have
\begin{equation}
\left\Vert \mathbb{P}_{\nu}[X_{t}(\Gamma)\in\cdot\,]-\mu_{\Gamma}\right\Vert _{\textrm{TV}}
\le\frac{1}{2}\left[\exp\left\{ \frac{n^{d}}{\log^{12d}n}\mathbf{d}_{t}^{2}\right\}
-1\right]^{1/2}+\frac{2}{n^{2d}}\;.\label{ep611}
\end{equation}
\end{prop}

\begin{proof}
Denote by $A_{1},\,A_{2},\,\dots,\,A_{K}$ the connected components
of $\Gamma$ in the sense of Definition \ref{def310}. Then, then $A_{1},\,A_{2},\,\dots,\,A_{K}$
are $m$-good with $m=\log^{4}n$. Now we recall the notations from
Definition \ref{mgood} and Lemma \ref{lempro}. We bound the total-variation
norm at the left-hand side of \eqref{ep611} by
\begin{equation}
\left\Vert \mathbb{P}_{\nu}[X_{t}(\Gamma)\in\cdot\,]-\mathbb{P}_{\nu}[Y_{t}^{*}(\Gamma )\in\cdot\,]\,\right\Vert _{\textrm{TV}}+\left\Vert \mathbb{P}_{\nu}[Y_{t}^{*}(\Gamma )\in\cdot\,]-\pi_{\Gamma }\right\Vert _{\textrm{TV}}+\left\Vert \pi_{\Gamma }-\mu_{\Gamma}\right\Vert _{\textrm{TV}}\;.\label{ep62}
\end{equation}
We recall Notation \ref{nots} for the notation $\mathbb{P}_{\nu}[Y_{t}^{*}(\Gamma)\in\cdot\,]$.
We now bound these three terms separately to complete the proof. For
the first term, by Lemma \ref{lemcou} we have
\begin{equation}
\left\Vert \mathbb{P}_{\nu}[X_{t}(\Gamma)\in\cdot\,]-\mathbb{P}_{\nu}[Y_{t}^{*}(\Gamma )\in\cdot\,]\,\right\Vert _{\textrm{TV}}\le\frac{1}{n^{2d}}\;.\label{ec60}
\end{equation}
By the Lemma \ref{lempro}, and the fact $K\le n^{d}/\log^{12d}n$,
the second term is bounded by
\begin{equation}
\left\Vert \mathbb{P}_{\nu}[Y_{t}^{*}(\Gamma)\in\cdot\,]-\pi_{\Gamma} \right\Vert _{\textrm{TV}}\le\frac{1}{2}\left[\exp\left\{ \frac{n^{d}}{\log^{12d}n}\mathbf{d}_{t}^{2}\right\} -1\right]^{1/2}\;.\label{ec61}
\end{equation}
Finally, the last term at \eqref{ep62} is at most $1/n^{2d}$ by Lemma \ref{lemmu}. Combining this with   \eqref{ep62}, \eqref{ec60}, and \eqref{ec61},
we can finish the proof.
\end{proof}

\subsubsection{Proof of Part (2): lower bound}
Given the above ingredients the proof of the lower bound  is almost verbatim from \cite[Section 3.3]{LS1} but nonetheless we include the proof in the appendix for completeness.

In the following section we finish the proof of Theorem \ref{tmain}.

\section{Proof of main result}\label{proofmain}
We keep the notation $r=3\log^{5}n$. The following lemma provides
a sharp bound on $\mathbf{d}_{t}$.
\begin{lem}\cite[Lemma 4.1]{LS1}
\label{lem71}For all small enough $p$, there exists a constant $C_{3}=C_{3}(p)>0$
such that
\begin{equation}
e^{-\lambda(r)(t+C_{3}\log\log n)}-n^{-2d}\le\mathbf{d}_{t}\le e^{-\lambda(r)(t-C_{3}\log\log n)}\label{e71}
\end{equation}
for all $t\in[C_{3}\log\log n,\,\smax]$.
\end{lem}

\begin{proof}
Since
\[
\mathbf{d}_{t}\le\max_{x_{0}^{\dagger}\in\Omega_{r}}\left\Vert \mathbb{P}_{x_{0}^{\dagger}} [X_{t}^{\dagger}\in\cdot\, ]-\pi^{\dagger}\right\Vert _{L^{2}(\pi^{\dagger})}\;,
\]
the upper bound part of \eqref{e71} is immediate from Theorem \ref{t41}.
We note from this bound that
\begin{equation}
r^{d/2}\mathbf{d}_{t}=o(1)\text{ for }t=C\log\log n\label{e72}
\end{equation}
with sufficiently large $C$. Here we implicitly used Corollary \ref{cor23}.
For the lower bound part, we first recall the bound
\[
e^{-\lambda(r)t}\le2\max_{x_{0}^{\dagger}\in\Omega_{r}}\left\Vert \mathbb{P}_{x_{0}^{\dagger}} [X_{t}^{\dagger}\in\cdot\, ]-\pi^{\dagger}\right\Vert _{\textrm{TV}}\;\;\text{for all }t\ge0\;,
\]
which is the continuous time version of Proposition \ref{propgap} (see \cite[Lemma 20.11]{LPW}).
By Lemma \ref{pro35} (in particular, \eqref{ep35}), we have
that
\[
e^{-\lambda(r)(t+\tini)}\le2\max_{\nu:\nu\in\perr}\left\Vert \mathbb{P}_{x_{0}^{\dagger}} [X_{t}^{\dagger}\in\cdot\, ]-\pi^{\dagger}\right\Vert _{\textrm{TV}}\;\;\text{for all }t\ge0\;.
\]
We now take $s=C_{1}\log\log n$ and  $t\in[C\log\log n,\,\smax]$, where $C_{1}$ and  $C$ are the constants appeared
in Theorem \ref{t61} and   in \eqref{e72}, respectively. Then, by the previous inequality
and part (1) of Theorem \ref{t61} for the lattice $\mathbb{Z}_{r}^{d}$
(cf. Remark \ref{rm62} with $m=r=3\log^5 n$), we have that
\begin{align*}
e^{-\lambda(r)(t+\tini+C_{1}\log\log n)} & \le2\max_{\nu:\nu\in\perr}\left\Vert \mathbb{P}_{x_{0}^{\dagger}} [X_{t+s}^{\dagger}\in\cdot\, ]-\pi^{\dagger}\right\Vert _{\textrm{TV}}\\
 & \le\left[\exp \{ r^{d}\mathbf{d}_{t}^{2}\} -1\right]^{\frac{1}{2}}+8n^{-2d}\le2r^{d/2}\mathbf{d}_{t}+8n^{-2d}\;,
\end{align*}
where the last inequality follows from \eqref{e72} and the elementary
inequality $e^{x}-1\le4x$ for $x\in[0,\,1]$. We can deduce the lower
bound from this computation.
\end{proof}
Given the above, the proof of Theorem \ref{tmain} involves two steps:
\begin{itemize}
\item Prove a version (Proposition \ref{p72}) with $\lambda_{\infty}$ replaced by $\lambda(r)$ where $r$ was chosen above.
\item Show that $\lambda(r)$ converges to $\lambda_{\infty}$ and have bounds on the convergence  rate (Proposition \ref{p73}).
\end{itemize}

Define
\[
t(n)=\frac{d}{2\lambda(r)}\log n\;\;\;\text{and\;\;\;}w(n)=\log\log n\;.
\]
\begin{prop}
\label{p72}For all small enough $p$, there exist two constants $c_{1}=c_{1}(p),\,c_{2}=c_{2}(p)$
such that
\begin{align}
\lim_{n\rightarrow\infty}\max_{x_{0}\in\Omega_{n}}\left\Vert \mathbb{P}_{x_{0}}\left[X_{t(n)-c_{1}w(n)}\in\cdot\,\right]-\mui\right\Vert _{\textrm{TV}} & =1\;,\label{e721}\\
\lim_{n\rightarrow\infty}\max_{x_{0}\in\Omega_{n}}\left\Vert \mathbb{P}_{x_{0}}\left[X_{t(n)+c_{2}w(n)}\in\cdot\,\right]-\mui\right\Vert _{\textrm{TV}} & =0\;.\label{e722}
\end{align}
\end{prop}

\begin{proof}
By Lemma \ref{pro35}, it suffices to consider the initial condition
$\nu$ satisfying {$\nu\preceq\per(\pini)$}.
We recall the constant $C_{3}$ from the statement of Lemma \ref{lem71}.
First, by the lower bound in Lemma \ref{lem71} and by part (3) of
Corollary \ref{cor23},
\begin{align*}
\bigg(\frac{n}{\log^{10}n}\bigg)^{d}\mathbf{d}_{t(n)-c_{1}w(n)}^{2} & \ge\bigg(\frac{n}{\log^{10}n}\bigg)^{d}e^{-2\lambda(r)(t(n)-c_{1}w(n)+C_{3}\log\log n)}-n^{-d}\\
 & \ge(\log n)^{2\lambda(c_{1}-C_{3})-10d}-n^{-d}\;.
\end{align*}
Therefore, for $c_{1}>C_{3}+\frac{11d}{2\lambda}$, we have
\[
\lim_{n\rightarrow\infty}\bigg(\frac{n}{\log^{10}n}\bigg)^{d}\mathbf{d}_{t(n)-c_{1}w(n)}^{2}=+\infty\;,
\]
and thus by part (2) of Theorem \ref{t61}  we obtain \eqref{e721}.
Now we turn to \eqref{e722}. For $c\in(0,\,C_{3})$, by the upper
bound of Lemma \ref{lem71},
\[
\frac{n^{d}}{\log^{12d}n}\mathbf{d}_{t(n)+cw(n)}^{2}\le\frac{n^{d}}{\log^{12d}n}e^{-2\lambda(r)(t(n)+cw(n)-C_{2}\log\log n)}\le(\log n)^{\lambda(C_{2}-c)-12d}\;.
\]
By taking $c$ close enough to $C_{2}$ we obtain
\begin{equation}
\frac{n^{d}}{\log^{12d}n}\mathbf{d}_{t(n)+cw(n)}^{2}\le\frac{1}{\log^{11d}n}\;.\label{e723}
\end{equation}
Let $c_{2}=C_{1}+c$ where $C_{1}$ is the constant appeared in Theorem
\ref{t61}. Then, by part (1) of Theorem \ref{t61} (note that this is where the sparsity assumption on $\nu$ is used) and \eqref{e723},
\begin{align*}
\max_{\nu:\nu\preceq\per(\tini)}\left\Vert \mathbb{P}_{\nu}\left[X_{t(n)+c_{2}w(n)}\in\cdot\,\right]-\mui\right\Vert _{\textrm{TV}} & \le\frac{1}{2}\left[\exp\left\{ \frac{n^{d}}{\log^{12d}n}\mathbf{d}_{t(n)+cw(n)}^{2}\right\} -1\right]^{\frac{1}{2}}+\frac{4}{n^{2d}}\\
 & \le\frac{1}{2}\left[\exp  \frac{1}{\log^{11d}n} -1\right]^{\frac{1}{2}}+\frac{4}{n^{2d}}\;.
\end{align*}
This completes the proof of \eqref{e722}.
\end{proof}
Notice that since ${ \lambda(r)=\Theta(1)}$ (Corollary \ref{cor23}), we have $w(n)\ll t(n)$, and therefore
the previous proposition already demonstrates the cutoff
phenomenon provided that $p$ is small enough.
The next result shows  that the sequence $(\lambda(r))_{r\ge1}$ is a convergent
sequence.
\begin{prop}\cite[Lemma 4.3]{LS1}
\label{p73}There exists ${\lambda}_{\infty}={\lambda}_{\infty}(p)>0$ such
that
\[
|\lambda(r)-{\lambda}_{\infty}|\le r^{-1/4+o(1)}\;.
\]
\end{prop}

\begin{proof} We only provide the modified choice of parameters needed for our purpose.
A careful reading of the proof shows that  entire arguments presented
above are still in force if we replace $r=3\log^{5}n$ with $r=\log^{4+\delta}$
for any $\delta$. Of course the constants that we obtained above
must be modified to depend on $\delta$, and the time $\smax$ should
be defined as $\log^{1+\delta}n$ (cf. Remark \ref{rm53}). Taking $r_{1}=\log^{4+\delta}$
and $r_{2}\in[r_{1},\,r_{1}^{2}]$ and applying Proposition \ref{p72}
with $r=r_{1}$ and $r=r_{2}$, respectively, yields
\[
\frac{d}{2\lambda(r_{1})}\log n-Cw(n)\le\frac{d}{2\lambda(r_{2})}\log n+Cw(n)
\]
for some constant $C=C(p,\,\delta)$. Since $\lambda(\cdot)$ is bounded
below, we obtain
\[
\lambda(r_{1})-\lambda(r_{2})\le C\frac{\log\log n}{\log n}\le r_{1}^{-1/4+\delta}
\]
for all sufficiently large $n$. The rest of the arguments are exactly the same as \cite[Lemma 4.3]{LS1} and are omitted.

\end{proof}

\subsection{Proof of Theorem \ref{tmain}}
As mentioned before  we can combine Propositions \ref{p72} and \ref{p73} to deduce Theorem \ref{tmain}.
Define
\[
t^{*}(n)=\frac{d}{2{\lambda_{\infty}}}\log n\;.
\]
Thus we need to show that for all small enough $p$, there exist two constants $c_{1}=c_{1}(p),\,c_{2}=c_{2}(p)$
such that
\begin{align}
\lim_{n\rightarrow\infty}\max_{x_{0}\in\Omega_{n}}\left\Vert \mathbb{P}_{x_{0}}\left[X_{t^{*}(n)-c_{1}w(n)}\in\cdot\,\right]-\mui\right\Vert _{\textrm{TV}} & =1\;,\label{e741}\\
\lim_{n\rightarrow\infty}\max_{x_{0}\in\Omega_{n}}\left\Vert \mathbb{P}_{x_{0}}\left[X_{t^{*}(n)+c_{2}w(n)}\in\cdot\,\right]-\mui\right\Vert _{\textrm{TV}} & =0\;.\label{e742}
\end{align}
The proof is now immediate from
\[
\left|t^{*}(n)-t(n)\right|\le C|\lambda(r)-\lambda_{\infty}|\log n\le C\log^{-1/4}n\;.
\]
\qed
\subsection{Comparison to infinite volume dynamics} \label{infinitevolume}
It is quite natural to predict that $\lambda_{\infty}$ is in fact the spectral gap of the infinite volume FK-dynamics with the same parameters $p$ and $q.$ Defining the latter is not trivial but this has been carried out in \cite[Chapter 8]{gri}.
For the Ising model a similar result was shown in \cite{LS1} using the monotonicity of the underlying dynamics as well as Log-Sobolev inequalities. Even though the lack of monotonicity of the Potts model prevented the authors in \cite{LS2} to prove a similar conclusion, this was settled in  \cite[Section 6.2]{NS} using the Information Percolation machinery which also implies the same for SW dynamics. Furthermore in \cite{NS}, the authors remark that the argument relies on bounds on disagreement propagation and an infinite version of the exponential $L^2$-mixing rate  and hence holds in more generality for spin systems.

Thus in our context of the FK-dynamics to prove a similar result, given the disagreement propagation bounds,  stated in Section \ref{sec5},
 the only remaining step is to establish an analog of Theorem \ref{t41} for the infinite system by proving an analog of Proposition \ref{p46} in the same setting. The argument in \cite{NS} proceeds by defining  Information Percolation clusters for the infinite process.  We believe that this can be carried out in our setting as well, by suitable extensions of the arguments for finite systems presented in Section \ref{sec4}. However, we do not pursue verifying the precise details in the paper.

\section{Appendix}\label{appendix1} We provide the proofs that were omitted from the main article.
\begin{proof}[Proof of Lemma \ref{lem0311}]
We recall notations from Section \ref{sec31} and in particular $\smax$ from \eqref{esm} and write $\smax=L\Delta$
so that
\[
0=\tau_{0}<\tau_{1}<\cdots<\tau_{L}=\log^{2}n\;.
\]
We regard $(Z_{t}^{-})_{t\ge0}$ as a Markov chain on {$\Omega_{n}$}
such that the configuration outside of $A^{+}$ is empty.
Also consider
$(X_{t})_{t\ge0},$ the FK-dynamics on $\Omega_{n}$ starting from an initial condition which agrees with $(Z_{t}^{-})$ on $A^+$.
One can observe that under the monotone coupling,
\begin{equation}
Z_{t}^{-}\le X_{t}\le Z_{t}^{+}\text{ for all }t\ge0\;.\label{eb51}
\end{equation}
We recall the enlarged percolation $\eful_{i}$ from Table \ref{chart}, and denote by $\mathcal{E}_{i}$ the event
that there is no open path of length $\beta\log n$ in $\eful_{i}$
for some $\beta=\beta(p)>0$. Then, by Proposition \ref{propdec}
and the union bound,
\[
\mathbb{P}[\mathcal{E}_{i}]\le{|E_{n}| \choose 2}\exp\left\{ -\gamma\beta\log n\right\} <\frac{1}{n^{4d}}
\]
provided that $\beta$ is large enough.
Define $\mathcal{E}=\bigcup_{i=1}^{L}\mathcal{E}_{i}$.
Since $L=\smax/\Delta=\Omega(\log^{2}n)$, the union bound implies
that
\begin{equation}
\mathbb{P}[\mathcal{E}]\le L\frac{1}{n^{4d}}<\frac{1}{n^{3d}}\;.\label{conn2}
\end{equation}

We now claim that $\mathcal{E}^{c}$ implies that $Z_{t}^{+}(A)=Z_{t}^{-}(A)$
for all $t\in[0,\,\smax]$. Thus the claim along with
\eqref{conn2}, finishes the proof of the lemma. To prove the claim, define
$A_{i}$, $i\in\llbracket0,\,L\rrbracket$, inductively as $A_{L}=A$
and
\[
A_{i-1}=\{e\in E_{n}:d(e,\,A_{i})\le\smax\Delta\}\;,
\]
so that
\[
A^{+}=A_{0}\supset A_{1}\supset\cdots\supset A_{L}=A\;.
\]
For all $i\in\llbracket0,\,L-1\rrbracket$, we shall prove that $Z_{\tau_{i}}^{+}(A_{i})=Z_{\tau_{i}}^{-}(A_{i})$
implies $Z_{t}^{+}(A_{i+1})=Z_{t}^{-}(A_{i+1})$ for all $t\in(\tau_{i},\,\tau_{i+1}]$.
since then the proof of the claim is completed by the induction. Now it suffices to observe that there exists a closed surface of $\eful_i$ in $A_i \setminus A_{i+1}$ under $\mathcal{E}_i$ since the set
$
\bigcup_{e\in \partial A_i} \conn(e;\eful_i)
$
is disjoint to $A_{i+1}$ as there is no connected path of length $\Omega(\log^2 n)$ in $\eful_i$ (call this surface as $V_i$).
The proof now follows by noticing that the FK-dynamics for both $Z^+_t$ and $Z^-_t$ agree on the component of $E_n\setminus V_i$ (say $\tilde A_i$) containing $A_{i+1}$ (and hence on $A_{i+1}$) throughout $[\tau_i,\tau_{i+1}],$ since the starting configurations for both the chains agree on $\tilde A_i$ by induction and the dynamics has zero boundary condition throughout $[\tau_i,\tau_{i+1}]$.
\end{proof}

\begin{proof}[Proof of Theorem \ref{t61}, Part (2): lower bound]
Recall $r=3\log^{5}n$, and let us divide $\mathbb{Z}_{n}^{d}$ by
$K=\left\lfloor n/r\right\rfloor ^{d}$ square boxes $A_{1}^{+},\,A_{2}^{+},\,\cdots,\,A_{K}^{+}$
of size $r$ as we did in Section \ref{sec61}. Then, let $A_{i}$
be the box of size $2r/3$ which is concentric with $A_{i}^{+}$.
Then, the collection $A_{1},\,A_{2},\,\cdots,\,A_{K}$ is $m$-good
with $m=(1/2)\log^{5}n$. We recall the notations from Definition
\ref{mgood}.
By definition \eqref{emt} of $\mathbf{d}_{t}$, we can find $x_{0}^{*}=x_{0}^{*}(t)\in\Omega_{r}$
satisfying
$
\mathbf{d}_{t}= \Vert \mathbb{P}_{x_{0}^{*}} [X_{t}^{\dagger}(\Lambda)\in\cdot\, ]-
\pi_{\Lambda}^{\dagger} \Vert _{L^{2}(\pi_{\Lambda}^{\dagger})}.
$
Let $U_{i}$ be a configurations
on $B_{i}$ distributed according to $\pi_{A_{i}}^{(i)}$, where $\{U_{i},\,1\le i\le K\}$
is a collection of independent random variables. Define a sequence
of i.i.d. random variable $u_{i}$ as
\[
u_{i}=\frac{\mathbb{P} [Y_{t}^{(i)}(A_{i})=U_{i}\,|\,Y_{0}^{(i)}=x_{0}^{*} ]}{\pi_{B_{i}}^{(i)}(U_{i})}\;\;;\;i\in\llbracket1,\,K\rrbracket\;.
\]
The condition $Y_{0}^{(i)}=x_{0}^{*}$ means $X_{0}(\widehat{A}_{i}^{+})=x_{0}^{*}$.
By the definition of $u_{i}$ and $x_{0}^{*}$, one can readily check
that
\begin{equation}
\mathbb{E}u_{i}=1\text{ \;and\;\;}\textrm{Var}\,u_{i}=\mathbf{d}_{t}^{2}\label{euu1}\;.
\end{equation}
By the $L^{\infty}$-$L^{2}$ reduction for reversible Markov chains,
we have
\begin{align*}
\left\Vert u_{i}-1\right\Vert _{\infty} & =\left\Vert \mathbb{P}\left[Y_{t}^{(i)}(A_{i})\in\cdot\,|\,Y_{0}^{(i)}=x_{0}^{*}\right]-\pi_{A_{i}}^{(i)}\right\Vert _{L^{\infty}(\pi_{A_{i}}^{(i)})}\\
 & \le\left\Vert \mathbb{P}\left[Y_{t/2}^{(i)}(A_{i})\in\cdot\,|\,Y_{0}^{(i)}=x_{0}^{*}\right]-\pi_{A_{i}}^{(i)}\right\Vert _{L^{2}(\pi_{A_{i}}^{(i)})}\le\mathbf{d}_{t/2}\;.
\end{align*}
Hence, by  Theorem \ref{t41} we  obtain
\begin{equation}\label{euu2}
\|u_{i}-1\|_{\infty}\le e^{-c\log\log n}
\end{equation}
for  some $c >0.$
Then, by \eqref{euu1} and \eqref{euu2}, we have
\begin{equation}
\mathbb{E}|u_{i}-1|^{3}\le e^{-c\log\log n}\,\mathbf{d}_{t}^{2}=o(1)\,\mathbf{d}_{t}^{2}\;.\label{euu3}
\end{equation}
Given the above inputs, the rest of the proof follows by arguments identical to \cite[Section 3.3]{LS1} and is omitted.
\end{proof}

\bibliography{rs}
\bibliographystyle{plain}

\end{document}